\documentclass[a4paper,11pt,twoside]{article}
\usepackage[a4paper,hscale=0.7,vscale=0.75,centering]{geometry}
\usepackage{fullpage}

\RequirePackage[colorlinks,citecolor=blue,urlcolor=blue]{hyperref}
\usepackage[utf8]{inputenc}
\usepackage[T1]{fontenc}
\usepackage[english]{babel}
\usepackage{charter}
\usepackage{enumerate}
\usepackage{indentfirst}
\usepackage{xcolor}
\usepackage{authblk}
\usepackage{mathtools}

\usepackage{amsmath}
\usepackage{amsthm}	
\usepackage{amsfonts}
\usepackage{mathrsfs}	
\usepackage{amssymb}
\usepackage{dsfont}
\usepackage{bbm}

\theoremstyle{plain}
\newtheorem{theorem}{Theorem}[section]
\newtheorem{lemma}[theorem]{Lemma}

\newtheorem{corollary}[theorem]{Corollary}
\theoremstyle{definition}

\newtheorem{example}[theorem]{Example}
\newtheorem{assumption}[theorem]{Assumptions}
\newtheorem{assumptionsingle}[theorem]{Assumption}

\theoremstyle{definition}
\newtheorem{remark}[theorem]{Remark}

\numberwithin{equation}{section}
\numberwithin{figure}{section}
 \usepackage[nodayofweek]{datetime}

\usepackage[nodayofweek]{datetime}

\newcommand{\bE}{\mathbb{E}}

\newcommand{\bN}{\mathbb{N}}
\newcommand{\bP}{\mathbb{P}}

\newcommand{\bR}{\mathbb{R}}
\newcommand{\bV}{\mathbb{V}}

\def \E {\mathbb{E}}
\newcommand{\cA}{\mathcal{A}}

\newcommand{\cL}{\mathcal{L}}

\newcommand{\cP}{\mathcal{P}}

\title{Non-asymptotic bounds for sampling algorithms without log-concavity}
\author[1,2]{Mateusz B.\ Majka}
\author[1,2]{Aleksandar Mijatovi\'{c}}
\author[1,3]{\L ukasz Szpruch}

\affil[1]{The Alan Turing Institute, London}
\affil[2]{Department of Statistics, University of Warwick}
\affil[3]{School of Mathematics, University of Edinburgh}

\date{ }

\begin{document}
\selectlanguage{english}
\maketitle

\begin{abstract}
Discrete time analogues of ergodic stochastic differential equations (SDEs) are one of the most popular and flexible tools for sampling high-dimensional probability measures. Non-asymptotic analysis in the $L^2$ Wasserstein distance of sampling algorithms based on Euler discretisations of SDEs has been recently developed by several authors for log-concave probability distributions.
In this work we replace the log-concavity assumption with a log-concavity at infinity condition. We provide novel $L^2$ convergence rates for Euler schemes, expressed explicitly in terms of problem parameters. From there we derive non-asymptotic bounds on the distance between the laws induced by Euler schemes and the invariant laws of SDEs, both for schemes with standard and with randomised (inaccurate) drifts. We also obtain bounds for the hierarchy of discretisation, which enables us to deploy a multi-level Monte Carlo estimator. Our proof relies on a novel construction of a coupling for the Markov chains that can be used to control both the $L^1$ and $L^2$ Wasserstein distances simultaneously. Finally, we provide a weak convergence analysis that covers both the standard and the randomised (inaccurate) drift case. In particular, we reveal that the variance of the randomised drift does not influence the rate of weak convergence of the Euler scheme to the SDE.
\end{abstract}

\vspace{-5mm}

\section{Introduction}

Our primary aim is to study non-asymptotic properties of Markov chains that typically arise as approximations of ergodic solutions of stochastic differential equations on $\mathbb{R}^d$. The simplest example is a process $(X_k)_{k=0}^{\infty}$ defined as
\begin{equation} \label{eq euler}
\begin{cases}
X_{k+1} = X_k + b(X_k)h + \sqrt{h} \xi_{k+1} \,, \, \quad  k \geq 0 \,, \\
X_0 \sim \mu_0\,,
\end{cases}
\end{equation}
where $h > 0$ is the discretisation parameter and $(\xi_k)_{k=1}^{\infty}$ are i.i.d. random variables with the standard normal distribution $N(0,I)$. Here $\mu_0\in \cP_2(\bR^d)$, the space of square integrable probability measures on $\mathbb{R}^d$, and $b:\bR^d \rightarrow \bR^d$ is a drift function. We use the notation $\cL(X_k):=\operatorname{Law}(X_k)$ and our main goal is to quantify convergence of $\cL(X_k)$ using the $L^p$-Wasserstein distance with $p \in \{ 1 , 2 \}$, defined for probability measures $\mu$, $\nu \in \cP_p(\bR^d)$ as 
\begin{equation}
\label{eq p-wasserstein}
W_p(\mu,\nu) := \left( \inf_{\pi \in \Pi(\mu,\nu)}  \int_{\bR^d\times \bR^d} |x-y|^p \, \pi(dx \, dy) \right)^{1/p}	\,,
\end{equation}
where $\Pi(\mu,\nu)$ denotes the family of all couplings between $\mu$ and $\nu$, i.e., all measures on $\mathscr{B}(\bR^d\times \bR^d)$ 
such that $\pi(B \times \bR^d) = \mu(B)$ and $\pi(\bR^d \times B) = \nu(B)$ 
for every $B \in \mathscr B(\bR^d)$. Here $| \cdot | = \sqrt{\langle \cdot \, , \cdot \rangle}$ is the Euclidean distance on $\mathbb{R}^d$. We will also work with a special class of $L^1$-Wasserstein (pseudo) distances denoted by $W_f$ in which the Euclidean distance $|x-y|$ is replaced by $f(|x-y|)$ for some increasing function $f : [0,\infty) \to [0, \infty)$. Namely, we put $W_f(\mu, \nu) := \inf_{\pi \in \Pi(\mu,\nu)}  \int_{\bR^d\times \bR^d} f(|x-y|) \, \pi(dx \, dy)$. We remark that Wasserstein distances are typically preferred metrics when quantifying the quality of sampling methods, see \cite{arjovsky2017wasserstein,dalalyan2017user, gorham2016measuring}. 

Convergence in Wasserstein distances is typically investigated under the contractivity condition on the drift, i.e., under the assumption that there exists a constant $K>0$ such that
\begin{equation} \label{eq mono}
 \langle x-y, b(x) - b(y) \rangle \leq - K |x-y|^2 \qquad \text{ for all } x, y \in \mathbb{R}^d \,.
\end{equation}
If $b(x) = - \nabla U(x)$ for a function $U \in C^2(\mathbb{R}^d)$, this condition corresponds to strong convexity of $U$, whereas a probability measure $\mu$ such that $\mu(dx) \propto \exp (-U(x))dx$ is then called log-concave. Convergence analysis for several sampling algorithms under condition \eqref{eq mono} and the Lipschitz continuity of the drift has been recently performed in the $L^2$-Wasserstein distance in papers such as \cite{dalalyan2017user, Baker2017, durmus2016high, ChatterjiVarianceReduction2018}. 

In this work instead of \eqref{eq mono} we work with the following assumptions. 
\begin{assumption}[Contractivity at infinity] \label{as diss}
Function $b : \mathbb{R}^d \to \mathbb{R}^d$ satisfies the following conditions:
\begin{enumerate}[i)]
\item Lipschitz condition: there is a constant $L > 0$ such that
\begin{equation}\label{driftLipschitz}
|b(x) - b(y)| \leq L|x - y| \qquad \text{ for all } x, y \in \mathbb{R}^d \,.
\end{equation}
\item Contractivity at infinity condition: there exist constants $K$, $\mathcal{R} > 0$ such that 
\begin{equation}\label{driftDissipativityAtInf}
 \langle  x - y , b(x) - b(y) \rangle  \leq -K|x - y|^2 \qquad \text{ for all } x , y \in \mathbb{R}^d \text{ with } |x-y| > \mathcal{R} \,.
\end{equation}
\end{enumerate}
\end{assumption}

This enables us to cover a much wider class of SDEs, including e.g. equations with drifts given by double-well potentials (see the example in Section \ref{sec motivation}). We will show in Section \ref{sec motivation} that tools typically used in the global contractivity setting to study convergence in Wasserstein distances, such as the synchronous coupling and Talagrand's inequality, do not necessarily work under Assumptions \ref{as diss}. This forces us to look for an alternative approach.

Our method is based on the idea of controlling the standard Wasserstein distances $W_1$ and $W_2$ by specially constructed (pseudo) distances $W_f$ based on functions $f: [0, \infty) \to [0, \infty)$ that are concave on a compact interval and strictly convex at infinity. In order to briefly introduce our approach, let us focus on a single step $X^x_h$ of an Euler scheme started at $x \in \mathbb{R}^d$, i.e., $X^x_h = x + b(x)h + \sqrt{h} \xi$, where $\xi\sim N(0,I)$. If we now choose $y \in \mathbb{R}^d$, $y \neq x$ and consider $Y^y_h = y + b(y)h + \sqrt{h} \eta$ with an arbitrarily chosen $\eta\sim N(0,I)$, then $\cL(Y^y_h) = \cL(X^y_h)$ and straight from the definition of the Wasserstein distance $W_f$ we see that
$W_f(\cL(X^x_h),\cL(X^y_h)) \leq \mathbb{E} f(|X^x_h - Y^y_h|)$,
since the joint law of $(X^x_h, Y^y_h)$ is a coupling of $\cL(X^x_h)$ and $\cL(X^y_h)$. Hence in order to obtain sharp upper bounds on $W_f(\cL(X^x_h),\cL(X^y_h))$, one needs to be able to find pairs $(X^x_h, Y^y_h)$ that make $\mathbb{E} f(|X^x_h - Y^y_h|)$ as small as possible by choosing the joint distribution of $(\xi, \eta)$ in an appropriate way. However, in the present paper we are interested not only in quantifying distances between laws of Euler schemes started at different points, but also in distances between laws of Euler schemes and their perturbed versions. Namely, let $\widetilde{Y}$ be an arbitrarily chosen random variable. Under Assumptions \ref{as diss} we are able to prove the following inequality (see Theorem \ref{theoremPerturbation} and the comments after its proof) for all sufficiently small $h > 0$ with constants $c$, $C > 0$.
\begin{equation}\label{perturbationInequality}
\mathbb{E}f(|X_h^x - \widetilde{Y}|) \leq (1-c h)f(|x - y|) + C(1 + |x-y|) (\mathbb{E}|\widetilde{Y} - Y_h^y| + \mathbb{E}|\widetilde{Y} - Y_h^y|^2) \,.
\end{equation}
The idea behind this result is that if $\widetilde{Y}$ is chosen as a perturbation of $Y_h^y$ such that we can control the distance between $\widetilde{Y}$ and $Y_h^y$ in the $L^1$ and $L^2$ norms, then we can also control the distance between $\widetilde{Y}$ and $X_h^x$, via an auxiliary pseudo distance function $f$ that dominates all $L^p$ distances for $p \in [1,2]$.
 The exact form of the function $f$ and all the constants will be given in Theorem \ref{theoremPerturbation}. The important fact is that the function $f$ is chosen in such a way that there exist constants $a$, $A$ such that $r<a f(r)$ and $r^2<A f(r)$ for $r\in[0,\infty)$ (see Lemma \ref{corollarySquareComparison}) and hence (\ref{perturbationInequality}) yields bounds for $W_1$ and $W_2$ distances.

Our inequality (\ref{perturbationInequality}) is related to other perturbation results for Wasserstein distances, see the discussion in Section 2.3 of \cite{EberleMajka2018} and papers such as \cite{RudolfSchweizer, PillaiSmith, huggins2016quantifying}. As we explain in the sequel, examples of processes that we consider in this paper as perturbed versions of Euler schemes include an SDE (Section \ref{sectionULA}), an Euler scheme with a randomised drift (Section \ref{sectionStochasticGradient}) and an Euler scheme with a different discretisation level (Section \ref{sectionMLMC}). 

The Markov chain given by \eqref{eq euler} arises as a discretisation of a diffusion process given by 
\begin{equation}
\label{eq sde}	 
dY_t =   b(Y_t )\,dt  
+  dW_t\,,
\end{equation}
where $(W_t)_{t \geq 0}$ is a standard Brownian motion.  Assumptions \ref{as diss} guarantee that the solution does not blow up on $[0,\infty)$, see Chapter 3 in \cite{khasminskii2011stochastic}. It can be shown that the corresponding semigroup is Feller and consequently the Krylov-Bogolyubov theorem yields the existence of a unique invariant measure $\pi$. We remark that the asymptotic results on the discretisation of \eqref{eq sde} in the context of approximating invariant measures are rather well understood. We refer to  \cite{talay1990second,lamberton2002recursive,panloup2008recursive,pages2012ergodic} for a thorough overview on that topic. 

Working with the global contractivity \eqref{eq mono} assumption and the Euler discretisation \eqref{eq euler} with either constant or variable time-step, several authors in a series of papers \cite{dalalyan2017theoretical,durmus2017nonasymptotic,durmus2016high,cheng2017convergence} obtained precise bounds on $W_2(\cL(X_k),\pi)$ in terms of dimension and problem parameters. These bounds have been then improved in \cite{dalalyan2017user}.  Here we obtain precise convergence rates in the $L^1$ and $L^2$-Wasserstein distances working only with Assumptions \ref{as diss}. In particular, our bound is reminiscent of the bounds in \cite{dalalyan2017user,durmus2016high} for the constant step size algorithms. Indeed we show the following result.

\begin{theorem}\label{mainTheoremULA}
	Suppose that Assumptions \ref{as diss} are satisfied. Then there exist a function $f$ and constants $h_0$, $\hat{c}_1$, $\hat{c}_2$, $C_1$, $C_2$, $\widetilde{c}_1$, $\widetilde{c}_2 > 0$ such that for all $h \in (0, h_0)$
	\begin{equation}\label{eq mainULAW2}
	W_2(\cL (X_k), \pi) \leq C_2 (1- \hat{c}_2 h)^{k/2} \left( W_f(\cL(X_0), \pi) \right)^{1/2} + \widetilde{c}_2 h^{1/4}
	\end{equation}
	and
	\begin{equation}\label{eq mainULAW1}
	W_1(\cL (X_k), \pi) \leq C_1 (1- \hat{c}_1 h)^{k} W_f(\cL(X_0), \pi) + \widetilde{c}_1 h^{1/2} \,.
	\end{equation}
\end{theorem}

The precise values of all the constants and the formula defining the function $f$ in Theorem \ref{mainTheoremULA} will be given in Section \ref{sectionULA}.

Our next observation is that the perturbation inequality (\ref{perturbationInequality}) sheds light onto stochastic gradient algorithms or Langevin models with inaccurate/randomised gradients. Examples of such processes have been used in \cite{welling2011bayesian} in the context of sampling and studied in \cite{dalalyan2017user,Raginsky2017,teh2016consistency,vollmer2015non,nagapetyan2017true}. We remark that randomisation is a successful technique that is known to improve convergence properties for problems with non-smooth coefficients \cite{kruse2017error,przybylowicz2014strong,jentzen2009random}. 
 We define a function $\bar{b}:\bR^d \times \bR^n \rightarrow \bR^d$. Consider an $\mathbb{R}^n$-valued random variable $U$ such that $\bE[ \bar{b}(x,U)]=b(x)$ for all $x \in \mathbb{R}^d$. Let $(\xi_k)_{k=1}^{\infty}$, as before, be i.i.d.\ with $\xi_k \sim N(0,I)$ and take i.i.d.\ copies $(U_k)_{k=0}^{\infty}$ of $U$ that are independent of $(\xi_k)_{k=1}^{\infty}$. We define the following Markov chain     
\begin{equation} \label{eq ineuler}
\begin{cases}
\bar{X}_{k+1} = \bar{X}_k + \bar{b}(\bar{X}_k,U_{k})h + \sqrt{h} \xi_{k+1}\,, \quad k \geq 0 \,, \\
\bar{X}_0 \sim \mu_0 \,.
\end{cases}
\end{equation}
Note that for each $k$, the random variable $\bar{X}_k$ is independent of $U_k$ and that $\bE[\bar{b}(\bar{X}_k,U_k)|\bar{X}_k] = b(\bar{X}_k)$. For each $k$ we consider the conditional variance  
$
\bV[ \bar{b}(\bar{X}_k,U_{k}) | \bar{X}_k] := \bE [|  \bar{b}(\bar{X}_k,U_{k}) - b(\bar{X}_k) |^2 | \bar{X}_k] ,
$
which as we shall see plays a key role in our analysis in Section \ref{sectionStochasticGradient}. We need to impose the following assumption.

 \begin{assumptionsingle}[Variance of inaccurate drift] \label{as ina}
	There exist constants $\sigma$, $\alpha > 0$ such that for any $x \in \mathbb{R}^d$, any $h > 0$ and any random variable $U$ such that $\mathbb{E}[\bar{b}(x,U)] = b(x)$, we have
		\begin{equation}\label{ina estimatorVariance}
		\mathbb{E} \left| \bar{b}(x, U) - b(x) \right|^2 \leq \sigma^2 (1+|x|^2) h^{\alpha} \,.
		\end{equation}
\end{assumptionsingle}

Note that the dependence on $h$ of the right hand side of (\ref{ina estimatorVariance}) is related to the choice of the estimator $\bar{b}$. On the other hand, the constants $\sigma$ and $\alpha$ do not depend on $h$. We will discuss how to verify Assumption \ref{as ina} in the case where the drift is estimated by subsampling in Example \ref{ex subsampling}.

We can study properties of $(\bar X_k)_{k=0}^{\infty}$ by treating it as a perturbed version of $(X_k)_{k=0}^{\infty}$ given by (\ref{eq euler}).  
This allows us to study convergence of $\cL(\bar X_{k})$ to $\pi$. Indeed, using the fact that $W_2$ satisfies the triangle inequality, we have  $  W_2(\cL(\bar{X}_{k} ), \pi ) \leq W_2(\cL(\bar X_{k} ), \cL(X_{k}) ) + W_2(\cL(X_{k}) , \pi) $ and the bound on $ W_2(\cL(\bar{X}_{k} ), \cL(X_{k}) ) $ will follow from inequality (\ref{perturbationInequality}). Hence we obtain the following result.

\begin{theorem}\label{mainStochasticGradientTheorem}
	Let Assumptions \ref{as diss} and \ref{as ina} hold. Then there exist constants $h_0$, $\bar{C}_1$, $\bar{C}_2 > 0$ such that for all $h \in (0, h_0)$
	\begin{equation}\label{mainSGW2}
	W_2(\cL(\bar{X}_k), \pi) \leq W_2(\cL(X_k), \pi) + \bar{C}_2 h^{\alpha/4}
	\end{equation}
	and
	\begin{equation}\label{mainSGW1}
	W_1(\cL(\bar{X}_k), \pi) \leq W_1(\cL(X_k), \pi) + \bar{C}_1 h^{\alpha/2} \,.
	\end{equation}
\end{theorem}
The precise values of the constants and the proof of Theorem \ref{mainStochasticGradientTheorem} can be found in Section \ref{sectionStochasticGradient}. Note that our bounds in Theorem \ref{mainStochasticGradientTheorem} are of similar form as the bounds obtained in Theorem 4 in \cite{dalalyan2017user} in a setting corresponding to the global contractivity assumption (\ref{eq mono}), i.e., in \cite{dalalyan2017user} the distance $W_2(\cL(\bar{X}_k), \pi)$ is also bounded by $W_2(\cL(X_k), \pi)$ plus an additional error term coming from the use of an inaccurate drift. However, the error term in $W_2$ in the global contractivity case in \cite{dalalyan2017user} is obtained under an assumption similar to our (\ref{ina estimatorVariance}) with $\alpha = 0$ and is of order $\sigma h^{1/2}$, whereas our error term is of order $\sigma h^{\alpha / 4}$ in $W_2$ and $\sigma^2 h^{\alpha / 2}$ in $W_1$ for $\alpha \geq 0$ (dependence on $\sigma$ of the constants $\bar{C}_1$ and $\bar{C}_2$ can be easily traced in Section \ref{sectionStochasticGradient}). We believe that the worse order of the constants in our case is a necessary consequence of the much more general contractivity at infinity assumption, however, it remains an open question whether our constants are actually optimal. To our knowledge, the only related result without assuming global contractivity of the drift was obtained in \cite{Raginsky2017}, see Proposition 10 therein, by using functional inequalities. However, the estimates in \cite{Raginsky2017} are not uniform in time.

 Note that the bounds we obtain in Theorem \ref{mainStochasticGradientTheorem} depend on the variance of the estimator $\bar{b}$ of the drift $b$ via $h^{\alpha}$ appearing in (\ref{ina estimatorVariance}). This is in contrast with the following weak convergence result, which at least in the context of randomised Euler schemes seems to be new.

\begin{theorem} \label{th weak}
	Let Assumption \ref{as ina} hold. Let $g\in C^{\infty}(\bR^d,\bR)$ with polynomial growth and assume that $b \in C^3$ has bounded derivatives and that there are constants $M_1$, $M_2 > 0$ such that for all $x \in \mathbb{R}^d$ we have
	\begin{equation}\label{eq th weak dissipativity}
     \langle x , b(x) \rangle \leq M_2 - M_1|x|^2 \,.
	\end{equation}
	Furthermore, assume that for any $p \geq 1$ there is a constant $C_{\bar{b}}^p > 0$ such that for any $x \in \mathbb{R}^d$ and for any random variable $U$ with $\mathbb{E}[\bar{b}(x,U)] = b(x)$ we have 
	\begin{equation}\label{driftEstimatorGrowth}
	\mathbb{E}|\bar{b}(x,U)|^p \leq C_{\bar{b}}^p(1 + |x|^p) \,.
	\end{equation}
	Then there exists a constant $c_w>0$ independent of $h$ such that
	\[
	\sup_{k\in \bN }\left| \bE[ g(Y_{k h}) ]   - \bE[g(\bar X_{k})] \right| \leq c_w h\,.
	\]	
\end{theorem}

Note that condition (\ref{eq th weak dissipativity}) is implied by (\ref{driftDissipativityAtInf}), see also Lemma \ref{lemmaContractivityAtInfinityImpliesLyapunov}. Moreover, condition (\ref{driftEstimatorGrowth}) for all $p \geq 1$ is a relatively weak assumption that is implied e.g.\ by the Lipschitz continuity of $\bar{b}(\cdot, U)$ (cf.\ (\ref{ina driftLipschitz})) and a moment bound on $\bar{b}(0,U)$ for all $U$ as in (\ref{driftEstimatorGrowth}). The proof of Theorem \ref{th weak}, unlike all the other results in this paper, does not rely on the perturbation inequality (\ref{perturbationInequality}) and instead uses PDE methods and estimates from \cite{talay1990second}. The proof can be found in Section \ref{sectionWeak}. The above result, in addition to being interesting on its own, will be also applied in our analysis of Multi-level Monte Carlo (MLMC) in Section \ref{sectionMLMC}.

Non-asymptotic results on the $L^2$-Wasserstein distance pave the way to the analysis of the MLMC method, which is a variance reduction technique introduced in \cite{heinrich2001multilevel}, \cite{MR2187308} and \cite{giles2008multilevel}. Let $\pi$ be the invariant measure of (\ref{eq sde}) and let $g\in C^{\infty}(\bR^d,\bR)$ be Lipschitz. Our framework allows us to consider multi-level type estimators of $\int_{\mathbb{R}^d} g(x) \pi(dx)$, involving Euler schemes with inaccurate drifts, which have been treated before only in \cite{szpruch2016multi}. However, the authors of \cite{szpruch2016multi} considered exclusively the globally contractive case, whereas in the present paper we employ the coupling method in order to deal with the more general contractivity at infinity setting. The multi-level estimator will be introduced in detail in Section \ref{sectionMLMC}. For our results on MLMC with inaccurate drift we need to impose additional assumptions.

 \begin{assumption}[MLMC with inaccurate drift] \label{as MLMCina}
	The function $\bar{b} : \mathbb{R}^d \times \mathbb{R}^n \to \mathbb{R}^d$ satisfies the following conditions:
	\begin{enumerate}[i)]
				\item Lipschitz condition: There is a constant $\bar{L} > 0$ such that for all $x$, $y \in \mathbb{R}^d$ and all random variables $U$ such that $\mathbb{E}[\bar{b}(x,U)] = b(x)$ and $\mathbb{E}[\bar{b}(y,U)] = b(y)$ we have
		        \begin{equation}\label{ina driftLipschitz}
				|\bar{b}(x,U) - \bar{b}(y,U)| \leq \bar{L}|x - y| \text{ a.s.}
				\end{equation}
				\item Contractivity at infinity condition: There exist constants $\bar{K}$, $\bar{R} > 0$ such that for all $x$, $y \in \mathbb{R}^d$ with $|x - y| \geq \bar{R}$ and for all random variables $U$ such that $\mathbb{E}[\bar{b}(x,U)] = b(x)$ and $\mathbb{E}[\bar{b}(y,U)] = b(y)$ we have
				\begin{equation} \label{ina contractivity}
				\langle  x - y , \bar{b}(x, U) - \bar{b}(y, U) \rangle \leq - \bar{K} |x - y|^2 \text{ a.s.}
				\end{equation}
	\end{enumerate}
\end{assumption}

We have the following result.

\begin{theorem}\label{theoremMainMLMC}
	Let all the assumptions of Theorem \ref{th weak} hold. Additionally, suppose that Assumptions \ref{as MLMCina} are satisfied and that $g$ is Lipschitz. Then for the estimation of $\int_{\mathbb{R}^d} g(x) \pi(dx)$ by MLMC with inaccurate drift we have computational complexity $\mathcal{O}(\epsilon^{-2-(1-\min\{1,\alpha\}/2)}\log \epsilon)$ with $\alpha$ given in (\ref{ina estimatorVariance}), as opposed to complexity $\mathcal{O}(\epsilon^{-3} \log \epsilon)$ in the standard Monte Carlo approach.
\end{theorem}

The proof of Theorem \ref{theoremMainMLMC} and a detailed description of our MLMC estimator can be found in Section \ref{sectionMLMC}. Note that our results apply also to MLMC for Euler schemes with non-randomised drifts. There $\alpha = \infty$ and we obtain a gain in complexity from $\mathcal{O}(\epsilon^{-3} \log \epsilon)$ to $\mathcal{O}(\epsilon^{-5/2} \log \epsilon)$ under Assumptions \ref{as diss}.

\begin{remark}
	While we were working on the present paper, Fang and Giles in \cite{FangGiles2018} independently obtained results on MLMC for SDEs with non-globally contractive drifts using a change of measure argument. Their approach, contrary to ours, does not aim at obtaining bounds on the variance that are uniform in time. While their variance grows linearly with time, they have better strong convergence rates and hence in the accurate drift case overall complexity of their algorithm is better than here. However, our framework is more flexible and allows us to cover also Euler schemes with inaccurate drifts.
\end{remark}

\begin{remark}
While completing this work, the paper \cite{ChengJordan2018} appeared, where results analogous to our Theorem \ref{mainTheoremULA} were obtained under assumptions on the drift similar to ours, although only in the $L^1$-Wasserstein distance. More precisely, the focus in \cite{ChengJordan2018} is on determining the smallest number of iterations of the Euler scheme required to approach the invariant measure of the SDE (\ref{eq sde}) in the $W_1$ distance with precision $\varepsilon$. In our setting, we can infer from our Theorem \ref{mainTheoremULA} that this number of iterations is of order $\mathcal{O}(d/\varepsilon^2)$, which is the same as in \cite{ChengJordan2018}. To see this, note that the constant $\widetilde{c}_1$ in (\ref{eq mainULAW1}) is of order $\mathcal{O}(\sqrt{d})$, cf. (\ref{eq ULA W1}) and (\ref{defCdif}), and hence a similar analysis as in the proof of Theorem 2.1 in \cite{ChengJordan2018} gives the required result. Interestingly, for an analogous problem in the $L^2$ Wasserstein distance, our estimates give us a required number of iterations of order $\mathcal{O}(d^2/\varepsilon^4)$, since the constant $\widetilde{c}_2$ in (\ref{eq mainULAW2}) is also of order $\mathcal{O}(\sqrt{d})$, but is multiplied by $h^{1/4}$ instead of $h^{1/2}$. However, we do not know whether this result is sharp or just an artifact of our proof. We are not aware of any comparable estimates in the $W_2$ distance in the non-convex setting in the literature. Furthermore, we remark that determining optimality of such bounds is usually a non-trivial task, as it requires finding lower bounds for the Wasserstein distances, see e.g.\ Example 4 in \cite{Eberle2016} for a related discussion in the diffusion case.
\end{remark}

\section{Contractivity of Euler schemes and applications}\label{sectionMain}

Before we present our results in detail, let us briefly discuss why the classical methods may fail if we do not assume global log-concavity.

\subsection{Motivation}\label{sec motivation}

\paragraph{Synchronous coupling}

It is well known that under the global convexity assumption (\ref{eq mono}) on the drift, we can show that Euler schemes are contractive in the $L^2$ distance just by applying the synchronous coupling. Namely, if we have
\begin{equation}\label{oneStepEuler}
X' := x + hb(x) + \sqrt{h}Z =: \hat{x} + \sqrt{h}Z \,,
\end{equation}
where $Z \sim N(0,I)$ and we define
\begin{equation*}
Y' := y + hb(y) + \sqrt{h}Z =: \hat{y} + \sqrt{h}Z \,,
\end{equation*}
then for all sufficiently small $h$ we obtain
\begin{equation*}
\mathbb{E}|X' - Y'|^2 = |\hat{x} - \hat{y}|^2 \leq (1-ch)|x-y|^2
\end{equation*}
for some constant $c \in (0,1)$.
However, under the dissipativity at infinity assumption (\ref{driftDissipativityAtInf}) the synchronous coupling is no longer sufficient. As a one-dimensional example, consider the function
\begin{equation*}
U(x) := x^2 (g_n(x))^2 + a^2 - 2a x g_n(x) 
\end{equation*}
for some parameters $a > 0$ and $n \geq 1$, where
\begin{equation*}
g_n(x) := \begin{cases}
- n &\text{ if } x \in (-\infty,-n) \\
x &\text{ if } x \in [-n,n] \\
n &\text{ if } x \in (n, \infty) 
\end{cases} \,.
\end{equation*}
The function $U$ is constructed by truncating the function $x \mapsto (x^2 - a)^2$ so that $U'$ is Lipschitz. We consider the drift
\begin{equation*}
b(x) = - U'(x) \,.
\end{equation*}
It is easy to check that $b$ satisfies the dissipativity at infinity assumption. Indeed, when $|x|$, $|y| \leq n$ then we have
\begin{equation}\label{exampleDissipativity}
\langle x - y , b(x) - b(y) \rangle = - 4(x-y)^2 (x^2 + xy + y^2 - a)
\end{equation}
and we see that there exist constants $R_0$, $C > 0$ such that we have
\begin{equation*}
\langle x - y , b(x) - b(y) \rangle \leq - C (x-y)^2
\end{equation*}
when $|x - y| > R_0$. Similarly, when $|x|$, $|y| \geq n$, we get
\begin{equation*}
\langle x - y , b(x) - b(y) \rangle = -2n(x-y)^2 + 2an\left( \frac{x}{|x|} - \frac{y}{|y|} \right)(x-y)
\end{equation*}
and again if the distance $|x-y|$ is large enough, then the desired inequality holds. When $y > n > x$ then the dissipativity condition
\begin{equation*}
\left(- 4x^3 + 4ax +2ny - 2an \frac{y}{|y|} \right)(x-y) \leq - C(x-y)^2
\end{equation*}
is equivalent to
\begin{equation*}
- 4x^3 + (4a + C)x +(2n-C)y - 2an \frac{y}{|y|} \geq 0 \,,
\end{equation*}
which we can make sure is satisfied by choosing the parameters in an appropriate way, and the other cases follow by symmetry. However, from (\ref{exampleDissipativity}) we also see that there exist constants $R_1$, $R_2 > 0$ such that when $|x|$, $|y| \leq R_1$ and $|x - y| \leq R_2$, we have
\begin{equation*}
\langle x - y , b(x) - b(y) \rangle \geq C_1(x-y)^2
\end{equation*}
for some constant $C_1 > 0$. Hence we see that if we use the synchronous coupling for our Euler scheme, we end up having
\begin{equation*}
\begin{split}
\mathbb{E}|X' - Y'|^2 &= |\hat{x} - \hat{y}|^2 = |x - y|^2 + h^2 |b(x) - b(y)|^2 + 2h (b(x) - b(y))(x - y) \\
&\geq |x-y|^2 + 2hC_1|x-y|^2 
\end{split}
\end{equation*}
and thus we cannot have a contraction.

\paragraph{Talagrand inequality}

An alternative approach to proving contractivity of Euler schemes in the $L^2$-Wasserstein distance in the global convexity setting relies on the fact that the Gaussian measure $\mu$ with covariance matrix $\Sigma$ satisfies the Talagrand inequality with the constant being the largest eigenvalue $\lambda_{max}(\Sigma)$, i.e., for any measure $\nu$ absolutely continuous with respect to $\mu$ we have
\begin{equation}\label{ineq talagrand}
W_2(\nu, \mu)^2 \leq 2 \lambda_{max}(\Sigma) H(\nu | \mu) = 2 \lambda_{max}(\Sigma) \int \log{\frac{d \nu}{d \mu}} d \nu \,,
\end{equation}
see e.g. \cite{Talagrand1996} or page 2726 in \cite{DjelloutGuillinWu2004} and the references therein. Namely, taking $\nu = N(x + hb(x), hI)$ and $\mu = N(y + hb(y), hI)$ we have
\begin{equation*}
\begin{split}
H(\nu | \mu) &= \int \log \exp{\left(\frac{-(z - (x + hb(x)))^2}{2h}\right)} \exp{\left(\frac{(z - (y + hb(y)))^2}{2h}\right)} \nu(dz) \\
&= \frac{1}{2h} \int  \left( -z^2 +2z(x+hb(x)) - (x+hb(x))^2 \right) + \left( z^2- 2z(y+hb(y)) + (y+hb(y))^2 \right) \nu(dz) \\
&= \frac{1}{2h} \left( (x+hb(x))^2 - 2(x+hb(x))(y+hb(y)) + (y + hb(y))^2 \right) \\
&= \frac{1}{2h} \left( x+ hb(x) - y - h b(y) \right)^2 
\end{split}
\end{equation*} 
and hence the right hand side of (\ref{ineq talagrand}) is equal to $(\hat{x} - \hat{y})^2$. In the global convexity setting this can be bounded from above by $(1-ch)|x-y|^2$ for all sufficiently small $h > 0$.
However, in the non-convex setting we can use the same example as the one presented above, where $(\hat{x} - \hat{y})^2$ is bounded from below by $|x-y|^2 + 2hC_1|x-y|^2$ with some $C_1 > 0$. This shows that the approach via Talagrand's inequality fails as well.

\subsection{Coupling and Wasserstein (pseudo) distances}\label{sectionCoupling}

The random vector $(X^x_h, Y^y_h)$ in (\ref{perturbationInequality}) is an example of a \emph{coupling} of random variables. Constructing different random objects with the same distributions is a widely applied tool in probability theory, see e.g. \cite{Lindvall, Thorisson} for general overview and \cite{EberleMajka2018, NueskenPavliotis2018} and the references therein for other examples of applications of couplings to sampling algorithms. 

In a series of works \cite{Eberle2011,Eberle2016}, Eberle used the reflection coupling for diffusions to prove $W_1$ contractivity for the SDE (\ref{eq sde}) under Assumptions \ref{as diss}. More precisely, he proved that there exists a constant $\lambda>0$, expressed explicitly in terms of problem parameters, such that
\begin{equation}\label{eq W1contractivity}
W_1(\cL(Y_t^{\mu}),\cL(Y_t^{\nu}))\leq e^{-\lambda t} W_1(\mu,\nu) \,,
\end{equation}
where $Y_t^{\mu}$ denotes the solution to \eqref{eq sde} with $Y_0\sim \mu$. By taking $\nu=\pi$ in (\ref{eq W1contractivity}), i.e., by initialising the process at the invariant measure $\pi$, one immediately obtains geometric convergence rate of the law of the process $(Y_t)_{t \geq 0}$ to its stationary measure. 

In the present paper we extend the ideas from \cite{EberleMajka2018}, where results analogous to the ones from \cite{Eberle2016} have been obtained directly on the level of the Markov chain \eqref{eq euler} in $W_1$, by employing a suitably chosen coupling and constructing a special Kantorovich ($L^1$-Wasserstein) distance, see Section 2.4 therein. Here we introduce a novel coupling construction and a new Wasserstein-type pseudo-distance, which allows us to obtain $L^2$ bounds on \eqref{eq euler} as well as on its perturbed version, and hence to analyse convergence of several sampling algorithms presented below.

One of the benefits of working with the $L^2$-Wasserstein distance is that through the Kantorovich duality theory (see e.g. Theorem 5.10 in \cite{Villani}) it significantly extends the class of functions $g$ for which we can obtain explicit convergence rates of functionals $g(X_k)$ of Euler schemes (\ref{eq euler}). When working with the $L^1$ Wasserstein distance, duality theory gives access only to Lipschitz functions, whereas in our setup we can also consider e.g. functions of quadratic growth. Moreover, $L^2$ bounds are necessary for our analysis of the variance of multi-level Monte Carlo estimators in Section \ref{sectionMLMC}.

We fix $h > 0$ and we consider one step of the Euler scheme for the equation (\ref{eq sde}), started at a point $x \in \mathbb{R}^d$, i.e., we have a random variable $X'$ given by
\begin{equation}\label{eq oneStep}
X' = \hat{x} + \sqrt{h} Z \,,
\end{equation}
where 
\begin{equation*}
\hat{x} = x + hb(x)
\end{equation*}
is the deterministic movement due to the drift $b$ and $Z \sim N(0,I)$, where $I$ is the $d \times d$ identity matrix. Hence we have $X' \sim N(x+hb(x),hI)$. For a given point $y \neq x$, we want to construct a new random variable $Y'$ that will have the distribution $N(y+hb(y),hI)$. We define
\begin{equation*}
\hat{y} = y + hb(y)
\end{equation*}
and
\begin{equation*}
\hat{r} = |\hat{x} - \hat{y}| \,.
\end{equation*}
We want to define a coupling of the random movement $(\hat{x},\hat{y}) \mapsto (X',Y')$. Note that in \cite{EberleMajka2018} bounds in $L^1$ were obtained by choosing 
\begin{equation}\label{EberleMajkaCoupling}
Y' = \begin{cases}
X' \,, & \text{if } \zeta \leq \left( \varphi_{\hat{y},hI}(X') \wedge \varphi_{\hat{x},hI}(X') \right) / \varphi_{\hat{x},hI}(X') \,, \\
\hat{y} + R_{\hat{x},\hat{y}}\sqrt{h} Z \,, & \text{if } \zeta > \left( \varphi_{\hat{y},hI}(X') \wedge \varphi_{\hat{x},hI}(X') \right) / \varphi_{\hat{x},hI}(X') \,.
\end{cases}
\end{equation}
Here $\zeta$ is a uniformly distributed random variable on $[0,1]$ independent of $Z$, $R_{x,y} = I - 2 (x-y)(x-y)^T/|x-y|^2$ is the reflection operator with respect to the hyperplane spanned by $(x-y) / |x-y|$ (note that if the dimension $d = 1$ then $R_{x,y}u = - u$ for any $u \in \mathbb{R}$) and $\varphi_{z,A}(u)$ is the density of $N(z,A)$ for $z \in \mathbb{R}^d$ and $A \in \mathbb{R}^{d \times d}$. We call (\ref{EberleMajkaCoupling}) the mirror coupling.

Note that this type of coupling is related to the one used in the optimal transport theory in \cite{McCann}, where it was shown to be the optimal coupling for all concave costs in the one-dimensional case, for a class of probability measures that includes Gaussian measures (but is in fact much broader). This makes (\ref{EberleMajkaCoupling}) the right choice of coupling for obtaining optimal bounds in concave Wasserstein distances. However, it needs to be modified to work for convex distances such as $W_2$.

In the present paper we use a combination of the synchronous coupling and the coupling given by (\ref{EberleMajkaCoupling}), defined in the following way.

We choose two parameters $H > 0$ and $m > 0$. If $\hat{r} > H$, we choose the synchronous coupling, i.e., we set $Y' = \hat{y} + \sqrt{h} Z$. If $\hat{r} \leq H$, we compare the size of the jump (i.e., the size of $\sqrt{h}Z$) with the parameter $m$. If $|\sqrt{h}Z| > m$, we again define $Y'$ synchronously. Otherwise, we use the variable $\zeta$ to apply the mirror coupling defined by (\ref{EberleMajkaCoupling}) to the Gaussians truncated by $m$.

To be precise, we have
\begin{equation}\label{ourCoupling}
\begin{split}
X' &= \hat{x} + \sqrt{h} Z \\
Y' &= \begin{cases}
X' \,, & \text{if } \zeta \leq \left( \varphi^m_{\hat{y},hI}(X') \wedge \varphi^m_{\hat{x},hI}(X') \right) / \varphi^m_{\hat{x},hI}(X') \text{ and } |\sqrt{h} Z| < m \text{ and } \hat{r} \leq H \,, \\
\hat{y} + R_{\hat{x},\hat{y}}\sqrt{h} Z \,, & \text{if } \zeta > \left( \varphi^m_{\hat{y},hI}(X') \wedge \varphi^m_{\hat{x},hI}(X') \right) / \varphi^m_{\hat{x},hI}(X') \text{ and } |\sqrt{h} Z| < m \text{ and } \hat{r} \leq H \,, \\
\hat{y} + \sqrt{h} Z \,, & \text{if } |\sqrt{h} Z| \geq m \text{ or } \hat{r} > H \,.
\end{cases}
\end{split}
\end{equation}

Here $\varphi^m_{z,A}(u) := \mathbf{1}_{\{ |u-z| \leq m \} } \varphi_{z,A}(u)$. Note that we only need to evaluate $\varphi^m_{\hat{x},hI}(X')$ when $|\sqrt{h}Z| < m$ (or, equivalently, when $|X' - \hat{x}| < m$) and hence we always have $\varphi^m_{\hat{x},hI}(X') \neq 0$, which implies that all the quantities in (\ref{ourCoupling}) are well-defined. It is easy to prove that (\ref{ourCoupling}) is indeed a coupling of $N(x+hb(x),hI)$ and $N(y+hb(y),hI)$, see Theorem \ref{couplingTheorem} in the Appendix.

Note that $Y'$ defined above can be thought of as a function of the initial points $x$, $y \in \mathbb{R}^d$ and the random variable $Z$. It also depends on the truncation parameters $m$ and $H$. In the sequel we will use the shorthand notation
\begin{equation}\label{psiCoupling}
\psi_{m,H}(x,y,Z) := Y' 
\end{equation}
for the random variable $Y'$ defined by (\ref{ourCoupling}).

Our goal is to expand on the methods introduced in \cite{EberleMajka2018}, where, under the same assumptions on the drift $b$ as in the present paper, the coupling $(X',Y')$ given by (\ref{EberleMajkaCoupling}) was used in order to obtain estimates of the form
\begin{equation*}
\mathbb{E}f(|X' - Y'|) \leq (1-ch)f(|x - y|)
\end{equation*}
for a specially constructed increasing, concave function $f$. Since the function $f$ in \cite{EberleMajka2018} is comparable both from above and from below with a rescaled identity function, this implies bounds for the first moment $\mathbb{E}|X' - Y'|$, as well as contractivity of the laws of the Euler scheme in a Kantorovich ($L^1$ Wasserstein) distance. The latter is a strong property with multiple important consequences, see e.g. \cite{JoulinOllivier} or the discussion in \cite{EberleMajka2018}. Here we use a modified coupling and a different construction of a distance function $f: [0,\infty) \to [0,\infty)$, which in our case is concave up to some $r_2 > 0$ and convex afterwards. This allows us to get upper bounds for the second moment $\mathbb{E}|X' - Y'|^2$, cf. Lemma \ref{corollarySquareComparison}. 

Related work in the diffusion case has been done in \cite{LuoWang}, where the authors adapted the $L^1$ bounds from \cite{Eberle2016} and obtained $L^p$ bounds for $p \geq 1$ under similar assumptions as in \cite{Eberle2016}. Here, however, we use a different, more direct construction of the auxiliary distance function $f$. We also introduce a novel coupling construction, which is specific to the discrete time case. 

Before we formulate our main result, let us introduce some notation. We define positive constants
\begin{equation*}
c_0 := 4 \min \left( \int_0^{1/2} u^2 (1 - e^{u-1/2})\varphi_{0,1}(u) du , (1 - e^{-1}) \int_0^{1/2} u^3 \varphi_{0,1}(u) du \right) \,,
\end{equation*}
where $\varphi_{0,1}$ is the density of $N(0,1)$, and
\begin{equation}\label{defh0}
h_0 := \min \left( \frac{K}{L^2}, \frac{4}{K}, \frac{1}{2L}, \frac{2 c_0 \ln{\frac{3}{2}}}{27L^2 \mathcal{R}^2}, \frac{\mathcal{R}^2}{4}, \frac{c_0^2 (\ln{2})^2}{144L^2 \mathcal{R}^2} \right) \,,
\end{equation}
where the constants $L$, $K$ and $\mathcal{R} > 0$ are all specified in Assumptions \ref{as diss}. Then we put $r_1 := (1 + h_0L)\mathcal{R}$ and $r_2 := r_1 + \sqrt{h_0}$.
Finally, we choose the parameter
\begin{equation}\label{defa}
a := \frac{6Lr_1}{c_0}
\end{equation}
and we construct our function $f: [0,\infty) \to [0,\infty)$ by setting
\begin{equation}\label{defF}
f(r) := \begin{cases} 
\frac{1}{a}(1 - e^{-ar}) &\text{ if } r \leq r_2 \\
\frac{1}{a}(1 - e^{-ar_2}) + \frac{1}{2r_2}e^{-ar_2}(r^2 - r_2^2) &\text{ if } r > r_2 
\end{cases} \,.
\end{equation}
Then we have the following result.

\begin{theorem}\label{mainContractivityTheorem}
	Let Assumptions \ref{as diss} hold. For the coupling $(X',Y')$ given by (\ref{ourCoupling}) with parameters $m = \frac{\sqrt{h_0}}{2}$ and $H = r_1$ and the function $f$ given by (\ref{defF}), we have
	\begin{equation}\label{contractivity}
	\mathbb{E}f(|X' - Y'|) \leq (1-ch)f(|x - y|)
	\end{equation}
	for all $h \in (0,h_0)$, where
	\begin{equation}\label{contractivityConstant}
	c = \min \left( e^{-ar_2} \frac{K}{4}, \frac{\frac{1}{2}e^{-ar_2}r_2}{\frac{1}{a}\left( 1 - e^{-ar_2}\right)}\frac{K}{4}, \frac{9L^2r_1^2}{2c_0}e^{-\frac{6Lr_1^2}{c_0}}, \frac{3Lr_1}{16\sqrt{h_0}} \right) \,.
	\end{equation}
\end{theorem}

The proof will be presented in Section \ref{sectionProofContractivity}. As an immediate consequence, we obtain bounds in the $W_f$ Wasserstein (pseudo) distance. 

\begin{corollary}\label{mainCorollary}
	Let $p(x,A)$ for $x \in \mathbb{R}^d$ and $A \in \mathscr{B}(\mathbb{R}^d)$ denote the transition kernel of $X'$ given by (\ref{eq oneStep}) and let $\mu p (dy) := \int \mu(dx) p(x,dy)$ for any probability measure $\mu$ on $\mathbb{R}^d$. Under the assumptions of Theorem \ref{mainContractivityTheorem}, we have
	\begin{equation*}
	W_f(\mu p, \nu p) \leq (1-ch) W_f(\mu, \nu)
	\end{equation*}
	for all $h \in (0, h_0)$ and for all $\mu$, $\nu \in \mathcal{P}_2(\mathbb{R}^d)$.
\end{corollary}
The above result is a straightforward consequence of (\ref{contractivity}) and the definition of $W_f$, see also the proof of Lemma 2.1 in \cite{EberleMajka2018}.

An important feature of the function $f$ given by (\ref{defF}) is that it is comparable from below with the identity and the square functions. This allows us to obtain both $L^1$ and $L^2$ bounds as an immediate consequence of Theorem \ref{mainContractivityTheorem}.

\begin{lemma}\label{corollarySquareComparison}
	We have
	\begin{equation}\label{L1L2boundsOnf}
	r^2 \leq Af(r) \text{ and } r \leq e^{ar_2}f(r) 
	\end{equation} 
	for all $r \in [0,\infty)$, where 
	\begin{equation}\label{defA}
	A := \max \left( \frac{ar_2^2}{1 - e^{-ar_2}} , 2r_2 e^{ar_2} \right) \,.
	\end{equation}
	\begin{proof}
		The second bound in (\ref{L1L2boundsOnf}) is obvious since $f'(r) \geq e^{-ar_2}$ for all $r \geq 0$. Now let us try to look for a constant $A > 0$ such that
		\begin{equation}\label{r2Bound}
		r_2^2 \leq Af(r_2) = A\frac{1}{a}(1 - e^{-ar_2}) \,.
		\end{equation}
		If this holds, then for any $r \leq r_2$ we have $r^2 \leq Af(r)$. Hence we need $A \geq ar_2^2/(1 - e^{-ar_2})$. Moreover, we need
		$r^2 = r_2^2 - r_2^2 + r^2 \leq A\frac{1}{a}(1 - e^{-ar_2}) + A\frac{1}{2r_2}e^{-ar_2}(r^2 - r_2^2)$
		for any $r \geq r_2$, hence, using (\ref{r2Bound}), we see that we need
		$r^2 - r_2^2 \leq A\frac{1}{2r_2}e^{-ar_2}(r^2 - r_2^2)$,
		which implies $A \geq 2r_2 e^{ar_2}$.
	\end{proof}
\end{lemma}

\begin{corollary}
	Under the assumptions of Theorem \ref{mainContractivityTheorem}, we have
	\begin{equation*}
	W_2(\mu p, \nu p) \leq A(1-ch)W_f(\mu, \nu) \qquad \text{ and } \qquad W_1(\mu p, \nu p) \leq e^{ar_2} (1-ch)W_f(\mu, \nu)
	\end{equation*}
	for all $h \in (0,h_0)$ and for all $\mu$, $\nu \in \mathcal{P}_2(\mathbb{R}^d)$, where $A$ is given by (\ref{defA}) and all the other constants are as in Theorem \ref{mainContractivityTheorem}.
\end{corollary}

An important consequence of Theorem \ref{mainContractivityTheorem} that turns out to be crucial for our applications, is the following inequality.

\begin{theorem}\label{theoremPerturbation}
	Let $(X',Y')$ be the coupling given by (\ref{ourCoupling}) with parameters $m = \frac{\sqrt{h_0}}{2}$ and $H = r_1$ and let $f$ be the function given by (\ref{defF}), where all the constants are as specified in Theorem \ref{mainContractivityTheorem}. Let $\widetilde{X}$ be a random variable. Then
	\begin{equation}\label{perturbation}
	\begin{split}
	\mathbb{E}f(|Y' - \widetilde{X}|) &\leq (1-ch)f(|y - x|) + \mathbb{E} \left[ \frac{1}{r_2}e^{-ar_2} |X' - \widetilde{X}|^2 \right] \\
	&+ \mathbb{E} \left[ \left( 1 + \frac{1}{r_2}e^{-ar_2} \left( |y-x|(1+hL) + \sqrt{h_0} \right) \right) |X' - \widetilde{X}| \right] \,.
	\end{split}
	\end{equation}
	\begin{proof}
		We have
		\begin{equation*}
		\mathbb{E}f(|Y' - \widetilde{X}|) - f(|y-x|) = \mathbb{E}f(|Y' - \widetilde{X}|) - \mathbb{E}f(|Y' - X'|) + \mathbb{E}f(|Y' - X'|) - f(|y-x|) \,.
		\end{equation*}	
		By (\ref{contractivity}) we know that
		$\mathbb{E}f(|Y' - X'|) - f(|y-x|) \leq - chf(|y-x|)$
		and hence it is sufficient to focus on the expression $\mathbb{E}f(|Y' - \widetilde{X}|) - \mathbb{E}f(|Y' - X'|)$. Since $f$ is increasing, we have
		\begin{equation}\label{perturbationProof1}
		\mathbb{E}f(|Y' - \widetilde{X}|) - \mathbb{E}f(|Y' - X'|) \leq \mathbb{E}f(|Y' - X'| + |X' - \widetilde{X}|) - \mathbb{E}f(|Y' - X'|) \,.
		\end{equation}
		We can now apply the Taylor formula to see that the right hand side of (\ref{perturbationProof1}) is equal to
		$\mathbb{E} f'(\zeta) |X' - \widetilde{Y}|$, where $\zeta \in (|Y' - X'|, |Y' - X'| + |X' - \widetilde{X}|)$.
		From the definition of $f$ we get
		$f'(\zeta) \leq \frac{1}{r_2} e^{-ar_2} |\zeta| + e^{-a|\zeta|} \leq \frac{1}{r_2} e^{-ar_2}|\zeta| + 1$
		and hence
		\begin{equation*}
		\mathbb{E} f'(\zeta) |X' - \widetilde{X}| \leq \mathbb{E} \left[ \left( \frac{1}{r_2} e^{-ar_2}\left(|Y' - X'| + |X' - \widetilde{X}|\right) + 1 \right) |X' - \widetilde{X}| \right] \,.
		\end{equation*}
		Now observe that due to our construction of the coupling $(X',Y')$ we have
		$|Y' - X'| \leq |\hat{y} - \hat{x}| + 2m$.
		This is because an increase in the distance between the two coupled processes can happen in the random step $(\hat{x}, \hat{y}) \mapsto (X', Y')$ only due to reflection. However, reflection happens only if the normal random variable $Z$ takes a sufficiently small value and hence that increase is bounded by $2m$. Moreover, due to the Lipschitz condition (\ref{driftLipschitz}) of the drift we have
		$|\hat{y} - \hat{x}| \leq |y-x|(1+hL)$.
		Recall that in order for (\ref{contractivity}) to hold, we choose $m = \sqrt{h_0}/2$. Combining all our estimates together, we arrive at (\ref{perturbation}).
	\end{proof}
\end{theorem}

From the result above we see that if we choose $\widetilde{X}$ in such a way that we have control on the first and the second moment of the distance between $\widetilde{X}$ and $X'$, then we can get a good bound on the distance between $Y'$ and $\widetilde{X}$. Hence our result is useful for random variables $\widetilde{X}$ that are small perturbations of $X'$. Note that the role of $X'$ and $Y'$ in the proof of Theorem \ref{theoremPerturbation} is symmetric, hence we immediately obtain (\ref{perturbationInequality}). The reason for our choice of the form of inequality (\ref{perturbation}) will become apparent in the proofs in Sections \ref{sectionULA}, \ref{sectionStochasticGradient} and \ref{sectionMLMC}.

We can also obtain a related perturbation result based on Theorem 2.12 from \cite{EberleMajka2018}, which is a result similar to our Theorem \ref{mainContractivityTheorem} but with a globally concave function $f$. This allows us to get simpler formulas in cases where we only need $L^1$ bounds. Let us define $q = 7 c_0^{-1} L \mathcal{R}$ and 
\begin{equation}\label{defFconcave}
f_1(r) := \begin{cases} 
\frac{1}{q}(1 - e^{-qr}) &\text{ if } r \leq r_1^1 \\
\frac{1}{q}(1 - e^{-qr_1^1}) + e^{-qr_1^1}(r - r_1^1) &\text{ if } r > r_1^1 
\end{cases} \,,
\end{equation}
where $r_1^1 := (1 + h_0^1L)\mathcal{R}$ with $h_0^1$ specified in Theorem \ref{theoremContractivityConcave}. Then we have the following result.

\begin{theorem}\label{theoremContractivityConcave}
	Let Assumptions \ref{as diss} hold. Then for the coupling $(X',Y')$ given by (\ref{ourCoupling}) with parameters $m = \infty$ and $H = \infty$ and the function $f_{1}$ given by (\ref{defFconcave}), we have
	\begin{equation}\label{contractivityConc}
	\mathbb{E}f_{1}(|X' - Y'|) \leq (1-c_1h)f_{1}(|x - y|)
	\end{equation}
	for all $h \in (0, h_0^{1})$, where
	\begin{equation}\label{contractivityConstantConc}
	\begin{split}
c_1 &= \min  \left( \frac K2,\, \frac{245}{24 c_0} {L^2\mathcal R^2} \right)\, \exp\left({-\frac{49}{6 c_0}{L \mathcal R^2}}\right) \qquad \text{ and }\\ 
h_0^{1} &= \frac 1L\min \left(\frac 16,\, \frac KL,\, \frac 13{L\mathcal R^2},\,\frac{c_0^2}{970}\frac 1{L\mathcal R^2}\right) \,.
\end{split}
	\end{equation}
\end{theorem}

Note that the exact statement of Theorem 2.12 in \cite{EberleMajka2018} is slightly different than Theorem \ref{theoremContractivityConcave} above, since the metric in \cite{EberleMajka2018} depends on $h$. This can be easily modified by replacing the constant $\Lambda = \Lambda(h)$ defined by (2.59) in \cite{EberleMajka2018} with $L$ and $r_1 = r_1(h)$ defined by (2.65) therein with $r_1^1 := (1 + h_0^1L)\mathcal{R}$. Then the proof in \cite{EberleMajka2018} easily carries over to our setting and we obtain Theorem \ref{theoremContractivityConcave}. We leave the details to the interested reader. More importantly, from inequality (\ref{contractivityConc}) we can easily derive another perturbation result.

\begin{theorem}\label{theoremPerturbationConcave}
	 Let Assumptions \ref{as diss} hold and let $\widetilde{X}$ be a random variable. Then for all $h \in (0, h_0^{1})$ we have
	 \begin{equation*}
	 \mathbb{E}f_{1}(|Y' - \widetilde{X}|) \leq (1-c_1h)f_{1}(|y-x|) + \mathbb{E}|X' - \widetilde{X}| \,.
	 \end{equation*}
	 \begin{proof}
	 	Note that the function $f_{1}$ defined in (\ref{defFconcave}) is concave (and hence subadditive), increasing and its derivative is such that $f_{1}'(x) \in [e^{-qr_1^1},1]$ for all $x \in \mathbb{R}_{+}$. Thus we get
	 	\begin{equation*}
	 	\begin{split}
	 	\mathbb{E}f_{1}(|Y' - \widetilde{X}|) &\leq \mathbb{E}f_{1}(|Y' - X'| + |X' - \widetilde{X}|) \leq \mathbb{E}f_{1}(|Y' - X'|) + \mathbb{E}f_{1}(|X' - \widetilde{X}|) \\
	 	&\leq (1-c_1h)f_{1}(|y-x|) + \mathbb{E} |X' - \widetilde{X}| \,.
	 	\end{split}
	 	\end{equation*}
	 \end{proof}
\end{theorem}

Even though most results in the present paper are based on the contractivity and the perturbation theorems presented above for Euler schemes $X' = x + hb(x) + \sqrt{h}Z$, it may sometimes be useful to be also able to couple processes with inaccurate drift and obtain respective counterparts of Theorems \ref{mainContractivityTheorem} and \ref{theoremPerturbation}. In the sequel we will need such results only in Section \ref{sectionMLMC} where we treat Multi-level Monte Carlo algorithms for the inaccurate drift case. To this end, consider
\begin{equation}\label{oneStepEulerInaccurate}
\bar{X}' = x + h \bar{b}(x,U) + \sqrt{h}Z
\end{equation} 
and define $\hat{x} = x + h \bar{b}(x,U)$ and $\hat{y} = y + h \bar{b}(y,U)$, where the random variable $U$ is independent of $Z$ and such that $\mathbb{E}\bar{b}(x,U) = b(x)$ for all $x \in \mathbb{R}^d$. Hence, we can still use the prescription (\ref{ourCoupling}) to define a coupling $(\bar{X}', \bar{Y}')$ of two copies of (\ref{oneStepEulerInaccurate}) started from arbitrary points $x$, $y \in \mathbb{R}^d$ by substituting $\hat{x}$ and $\hat{y}$ defined for $b$ that appear in (\ref{ourCoupling}) with our new $\hat{x}$ and $\hat{y}$ defined for $\bar{b}$. In other words, using similar notation as in (\ref{psiCoupling}), we have a new transformation
\begin{equation}\label{couplingForInaccurate}
\bar{Y}' = \bar{\psi}_{m,H}(x,y,U,Z)
\end{equation}
that we can use to couple two copies of the process given by (\ref{oneStepEulerInaccurate}). Since the statement about equality of marginal laws in the coupling given by (\ref{ourCoupling}) is actually a statement about the Gaussian steps $\hat{x} \mapsto X'$ and $\hat{y} \mapsto Y'$ and we assume independence of $U$ and $Z$, in order to prove that the pair $(\bar{X}',\bar{Y}')$ obtained via (\ref{couplingForInaccurate}) is indeed a coupling of two copies of (\ref{oneStepEulerInaccurate}), we can apply the reasoning from Theorem \ref{couplingTheorem} by replacing the expectation there with the conditional expectation with respect to $U$. For the coupling $(\bar{X}',\bar{Y}')$ obtained this way, we have the following result.

\begin{theorem}\label{theoremContractivityForInaccurate}
	Suppose Assumptions \ref{as MLMCina} are satisfied with constants $\bar{K} = K$, $\bar{R} = \mathcal{R}$ and $\bar{L} = L > 0$. Then for the coupling $(\bar{X}',\bar{Y}')$ of two copies of (\ref{oneStepEulerInaccurate}) defined via (\ref{couplingForInaccurate}) we have
	\begin{equation*}
	\mathbb{E}\left[ f(|\bar{X}' - \bar{Y}'|) \, | \, U \right] \leq (1-ch)f(|x-y|) \, \text{ a.s. }
	\end{equation*}
	for all $h \in (0, h_0)$, with the same constants $h_0$, $c$, $m$, $H$ and the same function $f$ as in Theorem \ref{mainContractivityTheorem}. We also have the perturbation inequality (\ref{perturbation}) in the unchanged form, with $(X',Y')$ replaced by $(\bar{X}', \bar{Y}')$.
\end{theorem}

The proof of Theorem \ref{theoremContractivityForInaccurate} will be presented in Section \ref{sectionProofContractivity}. Note that Theorem \ref{theoremContractivityForInaccurate} actually implies Theorems \ref{mainContractivityTheorem} and \ref{theoremPerturbation}, since the Euler scheme (\ref{eq oneStep}) can be easily interpreted as a special case of the scheme (\ref{oneStepEulerInaccurate}) with inaccurate drift. Hence we need to prove only Theorem \ref{theoremContractivityForInaccurate}.

\subsection{Unadjusted Langevin algorithm}\label{sectionULA}

In this section we will explain how to apply Theorem \ref{theoremPerturbation} in order to obtain bounds on the distance between the invariant measure of the solution $(Y_t)_{t \geq 0}$ of the SDE
\begin{equation}\label{ULASDE}
dY_t = b(Y_t)dt + dW_t \,,
\end{equation}
where $(W_t)_{t \geq 0}$ is a $d$-dimensional Brownian motion, and the laws of the Markov chain $(X_k)_{k=0}^{\infty}$ given by
\begin{equation}\label{ULAprototype}
X_{k+1} = X_k + hb(X_k) + \sqrt{h} \xi_{k+1} \,,
\end{equation}
for all $k \geq 0$, where $h > 0$ is fixed and $(\xi)_{k=1}^{\infty}$ are i.i.d. random variables with the standard normal distribution. 

If $b(x) = - \frac{1}{2}\nabla U(x)$ for a function $U \in C^2(\mathbb{R}^d)$, then the equation (\ref{ULASDE}) is called the (overdamped) Langevin SDE. Its invariant measure is given by 
\begin{equation*}
\pi(dz) := \frac{\exp(-U(z)) dz}{\int_{\mathbb{R}^d} \exp(-U(x)) dx} \,,
\end{equation*}
see e.g. \cite{durmus2017nonasymptotic} or Example 1 in \cite{Eberle2016}. The method of asymptotic sampling from $\pi$ by using the Markov chain defined by (\ref{ULAprototype}) is called the Unadjusted Langevin Algorithm.

In order to proceed, let us first observe that
\begin{equation}\label{defYdiscrete}
Y_{(k+1)h} = Y_{kh} + \int_{kh}^{(k+1)h} b(Y_s) ds + (W_{(k+1)h} - W_{kh})
\end{equation}
for all $k \geq 0$. We have the following result.

\begin{theorem}\label{theoremULA}
	Under the assumptions of Theorem \ref{mainContractivityTheorem}, for any random variables $X_0$, $Y_0$, any $k \geq 1$ and for any $h \in (0, h_0 \wedge \frac{K}{4L^2})$ we have
	\begin{equation*}
	W_2(\cL(X_k), \cL(Y_{kh})) \leq \left( A (1-ch)^{k}\mathbb{E}f(|X_0 - Y_0|) \right)^{1/2} + \left( \frac{A C_{ult}}{c}\right)^{1/2} h^{1/4} \,,
	\end{equation*}
	whereas for any $h \in (0, h_0^{1})$ we have
	\begin{equation}\label{eq ULA general W1}
	W_1(\cL(X_k), \cL(Y_{kh})) \leq e^{qr_1^1} (1-c_1h)^k \mathbb{E}f_{1}(|X_0 - Y_0|) + e^{qr_1^1} \frac{\sqrt{C_{dif}}}{c_1} h^{1/2} \,,
	\end{equation}
	where $c$ is given by (\ref{contractivityConstant}), $c_1$ and $h_0^1$ are given by (\ref{contractivityConstantConc}), $A$ is given by (\ref{defA}), $C_{dif}$ is given by (\ref{defCdif}) and $C_{ult}$ is defined in (\ref{defCult}). In particular, if $Y_0 \sim \pi$, then for any $k \geq 1$ and any $h \in (0, h_0 \wedge \frac{K}{4L^2})$ we have
	\begin{equation}\label{eq ULA W2}
	W_2(\cL(X_k), \pi) \leq \left( A (1-ch)^{k}\mathbb{E}f(|X_0 - Y_0|) \right)^{1/2} + \left( \frac{A C_{ult}}{c}\right)^{1/2} h^{1/4} \,,
	\end{equation}
	whereas for any $h \in (0, h_0^{1})$ we have
	\begin{equation}\label{eq ULA W1}
	W_1(\cL(X_k), \pi) \leq e^{qr_1^1} (1-c_1h)^k \mathbb{E}f_{1}(|X_0 - Y_0|) + e^{qr_1^1} \frac{\sqrt{C_{dif}}}{c_1} h^{1/2} \,.
	\end{equation}
\end{theorem}

\begin{remark}\label{remarkWhyWeUseConcaveContractivity}
	Note that in Theorem \ref{theoremULA} we apply our Theorem \ref{mainContractivityTheorem} only to obtain $L^2$ bounds, while $L^1$ bounds are obtained using Theorem \ref{theoremContractivityConcave} taken from \cite{EberleMajka2018}. We could still apply Theorem \ref{mainContractivityTheorem} in the $L^1$ case and obtain
	\begin{equation*}
	W_1(\cL(X_k), \cL(Y_{kh})) \leq e^{ar_2} (1-ch)^k \mathbb{E}f(|X_0 - Y_0|) + e^{ar_2} \frac{C_{ult}}{c} h^{1/2} 
	\end{equation*}
	instead of (\ref{eq ULA general W1}). However, (\ref{eq ULA general W1}) gives us better dependence of the constants on the dimension $d$, since $C_{ult}$ is of order $\mathcal{O}(d)$, whereas $\sqrt{C_{dif}}$ is of order $\mathcal{O}(\sqrt{d})$.
\end{remark}

\begin{proof}[Proof of Theorem \ref{theoremULA}]
We define
\begin{equation*}
\sqrt{h}Z_{k+1} := W_{(k+1)h} - W_{kh} 
\end{equation*}
and consider an auxiliary process $(S_k)_{k=1}^{\infty}$ given by
\begin{equation}\label{defS}
S_{k+1} := Y_{kh} + hb(Y_{kh}) + \sqrt{h}Z_{k+1} \,.
\end{equation}
We can think of $(S_k)_{k=1}^{\infty}$ as a process which at each step $k$ is moved to the position of $(Y_{kh})_{k=1}^{\infty}$ and then evolves in the same way as (\ref{ULAprototype}), with the random noise given by an increment of the Brownian motion $(W_t)_{t \geq 0}$. Now consider a new process $(G_k)_{k=0}^{\infty}$ coupled to $(S_k)_{k=1}^{\infty}$ according to the prescription (\ref{ourCoupling}). Using the notation introduced in (\ref{psiCoupling}), we set
\begin{equation}\label{defG}
G_{k+1} := \psi_{m,H}(Y_{kh},G_k,Z_{k+1}) \,.
\end{equation}
We can alternatively describe $(G_k)_{k=0}^{\infty}$ as
\begin{equation*}
G_{k+1} = G_k + hb(G_k) + \sqrt{h}\widetilde{Z}_{k+1}
\end{equation*}
where $(\widetilde{Z}_{k+1})_{k=0}^{\infty}$ are i.i.d. normally distributed random variables obtained through application of the transformation $\psi$. Note that if $G_0 \sim X_0$, then for all $k \geq 1$ we have $\cL(G_k) = \cL(X_k)$.

 We now want to apply Theorem \ref{theoremPerturbation} with $Y' = G_{k+1}$, $y = G_k$, $X' = S_{k+1}$, $x = Y_{kh}$ and $\widetilde{X} = Y_{(k+1)h}$. Hence for all $k \geq 0$ we obtain
\begin{equation}\label{ULAinequality}
\begin{split}
\mathbb{E}f(|G_{k+1} - Y_{(k+1)h}|) &\leq (1-ch)\mathbb{E}f(|G_k - Y_{kh}|) + \mathbb{E} \left[ \frac{1}{r_2}e^{-ar_2} |S_{k+1} - Y_{(k+1)h}|^2 \right] \\
&+ \mathbb{E} \left[ \left( 1 + \frac{1}{r_2}e^{-ar_2} \left( |G_k - Y_{kh}|(1+h_0L) + \sqrt{h_0} \right) \right) |S_{k+1} - Y_{(k+1)h}| \right] \,.
\end{split}
\end{equation}

We need a few technical lemmas to bound the quantities appearing on the right hand side. All these lemmas work under Assumptions \ref{as diss}, with the exception of Lemma \ref{lemmaEulerMomentUniformEstimate}, where an additional upper bound on $h$ is needed.

\begin{lemma}\label{lemmaContractivityAtInfinityImpliesLyapunov}
	Assumptions \ref{as diss} imply that for all $x \in \mathbb{R}^d$ we have
	\begin{equation}\label{driftDissipativeGrowth}
	\langle b(x), x \rangle \leq M_2 - M_1 |x|^2 \,,
	\end{equation}
	where
	\begin{equation*}
	M_2 := L\left(\max\left(\mathcal{R}, \frac{2|b(0)|}{K}\right)\right)^2 + |b(0)|\max\left(\mathcal{R}, \frac{2|b(0)|}{K}\right) \text{ and } M_1 := \frac{K}{2} \,.
	\end{equation*}
\end{lemma}

\begin{lemma}\label{lemmaSDEMomentUniformBound}
	Let $(Y_t)_{t \geq 0}$ be defined by (\ref{ULASDE}) and let $b$ satisfy Assumptions \ref{as diss}. Then for any $t > 0$ we have
	\begin{equation*}
	\mathbb{E} |Y_t|^2 \leq C_{SDE} \,,
	\end{equation*}
	where
	\begin{equation*}
	C_{SDE} := \mathbb{E} |Y_0|^2 + \frac{2M_2 + d}{2M_1} \,.
	\end{equation*}
\end{lemma}

\begin{lemma}\label{lemmaEulerMomentUniformEstimate}
	Let $(X_k)_{k=0}^{\infty}$ be given by (\ref{ULAprototype}). Let $h < h_0 \wedge \frac{K}{4L^2}$. Then for all $k \geq 1$ we have
	\begin{equation*}
	\mathbb{E}|X_k|^2 \leq C_{Eul} \,,
	\end{equation*}
	where
	\begin{equation}\label{defCEul}
	C_{Eul} := \mathbb{E}|X_0|^2 + \frac{2h_0 |b(0)|^2 + d + M_2}{M_1 - 2h_0 L^2} \,.
	\end{equation}
\end{lemma}

Note that the bound obtained above applies also to the process $(G_k)_{k=0}^{\infty}$ defined by (\ref{defG}), since $\cL(G_k) = \cL(X_k)$ for all $k \geq 0$.

\begin{lemma}\label{lemmaDifferenceEulerSDE}
	Let $(Y_{kh})_{k=0}^{\infty}$ and $(S_k)_{k=0}^{\infty}$ be defined by (\ref{defYdiscrete}) and (\ref{defS}), respectively. Then for any $k \geq 0$ we have
	\begin{equation*}
	\mathbb{E}|S_{k+1} - Y_{(k+1)h}|^2 \leq C_{dif} h^3 \,,
	\end{equation*}
	where
	\begin{equation}\label{defCdif}
	C_{dif} := L^2 \left( \frac{4 h_0}{3} \left( \left( \mathbb{E}|Y_0|^2 + \frac{2M_2 + 1}{M_1}\right)L^2 + |b(0)|^2 \right) + d \right) \,.
	\end{equation}
\end{lemma}

Proofs of all the lemmas can be found in the Appendix.

Combining all the estimates from the lemmas, we can come back to (\ref{ULAinequality}) and we see that
\begin{equation*}
\begin{split}
\mathbb{E}f(|G_{k+1} - Y_{(k+1)h}|) &\leq (1-ch)\mathbb{E}f(|G_k - Y_{kh}|) +  \frac{1}{r_2}e^{-ar_2} C_{dif}h^3 + \sqrt{C_{dif}}h^{3/2}\\
&+ \frac{1}{r_2}e^{-ar_2} (\sqrt{C_{Eul}} + \sqrt{C_{SDE}})(1+h_0L)\sqrt{C_{dif}}h^{3/2} + \frac{1}{r_2}e^{-ar_2}\sqrt{h_0}\sqrt{C_{dif}}h^{3/2} \,.
\end{split}
\end{equation*}
Let us define
\begin{equation}\label{defCult}
C_{ult} := \frac{1}{r_2}e^{-ar_2} C_{dif}h_0^{3/2} + \sqrt{C_{dif}} + \frac{1}{r_2}e^{-ar_2} (\sqrt{C_{Eul}} + \sqrt{C_{SDE}})(1+h_0L)\sqrt{C_{dif}} + \frac{1}{r_2}e^{-ar_2}\sqrt{h_0}\sqrt{C_{dif}} \,.
\end{equation}
Then we have
\begin{equation*}
\mathbb{E}f(|G_{k+1} - Y_{(k+1)h}|) \leq (1-ch)\mathbb{E}f(|G_k - Y_{kh}|) + C_{ult}h^{3/2}
\end{equation*}
and hence
\begin{equation}\label{ULAFinalInequality}
\begin{split}
\mathbb{E}f(|G_{k+1} - Y_{(k+1)h}|) &\leq (1-ch)^{k+1}\mathbb{E}f(|G_0 - Y_0|) + \sum_{j=0}^{k} (1-ch)^j C_{ult}h^{3/2} \\
&\leq (1-ch)^{k+1}\mathbb{E}f(|G_0 - Y_0|) + \frac{C_{ult}}{c}h^{1/2} \,.
\end{split}
\end{equation}
Using Lemma \ref{corollarySquareComparison} we get
\begin{equation*}
\left( \mathbb{E} \left| G_{k+1} - Y_{(k+1)h} \right|^2 \right)^{1/2} \leq \left( A (1-ch)^{k+1}\mathbb{E}f(|G_0 - Y_0|) \right)^{1/2} + \left( \frac{A C_{ult}}{c}\right)^{1/2} h^{1/4} \,.
\end{equation*}
Note that we could also use it to obtain $L^1$ bounds (cf. Remark \ref{remarkWhyWeUseConcaveContractivity}) 
\begin{equation*}
\mathbb{E}\left| G_{k+1} - Y_{(k+1)h} \right| \leq e^{ar_2}(1-ch)^{k+1}\mathbb{E}f(|G_0 - Y_0|) + e^{ar_2}\frac{C_{ult}}{c} h^{1/2} \,.
\end{equation*}
However, instead we will apply Theorems \ref{theoremContractivityConcave} and \ref{theoremPerturbationConcave}. Note that we have
\begin{equation*}
\mathbb{E}f_{1}(|G_{k+1} - Y_{(k+1)h}|) \leq (1-c_1h) \mathbb{E}f_{1}(|G_k - Y_{kh}|) + \mathbb{E}|S_{k+1} - Y_{(k+1)h}| 
\end{equation*}
and $\mathbb{E}|S_{k+1} - Y_{(k+1)h}|  \leq \sqrt{C_{dif}} h^{3/2}$ and hence
\begin{equation*}
\mathbb{E}f_{1}(|G_{k+1} - Y_{(k+1)h}|) \leq (1-c_1h)^{k+1} \mathbb{E}f_{1}(|G_0 - Y_{0}|) + \sum_{j=0}^{k} (1-c_1h)^j \sqrt{C_{dif}} h^{3/2} \,.
\end{equation*}
Using $f_{1}(x) \geq e^{-qr_1^1} x$ for $x \in \mathbb{R}_{+}$ we get
\begin{equation*}
\mathbb{E}|G_{k+1} - Y_{(k+1)h}| \leq e^{qr_1^1} (1-c_1h)^{k+1} \mathbb{E}f_{1}(|G_0 - Y_0|) + e^{qr_1^1} \frac{\sqrt{C_{dif}}}{c_1}h^{1/2} \,,
\end{equation*}
which finishes the proof.
\end{proof}

\begin{proof}[Proof of Theorem \ref{mainTheoremULA}]
	Since our reasoning in the proof of Theorem \ref{theoremULA} applies for any coupling of the initial conditions $X_0$ and $Y_0$, we can take the infimum on the right hand sides of (\ref{eq ULA W2}) and (\ref{eq ULA W1}) and thus we obtain upper bounds with $\mathbb{E}f(|X_0 - Y_0|)$ replaced by $W_f(\cL(X_0),\cL(Y_0))$. Hence both (\ref{eq mainULAW2}) and (\ref{eq mainULAW1}) follow immediately from Theorem \ref{theoremULA} with $C_2 = \sqrt{A}$, $\hat{c}_2 = c$, $\widetilde{c}_2 = \left( \frac{AC_{ult}}{c} \right)^{1/2}$, $C_1 = e^{qr_1^1}$, $\hat{c}_1 = c_1$ and $C_1 = e^{qr_1^1} \frac{\sqrt{C_{dif}}}{c_1}$.
\end{proof}

\subsection{Euler scheme with randomised (inaccurate) drift}\label{sectionStochasticGradient}

Here we consider
\begin{equation*}
\bar{X}_{k+1} = \bar{X}_k + h\bar{b}(\bar{X}_k, U_k) + \sqrt{h}Z_{k+1}
\end{equation*}
for $k \geq 0$ and compare it with the standard Euler scheme
\begin{equation}\label{SGEuler}
X_{k+1} = X_k + hb(X_k) + \sqrt{h} Z_{k+1} \,.
\end{equation}
We also note that since $U_k$ and $\bar{X}_k$ are independent, (\ref{ina estimatorVariance}) implies 
\begin{equation}\label{ina varianceBound}
\bV[ \bar{b}(\bar{X}_k,U_{k}) | \bar{X}_k] \leq \sigma^2 (1 + |\bar{X}_k|^2) h^{\alpha}
\end{equation}
for any $k \geq 0$. Moreover, note that $\mathbb{E} \left[ |\bar{b}(\bar{X}_k, U_k) |^2 | \bar{X}_k \right] = \bV[ \bar{b}(\bar{X}_k,U_{k}) | \bar{X}_k] + |b(\bar{X}_k)|^2$ and hence (\ref{driftLipschitz}) and (\ref{ina varianceBound}) imply
\begin{equation}\label{ina conditionalSecondMomentBound}
\mathbb{E} \left[ |\bar{b}(\bar{X}_k, U_k) |^2 | \bar{X}_k \right] \leq \sigma^2 (|\bar{X}_k|^2 + 1) h^{\alpha} + 2L^2|\bar{X}_k|^2 + 2|b(0)|^2 
\end{equation}
for any $k \geq 0$. These estimates will be used frequently in the sequel.

 Now let us analyse an example that appears often in the statistics literature, see e.g. \cite{welling2011bayesian, Shamir2016} and the references therein, and explain how Assumption \ref{as ina} can be verified.

\begin{example} \label{ex subsampling}
	Let $(\theta_i)_{i=1,\ldots,m}$ and $\theta_i\in \bR^d$, for all $i$. Moreover, let $U = (U_i)_{i=1\ldots,s}$ be a collection of $s$ independent random variables, uniformly distributed over the set $\{1,\ldots,m\}$.
	We define 
	\begin{equation}\label{subsamplingDrift}
	b(x)= \sum_{i=1}^m \hat{b}(x,\theta_i)\,\,\,\,\text{and}\,\,\,\,\bar{b}(x,U_k)= \frac{m}{s} \sum_{i=1}^s \hat{b}(x,\theta_{U_i})\,.
	\end{equation}
	In applications of Bayesian inference $m$ corresponds to the size of the data set and may be large. Consequently, the generation of $(\bar X_k)_{k=0}^{\infty}$ is costly. One then hopes that a randomisation strategy will reduce the computational cost without introducing significant variance. Notice that
	\begin{align*}
	\bE [\bar{b}(x,U)] = &  \frac{m}{s} \sum_{i=1}^s \bE[  \hat{b}(x,\theta_{U_i}) ]
	= \frac{m}{s} \sum_{i=1}^s \sum_{j=1}^m \hat{b}(x,\theta_{j})\bP(U_i=j) \\
	&  = \frac{1}{s} \sum_{i=1}^s \sum_{j=1}^m \hat{b}(x,\theta_{j})= \sum_{i=1}^m \hat{b}(x,\theta_{i}) = b(x)\,.
	\end{align*}
	On the other hand, we have
	\begin{equation*}
	\begin{split}
	\mathbb{E}|\bar{b}(x,U) - b(x)|^2 &= \mathbb{E} \left|\frac{m}{s}\sum_{i=1}^{s} \hat{b}(x,\theta_{U_i}) - \sum_{i=1}^{m} \hat{b}(x,\theta_i)\right|^2 = \mathbb{E} \left|\frac{1}{s}\sum_{i=1}^{s} \left( m \hat{b}(x,\theta_{U_i}) - \sum_{i=1}^{m} \hat{b}(x,\theta_i) \right) \right|^2 \\
	&= \frac{1}{s^2} \sum_{i=1}^{s} \mathbb{E} \left| m \hat{b}(x,\theta_{U_i}) - \sum_{i=1}^{m} \hat{b}(x,\theta_i) \right|^2 = \frac{1}{s^2} \sum_{i=1}^{s} \sum_{j=1}^{m} \left( m \hat{b}(x, \theta_j) - b(x) \right)^2 \frac{1}{m} \\
	&= \frac{1}{s^2} \left( sm \sum_{j=1}^{m}\hat{b}(x,\theta_j)^2 - sb(x)^2 \right) \,,
	\end{split}
	\end{equation*}
	where we used the fact that $m \hat{b}(x,\theta_{U_i}) - \sum_{i=1}^{m} \hat{b}(x,\theta_i)$ are centered, independent random variables. This implies that if for all $\theta$ and $x$ we have $|\hat{b}(x,\theta)|^2 \leq C(1 + |x|^2)$ with some constant $C > 0$, then condition (\ref{ina estimatorVariance}) can indeed be satisfied.
	Then we have
	\begin{equation*}
	\mathbb{E}|\bar{b}(x,U) - b(x)|^2 \leq \frac{m^2}{s} C(1 + |x|^2)
	\end{equation*}
	and hence in order for (\ref{ina estimatorVariance}) to be satisfied, we need to have $m^2C/s \leq \sigma^2 h^{\alpha}$ for some constants $\sigma$, $\alpha > 0$, which means that we need to choose $s$ and $h$ so that $s^{-1} \lesssim h^{\alpha}$. However, this can give us a constant $\sigma^2$ of order $m^2$, which can be very large in applications. In order to reduce the variance, a sensible choice seems to be to consider subsampling without replacement, see e.g. \cite{Shamir2016}. More precisely, we define an estimator
	\begin{equation*}
	\bar{b}^{wor}(x,U) := \frac{m}{s} \sum_{i=j}^{m} \hat{b}(x,\theta_j)Z_j \,,
	\end{equation*} 
	where $(Z_j)_{j=1}^{m}$ are correlated random variables such that
	$\mathbb{P}(Z_j = 1) = \frac{s}{m}$, $\mathbb{P}(Z_j = 0) = 1 - \frac{s}{m}$ and $\mathbb{P}(Z_i = 1, Z_j = 1) = \binom{m-2}{s-2}/\binom{m}{s}$ 
	for any $i$, $j \in \{ 1, \ldots , m \}$ such that $i \neq j$. Note that this definition of $\bar{b}^{wor}$ corresponds to sampling $s$ terms from the sum defining $b$ without replacement, see e.g. Lemma B in Section 7.3.1 and Problem 7.26 in \cite{rice1988}. It is immediate to check that this estimator of $b$ is indeed unbiased. As for the variance, we have
	\begin{equation*}
	\mathbb{V}\left( \frac{m}{s} \sum_{i=j}^{m} \hat{b}(x,\theta_j)Z_j \right) = \frac{m^2}{s^2} \left( \sum_{j=1}^{m}\hat{b}(x,\theta_j)^2 \mathbb{V}(Z_j) + \sum_{j=1}^{m} \sum_{i \neq j} \hat{b}(x,\theta_j) \hat{b}(x, \theta_i) \operatorname{Cov}(Z_i,Z_j) \right) \,.
	\end{equation*}
	Moreover, it is easy to check that $\operatorname{Cov}(Z_i,Z_j) = \frac{s(s-1)}{m(m-1)} - \frac{s^2}{m^2} = - \frac{s(1-\frac{s}{m})}{m(m-1)}$ and hence
	\begin{equation*}
	\begin{split}
	\mathbb{V}(\bar{b}^{wor}(x,U)) &= \frac{m^2}{s^2} \left( \sum_{j=1}^{m}\hat{b}(x,\theta_j)^2 \frac{s}{m}(1-\frac{s}{m}) - \sum_{j=1}^{m} \sum_{i \neq j} \hat{b}(x,\theta_j) \hat{b}(x, \theta_i) \frac{s(1-\frac{s}{m})}{m(m-1)} \right) \\
	&= \frac{m}{s}(1- \frac{s}{m}) \frac{m}{m-1} \sum_{j=1}^{m} \left( \hat{b}(x,\theta_j) - \frac{b(x)}{m}\right)^2 \,,
	\end{split}
	\end{equation*}
	where the last equality comes from the easily verifiable identity
	\begin{equation*}
	\sum_{j=1}^{m} \left( \hat{b}(x,\theta_j) - \frac{b(x)}{m}\right)^2 = \frac{m-1}{m} \left( \sum_{j=1}^{m} \hat{b}(x,\theta_j)^2 - \frac{1}{m-1} \sum_{j=1}^{m} \sum_{i \neq j} \hat{b}(x,\theta_j) \hat{b}(x,\theta_i) \right) \,.
	\end{equation*}
	Hence we see that the variance of the estimator $\bar{b}^{wor}$ is equal to the variance of $\bar{b}$ multiplied by $(1 - \frac{s}{m})$. Thus, assuming again that for all $\theta$ and $x$ we have $|\hat{b}(x,\theta)|^2 \leq C(1 + |x|^2)$, we see that we now need to have
	\begin{equation*}
	\frac{m^2}{s}(1 - \frac{s}{m})C \leq \sigma^2 h^{\alpha}
	\end{equation*}
	in order for (\ref{ina estimatorVariance}) to hold. Since the left hand side goes to zero when $s$ approaches $m$, this method allows us to choose a much smaller $\sigma$ than in the subsampling with replacement case, if we choose $s$ large enough. 
	
	Another possible way of reducing the variance $\sigma$ is via an appropriate rescaling (time-change) of the SDE (\ref{eq sde}). Namely, it is well-known that for any positive definite symmetric matrix $\Sigma$, the SDE $dY_t = \Sigma b(Y_t) dt + \Sigma^{1/2} dW_t$ has the same invariant measure as (\ref{eq sde}). See e.g.\ \cite{DurmusRobertsVilmartZygalakis2017, Xifara2014} and the references therein for discussions on different choices of $\Sigma$ in Monte Carlo algorithms. Hence, instead of considering the drift $b$ and its estimator $\bar{b}$ given by (\ref{subsamplingDrift}), we can take $b(x) = \frac{1}{m} \sum_{i=1}^m  \hat{b}(x,\theta_i)$ and $\bar{b}(x,U_k)= \frac{1}{s} \sum_{i=1}^s \hat{b}(x,\theta_{U_i})$ and we can consider a Markov chain
	\begin{equation*}
	\bar{X}_{k+1} = \bar{X}_k + h \frac{1}{s} \sum_{i=1}^s \hat{b}(\bar{X}_k,\theta_{U_i}) + \sqrt{h/m} Z_{k+1} \,.
	\end{equation*}
	This corresponds to choosing $\Sigma = (1/m)I$ in the SDE. Intuitively, for $m > 1$ this choice reduces the variance of the algorithm at the cost of slowing down convergence of the Markov chain. In other words, in our bounds in Theorem \ref{mainStochasticGradientTheorem} the term $\bar{C}_2 h^{\alpha/4}$ becomes smaller, whereas $W_2(\cL(X_k), \pi)$ becomes larger (since the contractivity constant $\hat{c}_2$ in Theorem \ref{mainTheoremULA} becomes smaller). A precise analysis of this trade-off falls beyond the scope of the present paper and is left for future work. 
\end{example}

By applying Theorem \ref{theoremPerturbation}, we can prove the following result.

\begin{theorem}\label{theoremSGcomparison}
	Let the assumptions of Theorem \ref{mainContractivityTheorem} and Assumption \ref{as ina} hold. Let $h \in (0, h_0 \wedge \frac{K}{4L^2 + 2\sigma^2} \wedge 1)$. Then for any random variables $X_0$, $\bar{X}_0$ and any $k \geq 1$ we have 
	\begin{equation}\label{SGcomparisonW2}
	W_2(\cL(X_k), \cL(\bar{X}_k)) \leq \left( A (1-ch)^{k}\mathbb{E}f(|X_0 - \bar{X}_0|) \right)^{1/2} + \left( \frac{A C_{Iult}}{c}\right)^{1/2} h^{\alpha/4}
	\end{equation}
	and
	\begin{equation}\label{SGcomparisonW1}
	W_1(\cL(X_k), \cL(\bar{X}_k)) \leq e^{ar_2} (1-ch)^k \mathbb{E}f(|X_0 - \bar{X}_0|) + e^{ar_2} \frac{C_{Iult}}{c} h^{\alpha/2} \,,
	\end{equation}
	where $c$ is given by (\ref{contractivityConstant}), $A$ is given by (\ref{defA}), $C_{Iult}$ is defined in (\ref{defCIult}) and $\alpha$ is specified in condition (\ref{ina estimatorVariance}).
\end{theorem}

\begin{proof}
We will need to use an auxiliary chain $(S_k)_{k=0}^{\infty}$ that at each step $k$ is moved to the position of $\bar{X}_k$ and then evolves as (\ref{SGEuler}), i.e.,
\begin{equation*}
S_{k+1} := \bar{X}_k + hb(\bar{X}_k) + \sqrt{h}Z_{k+1} \,.
\end{equation*}
We define a new process $(G_k)_{k=0}^{\infty}$ coupled to $(S_k)_{k=0}^{\infty}$ according to the prescription (\ref{ourCoupling}). Using the notation introduced in (\ref{psiCoupling}), we set
\begin{equation}\label{defSGG}
G_{k+1} := \psi_{m,H}(\bar{X}_k,G_k,Z_{k+1}) \,.
\end{equation}
In other words, we have
$G_{k+1} = G_k + hb(G_k) + \sqrt{h}\widetilde{Z}_{k+1}$
for some i.i.d. normal random variables $(\widetilde{Z}_k)_{k=1}^{\infty}$ that are determined via the coupling. Note that we have $\cL(G_k) = \cL(X_k)$ for all $k \geq 1$ if $G_0 = X_0$.

We apply now Theorem \ref{theoremPerturbation} and we have
\begin{equation}\label{SGinequality}
\begin{split}
\mathbb{E}f(|G_{k+1} - \bar{X}_{k+1}|) &\leq (1-ch)\mathbb{E}f(|G_k - \bar{X}_k|) + \mathbb{E} \left[ \frac{1}{r_2}e^{-ar_2} |S_{k+1} - \bar{X}_{k+1}|^2 \right] \\
&+ \mathbb{E} \left[ \left( 1 + \frac{1}{r_2}e^{-ar_2} \left( |G_k - \bar{X}_k|(1+h_0L) + \sqrt{h_0} \right) \right) |S_{k+1} - \bar{X}_{k+1}| \right] \,.
\end{split}
\end{equation}
Note that by (\ref{ina estimatorVariance}) we have $\mathbb{E} |\bar{b}(\bar{X}_k,U_k) - b(\bar{X}_k)|^2 \leq \sigma^2 (1 + \mathbb{E}|\bar{X}_k|^2) h^{\alpha}$ and hence, using Jensen's inequality,
\begin{equation*}
\mathbb{E} |S_{k+1} - \bar{X}_{k+1}| = h \mathbb{E} |b(\bar{X}_k) - b(\bar{X}_k, U_k)| \leq \sigma h^{1 + \alpha/2} (1 + \mathbb{E}|\bar{X}_k|^2)^{1/2} \,.
\end{equation*}
Moreover, we have uniform bounds on the moments of $(G_k)_{k=0}^{\infty}$ due to Lemma \ref{lemmaEulerMomentUniformEstimate}. Hence we only need to control the moments of $(\bar{X}_k)_{k=0}^{\infty}$. To this end, we can repeat the reasoning from the proof of Lemma \ref{lemmaEulerMomentUniformEstimate} to obtain the following result.

\begin{lemma}\label{lemmaEulerMomentUniformEstimateInacurate} 
	Let $(\bar{X}_k)_{k=0}^{\infty}$ be given by (\ref{eq ineuler}). Let $h < h_0 \wedge \frac{K}{4L^2 + 2\sigma^2} \wedge 1$. Then for all $k \geq 1$, under Assumptions \ref{as diss} and \ref{as ina}, we have
	\begin{equation*}
	\mathbb{E}|\bar{X}_k|^2 \leq C_{IEul} \,,
	\end{equation*}
	where
	\begin{equation}\label{defCIEul}
	C_{IEul} := \mathbb{E}|X_0|^2 + \frac{2h_0 |b(0)|^2 + d + M_2 + h_0 \sigma^2}{M_1 - 2h_0 L^2 - h_0 \sigma^2} \,.
	\end{equation}
\end{lemma}

The proof can be found in the Appendix. Now we come back to (\ref{SGinequality}) and we have
\begin{equation*}
\begin{split}
\mathbb{E}f(|G_{k+1} &- \bar{X}_{k+1}|) \leq (1-ch)\mathbb{E} f(|G_k - \bar{X}_k|) + \frac{1}{r_2}e^{-ar_2} \sigma^2 h^{1 + \alpha}(1 + C_{IEul})\\
&+ \mathbb{E} \left( 1 + \frac{1}{r_2}e^{-ar_2} \left( (\sqrt{C_{Eul}} + \sqrt{C_{IEul}})(1+h_0L) + \sqrt{h_0} \right) \right) \sigma h^{1 + \alpha/2} (1 + C_{IEul})^{1/2}\,.
\end{split}
\end{equation*}
Hence if we define
\begin{equation}\label{defCIult}
\begin{split}
C_{Iult} &:= \frac{1}{r_2}e^{-ar_2} \sigma^2 h_0^{\alpha/2} (1 + C_{IEul}) \\
&+ \left( \sigma + \sigma \frac{1}{r_2}e^{-ar_2}(\sqrt{C_{Eul}} + \sqrt{C_{IEul}})(1+h_0L) + \sigma \frac{1}{r_2}e^{-ar_2} \sqrt{h_0} \right) (1 + C_{IEul})^{1/2} \,,
\end{split}
\end{equation}
we get
\begin{equation*}
\mathbb{E}f(|G_{k+1} - \bar{X}_{k+1}|) \leq (1-ch)\mathbb{E} f(|G_k - \bar{X}_k|) + C_{Iult} h^{1+ \alpha/2}
\end{equation*}
and we can finish the proof as in the previous section, obtaining
\begin{equation}\label{SGfinalInequality}
\mathbb{E}f(|G_{k+1} - \bar{X}_{k+1}|) \leq (1-ch)^{k+1}\mathbb{E} f(|G_0 - \bar{X}_0|) + \frac{C_{Iult}}{c} h^{\alpha/2} \,.
\end{equation}
From this (\ref{SGcomparisonW2}) and (\ref{SGcomparisonW1}) follow easily due to Lemma \ref{corollarySquareComparison}.
\end{proof}

\begin{proof}[Proof of Theorem \ref{mainStochasticGradientTheorem}]
	Since $W_2(\cL(\bar{X}_k), \pi) \leq W_2(\cL(\bar{X}_k), \cL(X_k)) + W_2(\cL(X_k), \pi)$ and an analogous inequality holds for the $W_1$ distance, it is easy to see that Theorem \ref{theoremSGcomparison} with $\bar{X}_0 = X_0$ implies (\ref{mainSGW2}) and (\ref{mainSGW1}) with
	\begin{equation*}
	\bar{C}_2 = \left( \frac{A C_{Iult}}{c}\right)^{1/2}\text{ and } \bar{C}_1 = e^{ar_2} \frac{C_{Iult}}{c} \,,
	\end{equation*}
	respectively.
\end{proof}

\begin{remark}\label{remarkPthMoments}
	Note that the bound from Lemma \ref{lemmaEulerMomentUniformEstimateInacurate} can be generalised to hold for all $p$-th moments of $\bar{X}_k$ for $p \geq 1$. More precisely, for any $p \geq 1$ we can prove that under conditions (\ref{eq th weak dissipativity}) and (\ref{driftEstimatorGrowth})
, there exists a constant $C_{IEul}^{(p)} > 0$ such that for sufficiently small $h$ and for all $k \geq 1$ we have
\begin{equation}\label{pthMomentBounds}
\mathbb{E}|\bar{X}_k|^p \leq C_{IEul}^{(p)} \,.
\end{equation}
In order to see this, we first need to show an analogous bound for the Euler scheme (\ref{ULAprototype}) with accurate drift as in Lemma \ref{lemmaEulerMomentUniformEstimate}, under (\ref{eq th weak dissipativity}). This follows e.g.\ from Theorem 2.1 and Remark 2.4 in \cite{SzpruchZhang2018}. Indeed, note that (\ref{eq th weak dissipativity}) implies that for any $p \geq 2$ we have $\langle x , b(x) \rangle |x|^{p-2} \leq M_2|x|^{p-2} - M_1 |x|^p$ and hence for the generator $L$ of the solution to the SDE (\ref{eq sde}) for the function $V(x) = |x|^p$ we can show
\begin{equation*}
\begin{split}
LV(x) &= \langle \nabla V(x) , b(x) \rangle + \frac{1}{2} \Delta V(x) = p |x|^{p-2} \langle x , b(x) \rangle + \frac{1}{2} p (p-1) |x|^{p-2} \\
&\leq p M_2 |x|^{p-2} - pM_1 |x|^p  + \frac{1}{2}p (p-1) |x|^{p-2} \\
&\leq - \left( pM_1 - \varepsilon p M_2 - \frac{1}{2} \varepsilon p (p-1) \right) |x|^p + \left( pM_2 + \frac{1}{2} p (p-1) \right) C(\varepsilon, p) \,,
\end{split}
\end{equation*}
where $\varepsilon > 0$ can be arbitrary and $C(\varepsilon, p) > 0$ is such that $|x|^{p-2} \leq \varepsilon |x|^p + C(\varepsilon, p)$ for all $x \in \mathbb{R}^d$ (i.e., we can choose $C(\varepsilon, p) > \left( \frac{p-2}{\varepsilon p} \right)^{\frac{p-2}{2}} - \varepsilon \left( \frac{p-2}{\varepsilon p}\right)^{\frac{p}{2}}$). Hence, if we choose $\varepsilon$ such that $\varepsilon < M_1 / (M_2 + \frac{1}{2}(p-1))$, we obtain a Lyapunov condition $LV(x) \leq - \rho V(x) + C$ with some positive constants $\rho$, $C > 0$. Observe that Theorem 2.1 and Remark 2.4 in \cite{SzpruchZhang2018} imply that under a stronger condition, namely, $LV(x) \leq - \rho (1 + V(x))$ for all $x \in \mathbb{R}^d$, (cf.\ (2.8) in \cite{SzpruchZhang2018}), we get (\ref{pthMomentBounds}) for Euler schemes with accurate drifts. However, by analysing the proof in \cite{SzpruchZhang2018} it is easy to see that their condition (2.8) can be replaced by the weaker condition $LV(x) \leq - \rho V(x) + C$. Namely, in the last calculation in the proof of Theorem 2.1 in \cite{SzpruchZhang2018}, the inequality $\mathbb{E}(1 + V(\bar{X}_{k+1})) \leq (1 + (-\rho + \tilde{\rho})h) \mathbb{E}(1+V(\bar{X}_k))$ will be then replaced by $\mathbb{E}(1 + V(\bar{X}_{k+1})) \leq (1 + (-\rho + \tilde{\rho})h) \mathbb{E}(1+V(\bar{X}_k)) + (C+\rho)h$. Then by iterating we obtain $\mathbb{E}(1+ V(\bar{X}_{k+1})) \leq e^{(-\rho + \tilde{\rho})(k+1)h} \mathbb{E}(1+V(X_0)) + C+\rho$ and the desired bound for the accurate drift case follows. Now, using (\ref{driftEstimatorGrowth}) and following the argument from the proof of Lemma \ref{lemmaEulerMomentUniformEstimateInacurate}, it is possible to extend this result to Euler schemes with inaccurate drifts. We leave the details to the reader.
\end{remark}

\subsection{Multi-level Monte Carlo}\label{sectionMLMC}

In this section we focus on Euler schemes with inaccurate drifts. However, the reader who is interested only in MLMC in the accurate drift case, can easily recover relevant results by replacing $(\bar{X}_k)_{k=0}^{\infty}$ defined in (\ref{eq ineuler}) with $(X_k)_{k=0}^{\infty}$ defined in (\ref{eq euler}) and Assumptions \ref{as MLMCina} with Assumptions \ref{as diss}. Note that in such a case certain quantities featured below simply vanish, which only makes our calculations easier.

Let us start by briefly explaining the motivation behind considering Monte Carlo estimators of the multi-level type. A typical strategy for approximating $\int_{\bR^d} g(x) \pi(dx)$, is to resort to the standard Monte Carlo estimator where the average is taken  "over the space". More precisely, we fix the time $T = kh$ for some $k \geq 1$, we generate $N$ i.i.d. samples $(\bar{X}_k^i)_{i=1}^{N}$ of $\bar{X}_k$ defined in \eqref{eq ineuler} and compute the Monte Carlo average
\begin{equation} \label{eq mc} \tag{MCA}
\cA^{MCA}(T,h,N)(g):=\frac{1}{N} \sum_{i=1}^N g(\bar{X}^{i}_k)  \,.
\end{equation}
The aim is to find the optimal allocation of the parameters (terminal time $T$, number of MC samples $N$ and the size of the time-step $h$) to achieve required mean-square-error. We can compute 
\begin{align*}
mse(\cA^{MCA}(T,h,N)(g)) :=  &  \left[ \mathbb{E} \left( \int_{\bR^d} g(x) \pi(dx) - \frac{1}{N} \sum_{i=1}^N g(\bar{X}^{i}_{k})  \right)^2 \right]^{1/2}	\\
\leq &    \left[ \left( \int_{\bR^d} g(x) (\pi(dx) - \mu_{kh}(dx))  \right)^2 \right] ^{1/2}   	
+  \left[ \left( \bE[ g(Y_{kh}) ]  - \bE[g( \bar X_{k})]  \right)^2 \right] ^{1/2}    \\ 
& +  \left[ \bE \left( \bE[ g(\bar{X}_k)]   - \frac{1}{N} \sum_{i=1}^N g(\bar{X}^{i}_k)  \right)^2 \right] ^{1/2}  \,,
\end{align*}
where $\mu_{kh} := \cL(Y_{kh})$. The three error terms are: bias (due to finite time simulation) that we can estimate due to \cite{Eberle2016} with explicit constants if $g$ is Lipschitz, i.e., $W_1(\cL{(Y_{kh})}, \pi   ) \leq e^{- \lambda kh }W_1(\cL(Y_0), \pi  )$ or due to \cite{LuoWang} if $g$ has polynomial growth; weak time discretisation error studied in Theorem \ref{th weak} from which we know that $|\bE[ g(Y_{kh}) ]  - \bE[g( \bar X_{k})]| \lesssim  h$,  and the variance of the Monte Carlo estimator that we also control uniformly in time due to Lemma \ref{lemmaEulerMomentUniformEstimateInacurate}. Hence we have
$
\cA^{MCA}(T,h,N) \lesssim e^{{-} \lambda T} + h + 1/\sqrt{N}.
$
We fix $\epsilon >0$ and set 
$
\cA^{MCA}(T,h,N) \lesssim \epsilon.
$
This leads to the following choice of the parameters 
$
T \approx \lambda^{-1} \log{\epsilon}, \quad h \approx {\epsilon}, \quad N \approx \epsilon^{-2}. 
$
The computational cost is then given by
$
cost(\cA^{MCA}(t,h,N)) = T h^{-1} N \approx (\log{\epsilon}) \epsilon^{-3}. 
$ 
The above cost should be compared with $\epsilon^{-2}$ that holds for the MC estimator in the case when we have access to unbiased samples. 

The recently developed MLMC approach, \cite{giles2008multilevel,MR2187308,heinrich2001multilevel}, allows us to reduce the computational cost of \eqref{eq mc}. The idea is to introduce a family of Euler discretisations with varying time-steps. Fix $L>0$. 
For $\ell \in \{ 1, \ldots , L \}$ let us define $h^{\ell} := \tfrac{1}{M^{\ell}}$. In our analysis we consider $M=2$, for simplicity. We define  
\begin{equation}\label{eq lceuler}
\bar X^{\ell}_{(k+1)h^{\ell}} = \bar X^{\ell}_{k h^{\ell}}  +  \bar{b}(\bar X^{\ell}_{k h^{\ell}},U^{\ell}_{k h^{\ell}})h^{\ell} 
+ Z^{\ell}_{k+1} \,,
\end{equation}
where $Z^{\ell}_{k+1} := W_{(k+1)h^{\ell}} - W_{k h^{\ell}}$ and $(U^{\ell}_{k h^{\ell}})_{k=0}^{\infty}$ are mutually independent and such that for any $x \in \mathbb{R}^d$ we have $\mathbb{E}\bar{b}(x, U^{\ell}_{k h^{\ell}}) = b(x)$. For all $l \in \{ 1, \ldots , L \}$ and all $T$ such that $T=k h^{\ell}$ for some $k \geq 1$, we introduce appropriate modifications of \eqref{eq lceuler} denoted by $(\bar{X}^{f,\ell}_T,\bar{X}^{c,\ell}_T)$ such that
$\cL(\bar X^{f,\ell}_{T}) = \cL(\bar X^{c,\ell}_{T}) = \cL (\bar X^{\ell}_{T})$ for $\ell \in \{ 1,\ldots L \}$. Hence we have
\begin{equation} \label{eq telescoping}
\E[g(\bar X^{L}_{T})]	= \E[g(\bar X^0_T)] + \sum_{\ell=1}^{L} \E[ g(\bar X^{f,\ell}_{T}) - g(\bar X^{c,\ell- 1 }_{T}) ] \,.
\end{equation}
This identity leads to the following unbiased estimator of $\E[g(\bar X_T^{L})]$, for $T = k h^L$ with some $k \geq 1$,
\begin{align*} 
\cA^{MLMC}(T,L,(N_{\ell})_{\ell})(g):= \frac{1}{N_0} \sum_{i=1}^{N_0} g( \bar X_T^{i,0}) + \sum_{\ell=1}^{L}\left\{ \frac{1}{N_\ell} \sum_{i=1}^{N_\ell} 
(    g(\bar X_T^{i,f,\ell}) -  g(\bar X_T^{i,c,\ell-1}) )  \right\}\,,	
\end{align*}
where $\bar X_T^{i,f,\ell}$ and $\bar X_T^{i,c,\ell}$ for $i \in \{ 1, \ldots N_{\ell} \}$ are i.i.d. samples of $\bar X_T^{f,\ell}$ and $\bar X_T^{c,\ell}$, respectively. We assume that the samples across the levels (each summand in $\ell$) are independent, hence
\begin{equation}\label{eq varianceDecompositionForMLMC}
\mathrm{Var}(\cA^{MLMC}(T,L,(N_{\ell})_{\ell})(g))  = \frac{1}{N_0} \mathrm{Var}( g( \bar X_T^{1,0}) )
+ \sum_{\ell=1}^{L} \frac{1}{N_\ell} \mathrm{Var}( 
(    g(\bar X_{T}^{i,f,\ell}) -  g(\bar X_{T}^{i,c,\ell-1}) )  \,.
\end{equation}
If the samples at each summand in $\ell$, i.e., $(\bar{X}_T^{i,f,\ell},\bar{X}_T^{i,c,\ell-1})$ are appropriately coupled for each $\ell$, then $\mathrm{Var}(   f(X_T^{i,f,\ell}) -  f(X_T^{i,c,\ell-1})$ decays when $\ell$ increases. As a result MLMC, with optimally selected parameters, combines many simulations on low accuracy grids (at a corresponding low cost), with relatively few simulations computed with high accuracy and high cost on very fine grids. One has the following error decomposition 
\begin{align*}
mse(\cA^{MLMC}(T,L,&(N_{\ell})_{\ell})(g)) =  \left[ \mathbb{E} \left( \int_{\bR^d} g(x) \pi(dx) - \cA^{MLMC}(T,L,(N_{\ell})_{\ell})(g)  \right)^2 \right]^{1/2}	\\
&\leq    \left[ \left( \int_{\bR^d} g(x) (\pi(dx) - \mu_{T}(dx))  \right)^2 \right] ^{1/2}   	
+  \left[ \left( \bE[ g(Y_{T}) ]  - \bE[g(\bar X^{1,L}_{T})]  \right)^2 \right] ^{1/2}     \\
& +  \left[ \mathbb{E} \left( \bE[ g(\bar X^{1,L}_t)]   - \cA^{MLMC}(T,L,(N_{\ell})_{\ell})(g)  \right)^2 \right] ^{1/2}  \,.	
\end{align*}
Note that the first two errors are the same as for the standard MC and the only difference will come from the variance of the MLCM estimator. In order to evoke to the classical Multilevel Monte Carlo complexity analysis, note that we applied telescopic sum estimator to $\bE[ g(Y_{T}) ]$ rather than directly to $\int_{\bR^d} g(x) \pi(dx)$. In view of the analysis for the standard Monte Carlo we choose $T \approx \lambda^{-1} \log{\epsilon}$ and hence we need to multiply the final cost of MLMC by the factor $\log{\epsilon}$. 

Observe that in analysis of an MLMC estimator the crucial part is to investigate the behaviour of the pair of processes $(\bar{X}_{kh^{\ell}}^{f,\ell},\bar{X}_{kh^{\ell}}^{c,\ell-1})_{k=0}^{\infty}$ for any fixed $\ell \in \{ 1, \ldots , L \}$. Hence, to streamline the notation, from now on we drop the superscript $\ell$ and we will work with
\begin{equation}\label{defbarXf}
\bar{X}^f_{(k+1)h} = \bar{X}^f_{kh} + \bar{b}(\bar{X}^f_{kh}, U^f_{kh})h + \sqrt{h} Z_{k+1} \,, \quad
\bar{X}^f_{(k+2)h} = \bar{X}^f_{(k+1)h} + \bar{b}(\bar{X}^f_{(k+1)h}, U^f_{(k+1)h})h + \sqrt{h} Z_{k+2}
\end{equation}
and 
\begin{equation}\label{defbarXc}
\bar{X}^c_{(k+2)h} = \bar{X}^c_{kh} + \bar{b}(\bar{X}^c_{kh}, U^c_{kh})2h + \sqrt{2h} \hat{Z}_{k+2}\,, 
\end{equation}
where $\hat{Z}_{k+2} := (Z_{k+1} + Z_{k+2})/\sqrt{2}$. Thus $(\bar{X}_{kh}^f)_{k=0}^{\infty}$ is a process on a fine grid and $(\bar{X}_{kh}^c)_{k=0}^{\infty}$ is a process on a coarse grid that moves twice less frequently than $(\bar{X}_{kh}^f)_{k=0}^{\infty}$.

For the condition with the telescopic sum (\ref{eq telescoping}) to hold, it is required that for all $k\in \bN$, $\cL(U_{kh}^f)=\cL(U_{kh}^c)$. Taking as an example subsampling considered in Example \ref{ex subsampling}, we see that this condition forces us to take the same number of samples at each step of the algorithm. We also assume that $(U_{kh}^f)_{k=0}^{\infty}$ are mutually independent so that (\ref{eq varianceDecompositionForMLMC}) holds. The random variables $(U_{kh}^c)_{k=0}^{\infty}$ can be chosen as indepedent of $(U_{kh}^f)_{k=0}^{\infty}$, although coupling them in an appropriate way can help to further reduce the variance (see Remark \ref{remarkChoiceOfU}). 

We impose the following assumptions.
\begin{assumption}\label{as ina MLMC Ucoupling}
	We assume that 
	\begin{itemize}
		\item Random variables $(U_{kh}^f)_{k\in\bN}$ and $(U_{kh}^c)_{k\in\bN}$ are such that for all $k\in \bN$, $\cL(U_{kh}^f)=\cL(U_{kh}^c)$.
		\item There are constants $L_u > 0$ and $\alpha_c > 0$ such that for all $x \in \mathbb{R}^d$ we have
		\begin{equation}\label{ina MLMC UcouplingCondition}
		\bE|\bar{b}(x,U_{kh}^f)+\bar{b}(x,U_{(k+1)h}^f) - 2\bar{b}(x,U_{kh}^c)|^2 \leq L_u(1+|x|^2)h^{\alpha_c} \,,\,\,\,\, \forall k\geq 0\,. 
		\end{equation}
	\end{itemize}
\end{assumption}

We have the following result.

\begin{theorem}\label{theoremInaccurateMLMCVarianceBound}
	Suppose Assumptions \ref{as ina}, \ref{as MLMCina} with $\bar{L} = L$, $\bar{K} = K$, $\bar{R} = \mathcal{R}$, and \ref{as ina MLMC Ucoupling} are satisfied. Then there exists a process $(\bar{G}_k)_{k=0}^{\infty}$ such that $\cL(\bar{G}_k) = \cL(\bar{X}^f_{kh})$ for all $k \geq 1$ if $\bar{G}_0 = \bar{X}^f_0$ and we have
	\begin{equation*}
	\mathbb{E}|\bar{G}_k - \bar{X}^c_{kh}|^2 \leq A(1-ch)^k \mathbb{E}f(|\bar{G}_{0} - \bar{X}^c_{0}|) + \frac{A \cdot C_{IMLult}}{c} h^{\min(\alpha_c, 1)/2}
	\end{equation*}
	for all $h \in (0, h_0 \wedge \frac{K}{4L^2 + 2\sigma^2} \wedge 1)$, where the constant $C_{IMLult}$ is defined in (\ref{defCIMLult}), $c$ is given by (\ref{contractivityConstant}) and $A$ is given by (\ref{defA}).
\end{theorem}

This result allows us to construct an MLMC estimator with $L$ levels and good variance bounds by using couplings $(\bar{G}_k, \bar{X}^c_{kh})_{k=0}^{\infty}$. Namely, we start by considering the coarsest level $\ell = 0$, we take $\bar{X}^c_{kh} = \bar{X}^0_{k}$ and define the process $(\bar{G}_k)_{k=0}^{\infty}$ as explained in the proof of Theorem \ref{theoremInaccurateMLMCVarianceBound}, which corresponds to the finer level $\ell = 1$. Then we treat thus obtained $(\bar{G}_k)_{k=0}^{\infty}$ as a new coarse process and we repeat our procedure $L-1$ times.

\begin{remark}\label{remarkChoiceOfU}
	Assumptions \ref{as ina MLMC Ucoupling} can be easily satisfied by choosing the random variables $U^f_{kh}$, $U^f_{(k+1)h}$ and $U^c_{kh}$ independently and then using condition (\ref{ina estimatorVariance}) to verify (\ref{ina MLMC UcouplingCondition}). Then we obtain $\alpha_c = \alpha$.
	However, we would like to point out that such an approach to getting a bound on the variance of our MLMC estimator is not necessarily optimal. We believe that using more involved couplings of $U^f_{(k+1)h}$, $U^f_{kh}$ and $U^c_{kh}$ could lead to $\alpha_c > \alpha$ and hence to an improvement of the rate that we obtain above. We will consider this problem in our future research.
\end{remark}

\begin{proof}[Proof of Theorem \ref{theoremMainMLMC}]
 It is shown in Giles \cite{Giles2015Acta} that under the 
assumptions
\begin{eqnarray}
\label{as mlmc}
\bigl|\E[g(Y_T)- g(\bar X_T^{1,\ell}) ]|\lesssim  (h^{\ell})^{\widetilde{\alpha}},\quad 
\mathrm{Var}[g(\bar X_T^{1,\ell}) -g(\bar X_T^{1,\ell-1}) ]\lesssim  (h^{\ell})^{\beta}, 
\end{eqnarray}
for some \(\widetilde{\alpha}\geq 1/2,\) \(\beta>0,\) the computational complexity of the resulting multi-level estimator with  accuracy \(\varepsilon\)  is proportional to 
\begin{eqnarray*}
	\mathcal{C}=
	\begin{cases}
		\varepsilon^{-2}, & \beta>\gamma, \\
		\varepsilon^{-2}\log^2(\varepsilon), & \beta=\gamma, \\
		\varepsilon^{-2-(1-\beta)/\widetilde{\alpha}}, & 0<\beta <\gamma
	\end{cases}
\end{eqnarray*}
where the cost of sampling at level $\ell$ is of order $h^{-\ell \gamma}$. 
 Theorem \ref{th weak} tells us that in our case $\widetilde{\alpha}=1$  while Theorem \ref{theoremInaccurateMLMCVarianceBound} with $\bar{G}_0=\bar{X}^c_0$ gives $\beta = \min(\alpha_c, 1)/2$ since $g$ is Lipschitz. The cost of sampling at level $\ell$ is proportional to $(h^{\ell})^{-1}$, hence $\gamma=1$. Hence an overall cost of approximating $\bE[ g(Y_{T}) ]$ is of order $\epsilon^{-2-(1-\min(\alpha_c, 1)/2)}$. Consequently, the cost of approximating $\int_{\bR^d} g(x) \pi(dx)$ is $\epsilon^{-2-(1-\min(\alpha_c, 1)/2)}log{\epsilon}$. In particular, if we verify condition (\ref{ina MLMC UcouplingCondition}) by using (\ref{ina estimatorVariance}) as explained in Remark \ref{remarkChoiceOfU}, we have $\alpha_c = \alpha$ and the cost is $\epsilon^{-2-(1-\min(\alpha, 1)/2)}log{\epsilon}$. Hence the best gain that we can get is when we manage to choose $\alpha_c = 1$ (or $\alpha = 1$), which gives us $\beta = 1/2$ and the cost $\epsilon^{-5/2}log{\epsilon}$, which is a half order of magnitude better than the standard Monte Carlo approach.
 \end{proof}

\begin{remark}
	In \cite{szpruch2016multi}, authors considered a telescopic sum in two parameters, time-discretisation and time $t$ (length of the chain). We could also apply this idea in the current setting, which would lead to the reduction of overall cost by $\log{\epsilon}$.    
\end{remark}

\begin{remark}
As our primary interest in the present paper was to study the general randomised (inaccurate) Euler scheme, we refrain from comparing the complexity of the subsampling algorithms described in Example \ref{ex subsampling} with their multi-level counterparts. In that example one would need to take into consideration additional cost of simulating a step of Euler scheme ($m$ for accurate gradient and $s$ for inaccurate gradient). We leave the details to the interested reader.
\end{remark}

\subsubsection{Variance bound for the inaccurate drift case}\label{sectionVarianceBoundInaccurateMLMC}

\begin{proof}[Proof of Theorem \ref{theoremInaccurateMLMCVarianceBound}]
We consider the processes $(\bar{X}^f_{kh})_{k=0}^{\infty}$ and $(\bar{X}^c_{2kh})_{k=0}^{\infty}$ defined via (\ref{defbarXf}) and (\ref{defbarXc}), respectively.
We will obtain a bound on the variance by applying Theorem \ref{theoremContractivityForInaccurate} to $\bar{X}^f_{(k+2)h}$ and $\bar{X}^c_{(k+2)h}$. We consider the process $(\bar{S}_{kh})_{k=0}^{\infty}$ defined for $k \in 2 \mathbb{N}$ by
\begin{equation*}
\begin{split}
\bar{S}_{(k+1)h} &= \bar{X}^c_{kh} + h\bar{b}(\bar{X}^c_{kh}, U^f_{kh}) + \sqrt{h}Z_{k+1} \\
\bar{S}_{(k+2)h} &= \bar{S}_{(k+1)h} + h\bar{b}(\bar{S}_{(k+1)h}, U^f_{(k+1)h}) + \sqrt{h}Z_{k+2} \,.
\end{split}
\end{equation*}
We also need the process $(\bar{G}_k)_{k=0}^{\infty}$ coupled to $(\bar{S}_{kh})_{k=0}^{\infty}$ via
\begin{equation}\label{iMLMCdefbarG}
\bar{G}_{k+1} := \bar{\psi}_{m,H}(\bar{G}_k, \bar{S}_{kh}, U^f_{kh}, Z_{k+1}) \,.
\end{equation}
for all $k \in \mathbb{N}$, where we use the notation from (\ref{couplingForInaccurate}).
Then we have $\cL(\bar{G}_{k}) = \cL(\bar{X}^f_{kh})$ for all $k \geq 1$ if $\bar{G}_0 = \bar{X}^f_0$. Applying Theorem \ref{theoremContractivityForInaccurate}, we obtain
\begin{equation*}
\begin{split}
\mathbb{E}f(|\bar{G}_{k+2} &- \bar{X}^c_{(k+2)h}|) \leq (1-ch)\mathbb{E}f(|\bar{G}_{k+1} - \bar{S}_{(k+1)h}|) + \mathbb{E} \left[ \frac{1}{r_2}e^{-ar_2} |\bar{S}_{(k+2)h} - \bar{X}^c_{(k+2)h}|^2 \right] \\
&+ \mathbb{E} \left[ \left( 1 + \frac{1}{r_2}e^{-ar_2} \left( |\bar{G}_{k+1} - \bar{S}_{(k+1)h}|(1+h_0L) + \sqrt{h_0} \right) \right) |\bar{S}_{(k+2)h} - \bar{X}^c_{(k+2)h}| \right] \,.
\end{split}
\end{equation*}
Note now that
\begin{equation*}
\begin{split}
\mathbb{E} &\left[ \left( 1 + \frac{1}{r_2}e^{-ar_2} \left( |\bar{G}_{k+1} - \bar{S}_{(k+1)h}|(1+h_0L) + \sqrt{h_0} \right) \right) |\bar{S}_{(k+2)h} - \bar{X}^c_{(k+2)h}| \right] \\
&= \left( 1 + \frac{1}{r_2}e^{-ar_2}\sqrt{h_0} \right) \mathbb{E} |\bar{S}_{(k+2)h} - \bar{X}^c_{(k+2)h}| \\
&+ \frac{1}{r_2}e^{-ar_2}(1+h_0L) \mathbb{E} \left[ |\bar{G}_{k+1} - \bar{S}_{(k+1)h}| \cdot |\bar{S}_{(k+2)h} - \bar{X}^c_{(k+2)h}| \right] =: I_1 + I_2
\end{split}
\end{equation*}
and we have
\begin{equation*}
I_2 \leq \frac{1}{r_2}e^{-ar_2}(1+h_0L) \left( \mathbb{E} |\bar{G}_{k+1} - \bar{S}_{(k+1)h}|^2 \right)^{1/2} \cdot \left( \mathbb{E} |\bar{S}_{(k+2)h} - \bar{X}^c_{(k+2)h}|^2 \right)^{1/2} \,.
\end{equation*}
Now we use Minkowski's inequality and Lemma \ref{lemmaEulerMomentUniformEstimateInacurate} to estimate
\begin{equation*}
\left( \mathbb{E} |\bar{G}_{k+1} - \bar{S}_{(k+1)h}|^2 \right)^{1/2} \leq \left( \mathbb{E} |\bar{G}_{k+1}|^2 \right)^{1/2} + \left( \mathbb{E} |\bar{S}_{(k+1)h}|^2 \right)^{1/2} \leq \sqrt{C_{IEul}} + \left( \mathbb{E} |\bar{S}_{(k+1)h}|^2 \right)^{1/2} \,.
\end{equation*}
Moreover, due to (\ref{ina estimatorVariance}) we have
\begin{equation*}
\mathbb{E}|\bar{b}(\bar{X}^c_{kh}, U^f_{kh})|^2 \leq \sigma^2(\mathbb{E}|\bar{X}^c_{kh}|^2 + 1)h^{\alpha} + 2L^2 \mathbb{E}|\bar{X}^c_{kh}|^2 + 2|b(0)|^2 \leq C^{(2h)}_{IEul} (\sigma^2 + 2L^2) + \sigma^2 + 2|b(0)|^2 
\end{equation*}
where we used $h^{\alpha} \leq 1$ and we have the constant
\begin{equation*}
C^{(2h)}_{IEul} := \mathbb{E}|\bar{X}^c_0|^2 + \frac{4h_0 |b(0)|^2 + d + M_2 + 2h_0 \sigma^2}{M_1 - 4h_0 L^2 - 2h_0 \sigma^2}
\end{equation*}
which is specified as in (\ref{defCIEul}) but with $h$ replaced by $2h$. If we denote
\begin{equation*}
C_{ISM} := C^{(2h)}_{IEul} (\sigma^2 + 2L^2) + \sigma^2 + 2|b(0)|^2 
\end{equation*}
then we get
\begin{equation*}
\begin{split}
\left( \mathbb{E} |\bar{S}_{(k+1)h}|^2 \right)^{1/2} &= \left( \mathbb{E} |\bar{X}^c_{kh} + h\bar{b}(\bar{X}^c_{kh}, U^f_{kh}) + \sqrt{h}Z_{k+1}|^2 \right)^{1/2} \\
&\leq \left( \mathbb{E} |\bar{X}^c_{kh}|^2 \right)^{1/2} + \left(h^2  \mathbb{E} |\bar{b}(\bar{X}^c_{kh}, U^f_{kh})|^2 \right)^{1/2} + \left( h \mathbb{E} |Z_{k+1}|^2 \right)^{1/2} \\
&\leq \sqrt{C^{(2h)}_{IEul}} + h_0 \sqrt{C_{ISM}} + \sqrt{h_0 d} \,.
\end{split}
\end{equation*}
Hence, denoting $C_{IASP} := \left( \sqrt{C^{(2h)}_{IEul}} + h_0 \sqrt{C_{ISM}} + \sqrt{h_0 d} \right)^2$ we obtain
\begin{equation*}
I_2 \leq \frac{1}{r_2}e^{-ar_2}(1+h_0L)(\sqrt{C_{IEul}} + \sqrt{C_{IASP}}) \left( \mathbb{E} |\bar{S}_{(k+2)h} - \bar{X}^c_{(k+2)h}|^2 \right)^{1/2} \,.
\end{equation*}
Now we estimate
\begin{equation*}
\begin{split}
\mathbb{E}&|\bar{S}_{(k+2)h} - \bar{X}^c_{(k+2)h}|^2 = \mathbb{E}|\bar{X}^c_{kh} + h \bar{b}(\bar{X}^c_{kh}, U^f_{kh}) + \sqrt{h}Z_{k+1} \\
&+ h \bar{b}(\bar{X}^c_{kh} + h \bar{b}(\bar{X}^c_{kh}, U^f_{kh}) + \sqrt{h}Z_{k+1}, U^f_{(k+1)h})
+ \sqrt{h}Z_{k+2} - \bar{X}^c_{kh} - 2h \bar{b}(\bar{X}^c_{kh}, U^c_{kh}) - \sqrt{2h} \hat{Z}_{k+2}|^2 \\
&\leq 2 \mathbb{E}| h \bar{b}(\bar{X}^c_{kh}, U^f_{kh}) +  h \bar{b}(\bar{X}^c_{kh}, U^f_{(k+1)h}) - 2h \bar{b}(\bar{X}^c_{kh}, U^c_{kh})|^2 \\
&+ 2\mathbb{E}|h \bar{b}(\bar{X}^c_{kh} + h \bar{b}(\bar{X}^c_{kh}, U^f_{kh}) + \sqrt{h}Z_{k+1}, U^f_{(k+1)h})  -   h \bar{b}(\bar{X}^c_{kh}, U^f_{(k+1)h})|^2 =: \tilde{I}_1 + \tilde{I}_2 \,.
\end{split}
\end{equation*}
Note that due to (\ref{ina driftLipschitz}) we have
\begin{equation*}
\tilde{I}_2 \leq 2h^2 \bar{L}^2 \mathbb{E}|h \bar{b}(\bar{X}^c_{kh}, U^f_{kh}) + \sqrt{h}Z_{k+1}|^2 \leq 4h^4 \bar{L}^2 C_{ISM} + 4h^3 \bar{L}^2 d \,.
\end{equation*}
and due to (\ref{ina MLMC UcouplingCondition}) we have
\begin{equation*}
\tilde{I}_1 \leq 2h^2 L_u(1 + \mathbb{E}|\bar{X}^c_{kh}|^2) h^{\alpha_c} \leq 2 L_u (1 + C^{(2h)}_{IEul}) h^{2 + \alpha_c} \,.
\end{equation*}
If we denote
\begin{equation*}
C_{IMLdif} := (4h_0 \bar{L}^2 C_{ISM} + 4\bar{L} d)h_0^{(1-\alpha_c)^{+}} + 2L_u(1+C^{(2h)}_{IEul}) \,,
\end{equation*}
then we see that $\mathbb{E}|\bar{S}_{(k+2)h} - \bar{X}^c_{(k+2)h}|^2 \leq C_{IMLdif} h^{2 + \min(\alpha_c, 1)}$ and we have
\begin{equation*}
I_2 \leq \frac{1}{r_2}e^{-ar_2}(1+h_0L)(\sqrt{C_{IEul}} + \sqrt{C_{IASP}})\sqrt{C_{IMLdif}} h^{1 + \min(\alpha_c, 1)/2}
\end{equation*}
and
\begin{equation*}
I_1 \leq \left( 1 + \frac{1}{r_2}e^{-ar_2}\sqrt{h_0} \right) \sqrt{C_{IMLdif}} h^{1 + \min(\alpha_c, 1)/2} \,,
\end{equation*}
This gives us
\begin{equation*}
\mathbb{E}f(|\bar{G}_{k+2} - \bar{X}^c_{(k+2)h}|) \leq (1-ch)\mathbb{E}f(|\bar{G}_{k+1} - \bar{S}_{(k+1)h}|) + C_{IMLult} h^{1 + \min(\alpha_c, 1)/2} \,,
\end{equation*}
where
\begin{equation}\label{defCIMLult}
\begin{split}
C_{IMLult} &:= \frac{1}{r_2}e^{-ar_2}C_{IMLdif} h_0^{1 + \min(\alpha_c , 1)/2} + \left( 1 + \frac{1}{r_2}e^{-ar_2}\sqrt{h_0} \right) \sqrt{C_{IMLdif}} \\
&+  \frac{1}{r_2}e^{-ar_2}(1+h_0L)(\sqrt{C_{IEul}} + \sqrt{C_{IASP}})\sqrt{C_{IMLdif}} \,.
\end{split}
\end{equation}
Now we just apply the contractivity result for processes with inaccurate drifts (Theorem \ref{theoremContractivityForInaccurate}) to $\bar{G}_{k+1}$ and $\bar{S}_{(k+1)h}$ (note that they are appropriately coupled due to (\ref{iMLMCdefbarG})) and we obtain
\begin{equation*}
\mathbb{E}f(|\bar{G}_{k+2} - \bar{X}^c_{(k+2)h}|) \leq (1-ch)^2\mathbb{E}f(|\bar{G}_{k} - \bar{S}_{(k)h}|) + C_{IMLult} h^{1 + \min(\alpha_c, 1)/2} \,.
\end{equation*}
Hence for all $k \in 2\mathbb{N}$ we obtain
\begin{equation*}
\begin{split}
\mathbb{E}f(|\bar{G}_{k} - \bar{X}^c_{kh}|) &\leq (1-ch)^k\mathbb{E}f(|\bar{G}_{0} - \bar{X}^c_{0}|) + \sum_{j=0}^{k-1} (1-ch)^j C_{IMLult} h^{1 + \min(\alpha_c, 1)/2}\\
&\leq (1-ch)^k\mathbb{E}f(|\bar{G}_{0} - \bar{X}^c_{0}|) + \frac{C_{IMLult}}{c}h^{\min(\alpha_c, 1)/2} \,.
\end{split}
\end{equation*}
Using Lemma \ref{corollarySquareComparison} we obtain
\begin{equation*}
\mathbb{E}|\bar{G}_k - \bar{X}^c_{kh}|^2 \leq A(1-ch)^k \mathbb{E}f(|\bar{G}_{0} - \bar{X}^c_{0}|) + \frac{AC_{IMLult}}{c} h^{\min(\alpha_c, 1)/2} 
\end{equation*}
with the constant $A$ given by (\ref{defA}), which finishes the proof.
\end{proof}

\section{Proof of the contractivity result}\label{sectionProofContractivity}

As we have already remarked in Section \ref{sectionCoupling}, we only need to prove Theorem \ref{theoremContractivityForInaccurate} and hence we will be working here with inaccurate drifts with randomness induced by a random variable $U$ independent of $Z$. Recall that we have $\hat{x} = x + h \bar{b}(x,U)$ and $\hat{y} = y + h \bar{b}(y,U)$.

Note that some parts of this proof (Lemmas \ref{lemmaCouplingFirstMoment}, \ref{lemmaLowerBound} and Subsection \ref{subsectionSmallEstimates}) are based on Section 6 in \cite{EberleMajka2018}. We include all the details here for the reader's convenience and in order to highlight the modifications that need to be introduced in the proofs from \cite{EberleMajka2018} in order to obtain $L^2$ bounds. The reader who is only interested in Euler schemes with non-random drifts can obtain a direct proof of Theorem \ref{mainContractivityTheorem} from the reasoning below by setting $\hat{x} = x + hb(x)$ and replacing $\mathbb{E}[\, \cdot \, | \, U]$ with $\mathbb{E}[\, \cdot \,]$ and the assumptions (\ref{ina driftLipschitz}) and (\ref{ina contractivity}) with (\ref{driftLipschitz}) and (\ref{driftDissipativityAtInf}), respectively.

We denote $R' := |X' - Y'|$, $\hat{r} := |\hat{x} - \hat{y}|$ and $r := |x - y|$. We want to show that for any $r \in [0,\infty)$ we have
\begin{equation}\label{proofContractivity}
\mathbb{E}\left[ f(R') \, | \, U \right] - f(r) \leq - chf(r) \text{ a.s.}
\end{equation}
Before we proceed with the detailed proof, let us sketch some main ideas and formulate two crucial lemmas. 

We proceed by decomposing the Euler scheme step $r \mapsto R'$ into the drift step $r \mapsto \hat{r}$ and the Gaussian step $\hat{r} \mapsto R'$. We have
$\mathbb{E}\left[ f(R') \, | \, U \right] - f(r) = \mathbb{E}\left[ f(R') \, | \, U \right] - f(\hat{r}) + f(\hat{r}) - f(r)$ and we will want to use either $\mathbb{E}\left[ f(R') \, | \, U \right] - f(\hat{r})$ or $f(\hat{r}) - f(r)$ (depending on the values of $r$ and $\hat{r}$) to get an upper bound of the form $- ch f(r)$.

The main idea is that whenever we are in the region of space where the contractivity condition (\ref{ina contractivity}) holds, we should use the drift step and switch off the Gaussian movement by applying the synchronous coupling. On the other hand, when the contractivity condition does not work, we have to use an appropriate coupling to get the desired upper bounds from the Gaussian step, while controlling the drift step via the Lipschitz condition (\ref{ina driftLipschitz}).

One of the most important parts of the proof is an application of the Taylor formula to obtain
\begin{equation*}
f(R') - f(\hat{r}) = f'(\hat{r})(R' - \hat{r}) + f''(\theta)(R' - \hat{r})^2 
\end{equation*}
for some $\theta$ between $\hat{r}$ and $R'$. Hence we see that our crucial task is to control the first and the second moment of $R' - \hat{r}$. Actually, we will be able to consider the second moment only when $R'$ is restricted to a specific interval. In fact we will choose random (only through dependence on $U$) intervals
\begin{equation}\label{choiceOfIntervals}
I_{\hat{r}} = \begin{cases}
(0, \hat{r} + \sqrt{h}) & \text{ if } \hat{r} \leq \sqrt{h} \\
(\hat{r} - \sqrt{h}, \hat{r}) & \text{ if } \hat{r} > \sqrt{h} \,.
\end{cases}
\end{equation}
For such intervals, we will obtain the random variable $\underline{\alpha}(\hat{r})$, for which we have
\begin{equation*}
\underline{\alpha}(\hat{r}) \leq \mathbb{E}\left[(R' - \hat{r})^2 \mathbf{1}_{\{ R' \in I_{\hat{r}} \}} \, | \, U \right] \text{ a.s.}
\end{equation*}
The reason for this choice of intervals $I_{\hat{r}}$ will become apparent from the proof of Lemma \ref{lemmaLowerBound}. Generally speaking, we want to choose a small interval around $\hat{r}$ and it is convenient to take $(\hat{r} - \sqrt{h}, \hat{r})$ for getting bounds on $\sup f''$ on $[0,r_1]$. However, we cannot take just $(0,\hat{r})$ if $\hat{r} \leq \sqrt{h}$ as the length of such an interval would not have a uniform lower bound and it would be impossible to get the lower bound in Lemma \ref{lemmaLowerBound}, hence we need to take $(0, \hat{r} + \sqrt{h})$ when $\hat{r}$ is small. 

From now on, all the relations between random variables in the proof are supposed to be understood as holding almost surely. Let us first formulate two auxiliary results on our coupling, which allow us to control the conditional moments of $R' - \hat{r}$.

\begin{lemma}\label{lemmaCouplingFirstMoment}
	For the coupling $(X',Y')$ defined by (\ref{couplingForInaccurate}), we have
	$\mathbb{E}\left[ R' \, | \, U \right] = \hat{r}$.
\end{lemma}

\begin{lemma}\label{lemmaLowerBound}
	For the coupling $(X',Y')$ defined by (\ref{couplingForInaccurate}), if $h \leq 4 m^2$, then we have
	\begin{equation*}
	\underline{\alpha}(\hat{r}) \mathbf{1}_{\{ |\hat{r}| \leq H \}} \leq \mathbb{E}\left[(R' - \hat{r})^2 \mathbf{1}_{\{ R' \in I_{\hat{r}} \}} \, | \, U \right] \,,
	\end{equation*}
	where
	$\underline{\alpha}(\hat{r}) = \frac{1}{2}c_0 \min ( \sqrt{h}, \hat{r}) \sqrt{h}$
	and
	\begin{equation*}
	c_0 := 4 \min \left( \int_0^{1/2} u^2 (1 - e^{u-1/2})\varphi_{0,1}(u) du , (1 - e^{-1}) \int_0^{1/2} u^3 \varphi_{0,1}(u) du \right) \,.
	\end{equation*}
\end{lemma}

The proofs are based on the calculations from Section 6 in \cite{EberleMajka2018} and can be found in Appendix \ref{appendixContractivity}.

Before we proceed, let us make a remark about the choice of parameters in our coupling (\ref{ourCoupling}).

\begin{remark}
	The choice of $m = \sqrt{h_0}/2$ in (\ref{ourCoupling}) means that we rarely make the non-synchronous step. In principle, we could choose $m$ arbitrarily large, but then from the proof we see that we would also need to redefine $r_2 := r_1 + 2m$, and as $r_2$ increases, the constant $c$ decreases to zero. Moreover, we see that $A \to \infty$ as $r_2 \to \infty$. This shows that increasing $m$ actually gives us worse constants and hence $m$ should be kept as small as possible. This is related to the fact that the only place in the calculations where we gain something from the non-synchronous behaviour of our coupling is in the lower bound for $\underline{\alpha}$, which is the coefficient near $f''$. But this supremum is taken over a small interval and actually when we compute the lower bound for $\underline{\alpha}$ we only integrate a Gaussian density over an interval of length $[0,\sqrt{h} / 2]$. Everywhere else the non-synchronous behaviour is actually harmful to our estimates (due to the convex cost), so it makes sense to take $m$ very small, i.e., $m = \frac{\sqrt{h_0}}{2}$ and $r_2 = r_1 + 2m = r_1 + \sqrt{h_0}$.
\end{remark}

We are now ready to proceed with the proof of Theorem \ref{theoremContractivityForInaccurate}, which we will split into a few steps, depending on the size of the argument $r$. First note that if $r = |x - y| > \mathcal{R}$, then due to (\ref{ina driftLipschitz}) and (\ref{ina contractivity}) we have
\begin{equation*}
\begin{split}
\hat{r} &= \sqrt{r^2 + 2h \langle x - y , \bar{b}(x,U) - \bar{b}(y,U) \rangle + h^2 |\bar{b}(x,U) - \bar{b}(y,U)|^2} \\
&\leq \sqrt{1 - 2hK + h^2 L^2} \\
&\leq r\left(1 - hK + \frac{h^2 L^2}{2}\right) \,,
\end{split}
\end{equation*}
where we used the fact that $\sqrt{1 + x} \leq 1 + \frac{x}{2}$ for all $x \geq - 1$. Note that due to Assumptions \ref{as diss} we have $K \leq L$ and hence the expression under the square root is indeed non-negative.
Therefore if we choose $h$ small enough, we have $\hat{r} \leq r$ for $r > \mathcal{R}$. More precisely, we will assume $h \leq K / L^2$, which implies
\begin{equation*}
\frac{h^2 L^2}{2} = \frac{h L^2}{2} h \leq \frac{h L^2}{2} \frac{K}{L^2} = \frac{hK}{2} \,,
\end{equation*}
hence in fact we have
$- hK + \frac{h^2 L^2}{2} \leq - \frac{hK}{2}$
and we will later use the fact that for all $r > \mathcal{R}$ we have
\begin{equation}\label{inequalityrhatBound}
\hat{r} \leq r \left( 1 - \frac{hK}{2} \right) \,.
\end{equation}
However, even if $r \leq \mathcal{R}$, we can still control the size of $\hat{r}$ due to (\ref{ina driftLipschitz}). Namely, we have
\begin{equation*}
\hat{r} = \sqrt{|x - y + h(\bar{b}(x,U) - \bar{b}(y,U))|^2} \leq (1 + hL)\mathcal{R} \leq (1 + h_0L)\mathcal{R} = r_1
\end{equation*}
for $r \leq \mathcal{R}$. This is the condition motivating the choice of $r_1$, i.e., if $r$ is in $[0,\mathcal{R}]$, then the application of the drift can increase the distance $\hat{r}$ maximally up to $r_1$.

Having the above basic estimates, we can begin the proof of the easiest case, when $r \in [r_1,r_2]$.

\subsection{Estimates for $r \in [r_1,r_2]$.}\label{subsectionMiddleEstimates}

Since $r \geq r_1 > \mathcal{R}$, we have $\hat{r} \leq r \leq r_2$. This means that the interval $(\hat{r},r)$ is contained in $[0, r_2]$, where the function $f$ is concave. Hence for the drift step we have
\begin{equation}\label{middleEstimates1}
f(\hat{r}) - f(r) \leq f'(r)(\hat{r} - r) \,.
\end{equation}
Furthermore, the derivative of $f$ is bounded from below on the interval $[0,r_2]$ by $f'(r_2) = e^{-ar_2}$, which implies
\begin{equation}\label{middleEstimates3}
-\frac{Khr}{2}f'(r) \leq -\frac{Khr}{2}e^{-ar_2} \leq -\frac{Kh}{2}e^{-ar_2} f(r) \,,
\end{equation}
where in the last step we additionally use the fact that $f(r) \leq r$ for $r \in [0,r_2]$. Combining (\ref{inequalityrhatBound}), (\ref{middleEstimates1}) and (\ref{middleEstimates3}) gives
$f(\hat{r}) - f(r) \leq -\frac{Kh}{2}e^{-ar_2} f(r)$
for all $r \in [r_1, r_2]$, which is exactly what we want. Now we have to show that the Gaussian step does not spoil these estimates. We need to consider two cases.

\paragraph{The case of $\hat{r} \leq r_1$.}

Observe that due to our coupling construction, the value of $R'$ is always within the interval $[0, \hat{r} + 2m]$ (it can become zero when we jump to the same point, and if we reflect the jumps then $R'$ can increase maximally up to $\hat{r} + 2m$ due to the truncation of jumps by $m$). Hence the interval with endpoints $\hat{r}$ and $R'$ is always contained in $[0,\hat{r}+2m]$. Hence, if $\hat{r} \leq r_1$, then 
$(\hat{r}, R') \subset [0,r_1+2m] = [0,r_2]$,
where $(\hat{r}, R')$ should be interpreted as $(R', \hat{r})$ if $R' < \hat{r}$. Recalling that the function $f$ is concave on $[0,r_2]$, by the Taylor formula we get
\begin{equation*}
\mathbb{E}\left[ f(R') \, | \, U \right] - f(\hat{r}) = f'(\hat{r})\mathbb{E}\left[ (R' - \hat{r}) \, | \, U \right] + \mathbb{E}\left[ f''(\theta)(R' - \hat{r})^2 \, | \, U \right] \leq f'(\hat{r})\mathbb{E}\left[ (R' - \hat{r}) \, | \, U \right]  
\end{equation*}
for some $\theta \in (\hat{r}, R') \subset [0,r_2]$. However, due to Lemma \ref{lemmaCouplingFirstMoment} we have $\mathbb{E}\left[ (R' - \hat{r}) \, | \, U \right] = 0$ and hence
$\mathbb{E}\left[ f(R') \, | \, U \right] - f(\hat{r}) \leq 0$.

\paragraph{The case of $\hat{r} > r_1$.}

Here the interval $[0, \hat{r} + 2m]$ is no longer contained in $[0,r_2]$ and the function $f$ is convex for arguments larger than $r_2$, so we cannot bound $f''$ by zero as above. This is why for $\hat{r} > r_1$ we use the synchronous coupling and we have
$\mathbb{E}\left[ f(R') \, | \, U \right] - f(\hat{r}) = 0$.

Hence, combining all the estimates from this subsection together, we see that we have
$\mathbb{E}\left[ f(R') \, | \, U \right] - f(r) \leq -\frac{Kh}{2}e^{-ar_2} f(r)$ for all $r \in [r_1, r_2]$.

\subsection{Estimates for $r \in [r_2,\infty)$.}

Here we deal with the Gaussian step exactly as in the previous subsection. If $\hat{r} \leq r_1$, then we can use concavity of $f$ on $[0,r_2]$. Otherwise, for $\hat{r} > r_1$, we use the synchronous coupling. Hence $\mathbb{E}\left[ f(R') \, | \, U \right] - f(\hat{r}) \leq 0$ and we only need to deal with the drift step. Since $r \geq r_2 > \mathcal{R}$, we can bound $\hat{r}$ by using (\ref{inequalityrhatBound}).
However, we do not know whether $r \left( 1 - \frac{Kh}{2} \right)$ is smaller or greater than $r_2$, and since the function $f$ is given via two different formulas for arguments below and above $r_2$, we need to consider two different cases.

\paragraph{The case of $r\left( 1 - \frac{Kh}{2} \right) \geq r_2$.}\label{subsubsectionLargeEstimates1}

Since $f$ is increasing, using (\ref{inequalityrhatBound}) we have
\begin{equation}\label{largeEstimates1}
f(\hat{r}) \leq f\left(r\left( 1 - \frac{Kh}{2} \right)\right) = \frac{1}{a}\left( 1 - e^{-ar_2}\right) + \frac{1}{2r_2}e^{-ar_2} \left( r^2 \left( 1 - \frac{Kh}{2} \right)^2 - r_2^2 \right) \,.
\end{equation}
We want to bound this expression from above by
$f(r) - c'hf(r)$
for some constant $c' > 0$ and we know that, since $r \geq r_2$, we have
\begin{equation}\label{largef}
f(r) = \frac{1}{a}\left( 1 - e^{-ar_2}\right) + \frac{1}{2r_2}e^{-ar_2} \left( r^2 - r_2^2 \right) \,.
\end{equation}
A simple calculation shows that we need to find $c' > 0$ such that
\begin{equation}\label{largeEstimates4}
\frac{1}{2r_2}e^{-ar_2} \left(  -Kh + \frac{K^2h^2}{4} \right) r^2 \leq - c'hf(r)
\end{equation}
holds for all $r \geq r_2$. In order for this to be possible, the left hand side needs to be negative. Note that $-Kh + \frac{K^2h^2}{4} < 0$ when $h < \frac{4}{K}$, which holds due to our choice of $h_0$, cf. (\ref{defh0}). We have
\begin{equation}\label{largeEstimates2}
\begin{split}
\left( Kh - \frac{K^2 h^2}{4} \right) f(r) &= \left( Kh - \frac{K^2 h^2}{4} \right) \frac{1}{2r_2}e^{-ar_2} r^2 \\
&+ \left( Kh - \frac{K^2 h^2}{4} \right) \left( \frac{1}{a}\left( 1 - e^{-ar_2}\right) - \frac{1}{2r_2}e^{-ar_2} r_2^2 \right) \,.
\end{split}
\end{equation}
Now observe that we can find a constant $c_2 > 0$ such that
\begin{equation}\label{largeEstimates3}
\frac{1}{a}\left( 1 - e^{-ar_2}\right) - \frac{1}{2r_2}e^{-ar_2} r_2^2 \leq c_2 \frac{1}{2r_2}e^{-ar_2} r^2
\end{equation}
for all $r \geq r_2$. Namely, we set
\begin{equation*}
c_2 := \frac{\frac{1}{a}\left( 1 - e^{-ar_2}\right) - \frac{1}{2r_2}e^{-ar_2} r_2^2}{\frac{1}{2r_2}e^{-ar_2} r_2^2} \,.
\end{equation*}
We know that $c_2 > 0$ since the function
$g(r) := \frac{1}{a}\left( 1 - e^{-ar}\right) - \frac{1}{2}e^{-ar} r$
is increasing for all $r > 0$ and $g(0) = 0$.
From (\ref{largeEstimates2}) and (\ref{largeEstimates3}) we get
\begin{equation*}
\frac{\left( Kh - \frac{K^2 h^2}{4} \right)}{1 + c_2} f(r) \leq \left( Kh - \frac{K^2 h^2}{4} \right) \frac{1}{2r_2}e^{-ar_2} r^2
\end{equation*}
for all $r \geq r_2$. Combining this with (\ref{largeEstimates4}), we obtain
\begin{equation*}
f(\hat{r}) \leq f(r) - \frac{\left( Kh - \frac{K^2 h^2}{4} \right)}{1 + c_2} f(r) = f(r) - \frac{\frac{1}{2}e^{-ar_2}r_2}{\frac{1}{a}\left( 1 - e^{-ar_2} \right)} \left( Kh - \frac{K^2 h^2}{4} \right) f(r) \,.
\end{equation*}

\paragraph{The case of $r\left( 1 - \frac{Kh}{2} \right) \leq r_2$.}

We again use the fact that $\hat{r} \leq r\left( 1 - \frac{Kh}{2} \right)$ and $f(\hat{r}) \leq f\left(r\left( 1 - \frac{Kh}{2} \right)\right)$. The difficulty in this case comes from the fact that we need to compare the values of $f$ given by two different formulas, i.e., $f(r)$ given by (\ref{largef}) and
$f \left( r \left( 1 - \frac{Kh}{2} \right) \right) = \frac{1}{a} \left( 1 - e^{-a\left( r \left( 1 - \frac{Kh}{2} \right) \right)} \right)$.
We can circumvent this problem by considering the midpoint between $r \left( 1 - \frac{Kh}{2} \right)$ and $r$, namely, the point $r \left( 1 - \frac{Kh}{4} \right)$. We have
\begin{equation}\label{largeEstimates5}
f(\hat{r}) \leq f\left( r \left( 1 - \frac{Kh}{2} \right) \right) = f\left( r \left( 1 - \frac{Kh}{2} \right) \right) - f\left( r \left( 1 - \frac{Kh}{4} \right) \right) + f\left( r \left( 1 - \frac{Kh}{4} \right) \right) \,.
\end{equation}
We have either $r \left( 1 - \frac{Kh}{4} \right) \leq r_2$ or $r \left( 1 - \frac{Kh}{4} \right) > r_2$. In the former case, we use the fact that $f$ is increasing to get
$f\left( r \left( 1 - \frac{Kh}{4} \right) \right) \leq f(r)$
and for the remaining term in (\ref{largeEstimates5}) we proceed similarly to Subsection \ref{subsectionMiddleEstimates}. Namely, we have
\begin{equation*}
\begin{split}
f\left( r \left( 1 - \frac{Kh}{2} \right) \right) - f\left( r \left( 1 - \frac{Kh}{4} \right) \right) &\leq f'\left( r \left( 1 - \frac{Kh}{4} \right) \right) \left( r\left( 1 - \frac{Kh}{2} \right) - r\left( 1 - \frac{Kh}{4} \right) \right) \\
&= - \frac{Kh}{4} f'\left( r \left( 1 - \frac{Kh}{4} \right) \right) r
\end{split}
\end{equation*}
due to concavity of $f$ on $[0,r_2]$ and the fact that $[r \left( 1 - \frac{Kh}{2} \right), r \left( 1 - \frac{Kh}{4} \right)] \subset [0,r_2]$. Now we use the fact that $f'(r) \geq e^{-ar_2}$ and $f(r) \leq r$ for $r \in [0,r_2]$, which gives us 
\begin{equation*}
- \frac{Kh}{4} f'\left( r \left( 1 - \frac{Kh}{4} \right) \right) r \leq - e^{-ar_2} r \frac{Kh}{4} \leq - e^{-ar_2} \frac{Kh}{4} f(r) \,.
\end{equation*}
On the other hand, if $r \left( 1 - \frac{Kh}{4} \right) > r_2$, then we use the fact that $f$ is increasing to get
$f\left( r \left( 1 - \frac{Kh}{2} \right) \right) - f\left( r \left( 1 - \frac{Kh}{4} \right) \right) \leq 0$
and we have
$f(\hat{r}) - f(r) \leq f\left( r \left( 1 - \frac{Kh}{4} \right) \right) - f(r)$.
Since $r \left( 1 - \frac{Kh}{4} \right) > r_2$, we can use our calculations from Subsection \ref{subsubsectionLargeEstimates1}. We just need to replace $\left( 1 - \frac{Kh}{2} \right)$ therein with $\left( 1 - \frac{Kh}{4} \right)$, and, in consequence, we obtain
\begin{equation*}
f\left( r \left( 1 - \frac{Kh}{4} \right) \right) - f(r) \leq - \frac{\frac{1}{2}e^{-ar_2}r_2}{\frac{1}{a}\left( 1 - e^{-ar_2} \right)} \left( \frac{Kh}{2} - \frac{K^2 h^2}{16} \right) f(r) \,.
\end{equation*}
Note that we have $\frac{K}{2} - \frac{K^2 h}{16} > \frac{K}{4}$ since $h < \frac{4}{K}$ by (\ref{defh0}). Hence, combining all the estimates from this subsection together, we see that for all $r \in [r_2, \infty)$ we have
\begin{equation*}
\mathbb{E}\left[ f(R') \, | \, U \right] - f(r) \leq f(\hat{r}) - f(r) \leq - \min \left( e^{-ar_2} \frac{Kh}{4} , \frac{\frac{1}{2}e^{-ar_2}r_2}{\frac{1}{a}\left( 1 - e^{-ar_2} \right)} \frac{Kh}{4} \right) f(r) \,.
\end{equation*}
This is also valid for $r \in [r_1, r_2]$, cf. the constant obtained in Subsection \ref{subsectionMiddleEstimates}.

\subsection{Estimates for $r \in [0,r_1]$.}\label{subsectionSmallEstimates}

For the drift step we have two cases. If $\hat{r} \leq r$, then we just use the fact that $f$ is increasing and we get
$f(\hat{r}) - f(r) \leq 0$.
If $\hat{r} > r$, we will need to use
$f(\hat{r}) - f(r) \leq f'(r)(\hat{r} - r)$,
which holds due to the fact that $r \leq r_1$ implies $\hat{r} \leq r_1$ and $f$ is concave on $[0,r_1]$. We will show later in this section how to bound this term in order to be able to use it in connection with the Gaussian step contribution to obtain the desired bounds.

For the Gaussian step, we again use that $r \in [0,r_1]$ implies $\hat{r} \in [0,r_1]$. Therefore, by our coupling construction, $R' \in [0, r_1 + 2m] = [0, r_2]$ and we can use the fact that $f''$ is concave on $[0,r_2]$.
We apply the Taylor formula (note that the function $f$ is twice differentiable), i.e., we have
\begin{equation*}
\mathbb{E}\left[ f(R') \, | \, U \right] - f(\hat{r}) = f'(\hat{r}) \mathbb{E}\left[ (R' - \hat{r}) \, | \, U \right] + \mathbb{E} \left[ f''(\theta) (R' - \hat{r})^2 \, | \, U \right] \,,
\end{equation*}
where $\theta$ is a point between $\hat{r}$ and $R'$. We bound
\begin{equation*}
\begin{split}
\mathbb{E} \left[ f''(\theta) (R' - \hat{r})^2 \, | \, U \right] &\leq \mathbb{E} \left[ \sup_{\theta \in (\hat{r}, R')} f''(\theta) (R' - \hat{r})^2 \, \Big| \, U \right] \\
&= \mathbb{E} \left[ \sup_{\theta \in (\hat{r}, R')} f''(\theta) (R' - \hat{r})^2 \mathbf{1}_{\{ R' \in I_{\hat{r}}\}} + \sup_{\theta \in (\hat{r}, R')} f''(\theta) (R' - \hat{r})^2 \mathbf{1}_{\{ R' \notin I_{\hat{r}}\}} \, \Big| \, U \right] \\
&\leq \mathbb{E} \left[ \sup_{\theta \in I_{\hat{r}}} f''(\theta) (R' - \hat{r})^2 \mathbf{1}_{\{ R' \in I_{\hat{r}}\}} + \sup_{\theta \in (\hat{r}, R')} f''(\theta) (R' - \hat{r})^2 \mathbf{1}_{\{ R' \notin I_{\hat{r}}\}} \, \Big| \, U \right] \\
&\leq \sup_{\theta \in I_{\hat{r}}} f''(\theta) \mathbb{E} \left[ (R' - \hat{r})^2 \mathbf{1}_{\{ R' \in I_{\hat{r}}\}} \, | \, U \right] \,,
\end{split}
\end{equation*}
where $(\hat{r}, R')$ denotes either $(\hat{r}, R')$ or $(R', \hat{r})$, depending on the relation between $\hat{r}$ and $R'$, $I_{\hat{r}}$ can be any (random) open interval such that $\hat{r}$ belongs to the closure of $I_{\hat{r}}$, in the third line we use the fact that $R' \in I_{\hat{r}}$ implies that $(\hat{r}, R') \subset I_{\hat{r}}$ and in the last step we use the fact that $f'' \leq 0$ on $[0, r_2]$ so we can bound the second term by zero, whereas in the first expression $\sup_{\theta \in I_{\hat{r}}} f''(\theta)$ is pulled out in front of the conditional expectation as it is a measurable function of $U$.

We also use Lemma \ref{lemmaCouplingFirstMoment} to get $\mathbb{E}\left[ (R' - \hat{r}) \, | \, U \right]  = 0$. Hence we have
\begin{equation}\label{smallEstimates1}
\mathbb{E}\left[ f(R') \, | \, U \right] - f(\hat{r}) \leq \frac{1}{2}\underline{\alpha}(\hat{r}) \sup_{I_{\hat{r}}} f'' \,,
\end{equation}
where
$\underline{\alpha}(\hat{r}) \leq \mathbb{E} \left[ (R' - \hat{r})^2 \mathbf{1}_{\{ R' \in I_{\hat{r}}\}} \, | \, U \right]$
and the exact value of $\underline{\alpha}(\hat{r})$ is specified in Lemma \ref{lemmaLowerBound} (note that here $\hat{r} \leq r_1 = H$) for our particular choice of intervals $I_{\hat{r}}$ given by (\ref{choiceOfIntervals}). Recall that for the drift step we get
\begin{equation}\label{smallEstimates2}
f(\hat{r}) - f(r) \leq f'(r)(\hat{r} - r) 
\end{equation}
and that due to (\ref{ina driftLipschitz}) we have
$\hat{r} = |x - y + h(\bar{b}(x,U) - \bar{b}(y,U))| \leq |x - y| + h|\bar{b}(x,U) - \bar{b}(y,U)| \leq r + hLr$.
We would like to replace $f'(r)$ in (\ref{smallEstimates2}) with $f'(\hat{r})$, since in (\ref{smallEstimates1}) we consider supremum of $f''$ on a small neighbourhood of $\hat{r}$ rather than $r$. Recall now that $f'(r) = e^{-ar}$ for $r \leq r_2$ (here we consider $r \leq r_1 \leq r_2$) and hence
$f'(r) = e^{-a(r - \hat{r} + \hat{r})} = e^{-a\hat{r}}e^{-a(r - \hat{r})} = f'(\hat{r}) e^{-a(r - \hat{r})}$.
We also have
$\hat{r} - r \leq hLr$
and hence we can bound
\begin{equation*}
f(\hat{r}) - f(r) \leq f'(r)(\hat{r} - r) \leq f'(\hat{r})e^{ahL r_1} (\hat{r} - r) \leq f'(\hat{r})e^{ahL r_1} hL\hat{r} \,,
\end{equation*}
where the last step holds for $r \leq \hat{r}$. Note, however, that for $r > \hat{r}$ we have $f(\hat{r}) - f(r) \leq 0$ and hence the bound
\begin{equation}\label{smallEstimates3}
f(\hat{r}) - f(r) \leq f'(\hat{r})e^{ahL r_1} hL\hat{r}
\end{equation}
holds for all $r \in [0,r_1]$.
Thus, combining (\ref{smallEstimates1}) and (\ref{smallEstimates3}), we see that we need to show that
\begin{equation*}
f'(\hat{r})e^{ahL r_1} hL\hat{r} + \frac{1}{2} \underline{\alpha}(\hat{r}) \sup_{I_{\hat{r}}} f'' \leq - c' h f(r)
\end{equation*}
for some $c' > 0$ and for all $r \in [0,r_1]$. We will in fact show
\begin{equation}\label{smallEstimatesGoal}
f'(\hat{r})e^{ahL r_1} hL\hat{r} + \frac{1}{2} \underline{\alpha}(\hat{r}) \sup_{I_{\hat{r}}} f'' \leq - c h f(\hat{r}) \,.
\end{equation}
Note that this implies an upper bound by $-chf(r)/2$ as long as $hL \in (0,1/2)$. Indeed, we have
$r = |x - y| \leq \hat{r} + hLr$
and thus 
$f(\hat{r}) \geq f((1 - hL)r) \geq (1 - hL)f(r) \geq f(r)/2$,
since $f$ is increasing and concave on $[0, r_1]$ and $hL \in (0,1/2)$ by assumption (\ref{defh0}).

Recall now that
\begin{equation*}
I_{\hat{r}} = \begin{cases}
(0, \hat{r} + \sqrt{h}) & \text{ if } \hat{r} \leq \sqrt{h} \\
(\hat{r} - \sqrt{h}, \hat{r}) & \text{ if } \hat{r} > \sqrt{h}
\end{cases}
\end{equation*}
and $f''$ is increasing on $[0,r_2]$, since $f'''(r) = a^2 e^{-ar}$ for $r \in [0,r_2]$. Therefore
\begin{equation*}
\sup_{I_{\hat{r}}} f'' = \begin{cases}
-ae^{-a(\hat{r} + \sqrt{h})} & \text{ if } \hat{r} \leq \sqrt{h} \\
-ae^{-a\hat{r}} & \text{ if } \hat{r} > \sqrt{h}
\end{cases} \,.
\end{equation*}
We proceed by considering two separate cases, related to the form of the interval $I_{\hat{r}}$ and, consequently, the form of the bound $\underline{\alpha}(\hat{r})$ from Lemma \ref{lemmaLowerBound}.

\paragraph{The case of $\hat{r} > \sqrt{h}$.}
Looking at (\ref{smallEstimatesGoal}), we see that we need to show that
\begin{equation}\label{smallEstimatesGoal2}
e^{-a\hat{r}}e^{ahL r_1} hL\hat{r} - \frac{1}{2} \underline{\alpha}(\hat{r}) ae^{-a\hat{r}} \leq - ch \frac{1}{a}(1 - e^{-a\hat{r}}) 
\end{equation}
for all $\hat{r} \leq r_1$. Recall that
$\underline{\alpha}(\hat{r}) = c_0 \min (\hat{r}, \sqrt{h})\sqrt{h}$.
We will show that
\begin{equation}\label{smallEstimatesAuxilliary}
e^{-a\hat{r}}e^{ahL r_1} L\hat{r} - \frac{1}{2} c_0 ae^{-a\hat{r}} \leq - c \frac{1}{a}(1 - e^{-a\hat{r}}) \,.
\end{equation}
To this end it is sufficient to show that the left hand side of (\ref{smallEstimatesAuxilliary}) is bounded by $-c/a$ for all $\hat{r} \in [0,r_1]$.
Hence we need
\begin{equation*}
e^{ahLr_1} L \hat{r} - \frac{1}{2}c_0 a \leq -\frac{c}{a} e^{a \hat{r}} \,.
\end{equation*}
Recall that $r_1 = (1 + h_0L)\mathcal{R} \leq 3 \mathcal{R}/2$, since $h_0L < 1/2$ by assumption (\ref{defh0}). Since $h < h_0$ and we have
\begin{equation*}
h_0 \leq \frac{2c_0 \ln{\frac{3}{2}}}{27L^2 \mathcal{R}^2} \leq \frac{c_0 \ln{\frac{3}{2}}}{9L^2 r_1 \mathcal{R}} \leq \frac{\ln{\frac{3}{2}}}{aLr_1} \,,
\end{equation*}
we see that
$e^{ahLr_1} \leq 3/2$.
We would like to stress the fact that this bound (and hence the assumption on $h_0$) is to some extent arbitrary and we can modify it in order to have in the last estimate any number strictly greater than $1$ instead of $3/2$. We need to have
\begin{equation}\label{smallEstimatesAuxilliary2}
\frac{3}{2} L \hat{r} - \frac{1}{2}c_0 a \leq -\frac{c}{a} e^{a \hat{r}} \,.
\end{equation}
First we need to make sure that
$3 L \hat{r} - c_0 a < 0$
holds for all $\hat{r} \leq r_1$, by choosing $a$ in an appropriate way. This is guaranteed by our choice of $a = 6 L r_1 / c_0$ made in (\ref{defa}), i.e., it is chosen (again to some extent arbitrarily, up to a constant) so that
$\frac{3}{2} L r_1 = \frac{1}{4}c_0 a$.
With this choice of $a$ we have
$\frac{3}{2} L \hat{r} - \frac{1}{2}c_0 a \leq - \frac{1}{4} c_0 a$
for all $\hat{r} \leq r_1$ and hence in order to get (\ref{smallEstimatesAuxilliary2}) we need to have
$- \frac{1}{4} c_0 a \leq -\frac{c}{a} e^{a \hat{r}}$
for all $\hat{r} \leq r_1$, which implies that we should choose
\begin{equation*}
c \leq \frac{9 L^2 r_1^2}{c_0} e^{-\frac{6 L r_1^2}{c_0}} \,.
\end{equation*}

\paragraph{The case of $\hat{r} \leq \sqrt{h}$.}

We use the fact that $f(\hat{r}) \leq \hat{r}$ and in fact for small $\hat{r}$ these quantities are close to each other. Coming back to (\ref{smallEstimatesGoal2}), we see that we need to show
\begin{equation}\label{smallEstimates4}
e^{ahLr_1} h L - \frac{1}{2}c_0  \sqrt{h} a e^{-a \sqrt{h}} \leq -ch e^{a\hat{r}} \,.
\end{equation} 
Since $h < h_0$ and by assumption
\begin{equation*}
h_0 \leq \frac{c_0^2 (\ln{2})^2}{144L^2\mathcal{R}^2} \,,
\end{equation*}
using once again the bound $r_1 \leq 3 \mathcal{R}/2$, we see that
\begin{equation*}
\sqrt{h} \leq \frac{c_0\ln{2}}{12L\mathcal{R}} \leq \frac{3c_0\ln{2}}{24Lr_1} = \frac{3}{4a} \ln{2} \leq \frac{1}{a} \ln{2} 
\end{equation*}
and we have
$e^{a \sqrt{h}} \leq 2$.
Hence the left hand side of (\ref{smallEstimates4}) is bounded (recall that we also have $e^{ahLr_1} \leq \frac{3}{2}$) by
$\frac{3}{2} h L - \frac{1}{4}c_0  \sqrt{h} a$,
which we now need to make sure is negative. We can do this e.g. by making sure that
$\frac{3}{2} h L \leq \frac{1}{8}c_0 \sqrt{h} a$,
which is equivalent to
$\sqrt{h} \leq \frac{r_1}{2}$,
but this is ensured by our assumption $h_0 \leq \mathcal{R}^2 / 4$, since $\mathcal{R} \leq r_1$. Hence we see that the left hand side of (\ref{smallEstimates4}) is bounded from above by
$- c_0 a \sqrt{h}/8$.
Observe that we have
$e^{a \hat{r}} \leq e^{a \sqrt{h}} \leq 2$,
which implies
$- \frac{1}{8} c_0 a \sqrt{h} \leq - \frac{1}{16} c_0 a \sqrt{h} e^{a \hat{r}}$.
Hence we need to have
\begin{equation*}
- \frac{1}{16} c_0 a \sqrt{h} e^{a \hat{r}} \leq - c h e^{a \hat{r}} \,,
\end{equation*}
which implies that we should choose 
\begin{equation*}
c \leq \frac{3 L r_1}{8 \sqrt{h_0}} = \frac{c_0 a}{16 \sqrt{h_0}} \,.
\end{equation*}
Then (\ref{smallEstimates4}) is satisfied and the proof is finished.

\section{Weak error analysis}\label{sectionWeak}

\begin{proof}[Proof of Theorem \ref{th weak}]
	The novelty of this proof is in the fact that we are working with an inaccurate drift. This is what we focus on, while keeping some standard estimates left to the reader. We refer to \cite{talay1990second} and \cite{milstein2013stochastic} for details of proofs of weak convergence rates. Moreover, in order to clearly showcase the main ideas, we present the proof only in the one-dimensional setting. An extension to the multidimensional case follows in the same way, although with more cumbersome notation. Let $u(t,x):=\bE[ g(Y_{t}^{0,x}) ]$, where $Y_{t}^{0,x}$ is a solution to the SDE (\ref{eq sde}) at time $t$, started from the initial condition $x \in \mathbb{R}$ at time $0$. Since $b \in C^3_b$ and $g\in C^{\infty}$ with polynomial growth, we can deduce from \cite{krylov98} that $u\in C^{1,3}([0,T],\bR)$ and that it satisfies the Kolmogorov equation 
	\begin{equation}\label{eg pde}
	\begin{cases}
	(Lu)(t,x):=\partial_t u(t,x) - \partial_x u(t,x) b(x) - \frac{1}{2}\partial_x^2u(t,x) =0\,,\quad (t,x)\in[0,\infty)\times \bR,\\
	u(0,x) = g(x)\,. 	
	\end{cases}	
	\end{equation}
	Furthermore, due to our assumption (\ref{eq th weak dissipativity}), Theorem 3.4 in \cite{talay1990second} tells us that for any multi-index $I$ with $|I| \leq 2$ there exist an integer $p_{|I|}$ and positive constants $\gamma_{|I|}$ and $\Gamma_{|I|}$ such that
	\begin{equation} \label{eq ubound}
	|\partial_I u(t,x) | \leq \Gamma_{|I|} (1+|x|^{p_{|I|}})e^{-\gamma_{|I|} t} \,.
	\end{equation}
	Let $T = Mh$ for some $M \geq 1$. We define the continuous extension of \eqref{eq ineuler}, given by
	\[
	\bar X_{t} = X_0 + \int_0^t \bar{b}(\bar X_{\eta(s)},U_{\eta(s)})ds + W_t\,,
	\]
	where $\eta(s)=kh $ for $s\in(kh,(k+1)h]$ and $(U_{kh})_{k=0}^{\infty}$ are defined in an analogous way as $(U_k)_{k=0}^{\infty}$ in (\ref{eq ineuler}).  It is important to note that $\bar X_{kh}$ is independent of $U_{kh}$.
	From the definition of $u$ we have $u(0,\bar X_T)=g(\bar X_T)$. Consequently $\bE[g(\bar X_T)] - \bE[g(Y_T)]= \bE[u(0,\bar X_T)  - u(T,X_0) ]$ due to the Feynman-Kac formula. Then the It\^{o} formula gives 
	\begin{equation*}
	\begin{split}
	\bE&[g(\bar X_T)] - \bE[g(Y_T)]  =
	\sum_{k=1}^M  \bE [u(T - kh,\bar X_{kh})  - u(T- (k-1)h,\bar X_{(k-1)h})]	 \\
	=& \sum_{k=1}^M  \bE \int_{T-(k-1)h}^{T- kh}
	\left(- \partial_t u(T-t,\bar X_t) + \partial_x u(T- t, \bar X_t)\bar{b}(\bar X_{\eta(t)}, U_{\eta(t)}) + \frac{1}{2}\partial_{xx}^2u(T-t,\bar X_t) \right)dt\,.
	\end{split}
	\end{equation*}
	From now on, let us denote $t_k := kh$. Since $u$ satisfies the PDE \eqref{eg pde}, we see that for all $x\in \bR^d$ we have $-\partial_t u(T-t,x) + \frac{1}{2}\partial_{xx}^2u(T-t,x)  = -  \partial_x u(T-t,x) b(x)$. Hence
	\begin{equation*}
	\begin{split}
	& \bE[g(\bar X_T)] - \bE[g(Y_T)]=	 \sum_{k=1}^M  \bE\int_{T-t_{k-1}}^{T- t_k}
	 \partial_x u(T-t, \bar X_t)\left( \bar{b}(\bar X_{\eta(t)}, U_{\eta(t)}) - b(\bar X_t) \right)dt \\
	= &  \sum_{k=1}^M  \bE \int_{T-t_{k-1}}^{T- t_k}
	\left[\left( \partial_x u(T- t, \bar X_t) - \partial_x u(T- t, \bar X_{\eta(t)})\right) \left( \bar{b}(\bar X_{\eta(t)}, U_{\eta(t)}) - b(\bar X_t) \right) \right]dt \\
	& + \sum_{k=1}^M  \bE \int_{T-t_{k-1}}^{T- t_{k}}
	\left[ \partial_x u(T-t, \bar X_{\eta(t)})\left( \bar{b}(\bar X_{\eta(t)}, U_{\eta(t)}) - b(\bar X_t) \right) \right]dt 
	\\
	= &  \sum_{k=1}^M  \bE \int_{T-t_{k-1}}^{T- t_{k}}
	\left[\left( \partial_x u(T- t, \bar X_t) - \partial_x u(T-t, \bar X_{\eta(t)})\right) \left( \bar{b}(\bar X_{\eta(t)}, U_{\eta(t)}) - b(\bar X_{\eta(t)})\right) \right]dt\,\,\,\,\, (:= E_1) \\
	& + \sum_{k=1}^M  \bE \int_{T-t_{k-1}}^{T- t_{k}}
	\left[\left( \partial_x u(T - t, \bar X_t) - \partial_x u(T - t, \bar X_{\eta(t)})\right) \left( b(\bar X_{\eta(t)}) - b(\bar X_t) \right) \right]dt \,\,\,\,\, (:= E_2)\\
	& + \sum_{k=1}^M  \bE \int_{T-t_{k-1}}^{T- t_{k}}
	\left[ \partial_x u(T - t, \bar X_{\eta(t)}) \left( \bar{b}(\bar X_{\eta(t)}, U_{\eta(t)}) - b(\bar X_t) \right) \right]dt\,\,\,\,\, (:= E_3) \,.
	\end{split}
	\end{equation*}
	We begin by considering the last term, denoted by $E_3$. Observe that 
	\[
	\begin{split}
	\bE& 
	\left[ \partial_x u(T - t, \bar X_{\eta(t)}) \left( \bar{b}(\bar X_{\eta(t)}, U_{\eta(t)}) - b(\bar X_t) \right) |  \bar X_{\eta(t)} \right] \\
	&= \bE 
	\left[ \partial_x u(T-t, \bar X_{\eta(t)}) \bE \left[ \left( b(\bar X_{\eta(t)}) - b(\bar X_t)\right) |  \bar X_{\eta(t)} \right] \right]	\\
	&\leq \,
	\Gamma_{1}e^{-\gamma_{1} (T-t) } \bE 
	\left[  (1+|\bar{X}_{\eta(t)}|^{p_1}) \cdot | \bE \left[\left( b(\bar X_{\eta(t)}) - b(\bar X_t)\right) | \bar X_{\eta(t)}\right] | \, \right] \\
	&\leq \Gamma_{1}e^{-\gamma_{1} (T-t) } (2+ 2C_{IEul}^{(2p_1)})^{1/2} \left(\bE| \bE \left[\left( b(\bar X_{\eta(t)}) - b(\bar X_t) \right) | \bar X_{\eta(t)}\right]|^2\right)^{1/2}\,,
	\end{split}
	\]
	where we used $\mathbb{E}|\bar X_{\eta(t)}|^{2p_1} \leq C_{IEul}^{(2p_1)}$ due to Remark \ref{remarkPthMoments}. Application of the It\^{o} formula and the fact that $b\in C^3_b$ leads to 
	\[
	\begin{split}
	& \bE \left[ \left( b(\bar X_t) - b(\bar X_{\eta(t)}) \right) | \bar X_{\eta(t)}\right] =\bE [ \int_{\eta(t)}^t b'(\bar X_s) \bar{b}(\bar X_{\eta(s)},U_{\eta(s)}) + \frac{1}{2} b''(\bar X_s) ds \,| \,\bar X_{\eta(t)}] \\
	\leq & \, \frac{1}{2} C_{b''}(t-\eta(t))  +	 C_{b'} \int_{\eta(t)}^t \bE \left[| \bar{b}(\bar X_{\eta(s)},U_{\eta(s)})| \, | \bar X_{\eta(t)}\right] ds   \,,
	\end{split}
	\]
	where $C_{b'}$ and $C_{b''}$ are bounds on the derivatives of $b$. By Minkowski's inequality and \eqref{ina conditionalSecondMomentBound}
	\[
	\left( \bE\left[|\bar{b}(\bar X_{\eta(s)},U_{\eta(s)})| \,  | \bar X_{\eta(t)}\right] \right)^2\leq \sigma^2 (|\bar{X}_{\eta(s)}|^2 + 1) h^{\alpha} + 2L^2|\bar{X}_{\eta(s)}|^2 + 2|b(0)|^2 \leq C(|\bar{X}_{\eta(s)}|^2 + 1)\,,
	\]		
	where $C:= \max \left( \sigma^2 + 2L^2 , \sigma^2 + 2|b(0)|^2 \right)$, since $h \leq 1$. Consequently 
	\[
	E_3 \leq \bar C\sum_{k=1}^M \int_{T-t_{k-1}}^{T- t_{k}}  e^{-\gamma_{1}(T- t) }  
	(t-\eta(t))dt\,, 
	\]
	with $\bar C := \Gamma_{1} (2+ 2C_{IEul}^{(2p_1)})^{1/2} \left( \frac{1}{2}C_{b''}^2 + 2 C_{b'}^2 (C_{IEul} + 1) \max(\sigma^2 + 2L^2 , \sigma^2 + 2|b(0)|^2) \right)^{1/2}$. We conclude our estimation of $E_3$ by observing that, since $|t - \eta(t)| \leq h$, we have  
	\[
	\bar{C} \sum_{k=1}^M   \int_{T-t_{k-1}}^{T- t_{k}} e^{-\gamma_{1}(T- t) }  
	(t-\eta(t))dt	 \leq \bar{C} h \sum_{k=1}^M   \int_{T-t_{k-1}}^{T- t_{k}} e^{-\gamma_{1} (T-t) }dt = \bar{C} h \int_{0}^T e^{-\gamma_{1} (T-t) }dt \leq \bar{C} \frac{h}{\gamma_{1}} \,.
	\] 
	That completes the estimates of the third term $E_3$. To estimate the second term $E_2$ we calculate 
	\begin{equation*} 
	\begin{split}
	& \sum_{k=1}^M  \bE \int_{T-t_{k-1}}^{T- t_{k}}
	\left[\left( \partial_x u(T - t, \bar X_t) - \partial_x u(T - t, \bar X_{\eta(t)})\right) \left( b(\bar X_{\eta(t)}) - b(\bar X_t) \right) \right]dt \\
	\leq \, & L \sum_{k=1}^M  \bE \int_{T-t_{k-1}}^{T- t_{k}}
	\left[\int_{0}^{1} | \partial_{xx} u(T-t, \alpha \bar X_t  + (1-\alpha)\bar X_{\eta(t)})|d\alpha \, |\bar X_{\eta(t)} - \bar X_t|^2  \right]dt \\
	\leq \, &  \sum_{k=1}^M \int_{T-t_{k-1}}^{T- t_{k}} \Gamma_{2} e^{-\gamma_2 (T-t) } 2^{p_2 - 1} 
	\bE \left[ (1+ | \bar X_t|^{p_2}  + |\bar X_{\eta(t)}|^{p_2})   |\bar X_{\eta(t)} - \bar X_t|^2  \right]dt\, \\
	\leq \, &  \sqrt{3} 2^{p_2 - 1} \sum_{k=1}^M \int_{T-t_{k-1}}^{T- t_{k}} \Gamma_{2} e^{-\gamma_2 (T-t) }
	\left( \bE[ (1+ | \bar X_t|^{2p_2}  + |\bar X_{\eta(t)}|^{2p_2}) ]\right)^{1/2}  \left( \bE[ |\bar X_{\eta(t)} - \bar X_t|^4  ]\right)^{1/2}dt \,.
	\end{split}	
	\end{equation*}
	The first expectation in the integrand can be bounded using Remark \ref{remarkPthMoments}, while for the second we write
	\begin{equation*}
	\begin{split}
	\bE | \bar X_t - \bar X_{\eta(t)}|^4 = &  \bE\left|  \int_{\eta{(t)}}^t \bar{b}(\bar{X}_{\eta(s)},U_{\eta(s)}) ds + W(t) - W(\eta(t)) \right|^4 \\
	\leq & 8 \left( (t- \eta{(t)})^3  \int_{\eta{(t)}}^t \bE|\bar{b}(\bar{X}_{\eta(s)},U_{\eta(s)})|^4 ds + 3(t- \eta{(t)})^2  \right) \,,
	\end{split}
	\end{equation*}
	where we use $(a+b)^4 \leq 8a^4 + 8b^4$, $\mathbb{E}|W(t) - W(\eta(t))|^4 = 3|t - \eta(t)|^2$ and $\left( \int_a^b |g| \right)^p \leq (b-a)^{p-1} \int_a^b |g|^p$ due to the H\"{o}lder inequality.

	Since we assume that for all $x \in \mathbb{R}$ we have $\mathbb{E}|\bar{b}(x,U)|^4 \leq C_{\bar{b}}^4(1+|x|^4)$,  
	we can now bound $|t - \eta(t)| \leq h$ and complete the estimate for $E_2$ in a similar way as for $E_3$.
	
	It now remains to deal with $E_1$. We have 
	\begin{equation*}
	\begin{split}
	& \sum_{k=1}^M  \bE \int_{T- t_{k-1}}^{T- t_k}
	\left[\left( \partial_x u(T-t, \bar X_t) - \partial_x u(T-t, \bar X_{\eta(t)}) \right) \left( \bar{b}(\bar X_{\eta(t)}, U_{\eta(t)}) - b(\bar X_{\eta(t)})\right) \right]dt \\
	\leq &  \sum_{k=1}^M   \int_{T- t_{k-1}}^{T- t_{k}} \left( \mathbb{E} |\partial_x u(T-t,\bar{X}_t) - \partial_x u(T-t,\bar{X}_{\eta(t)})|^2 \right)^{1/2} \left( \mathbb{E}| \bar{b}(\bar{X}_{\eta(t)}, U_{\eta(t)}) - b(\bar{X}_{\eta(t)}) |^2 \right)^{1/2} dt \\
	\leq &  \sigma (1 + C_{IEul})^{1/2} \sum_{k=1}^M   \int_{T- t_{k-1}}^{T - t_{k}} \left( \mathbb{E} |\partial_x u(T-t,\bar{X}_t) - \partial_x u(T-t,\bar{X}_{\eta(t)})|^2 \right)^{1/2} dt \,,
	\end{split}
	\end{equation*}
	where we used \eqref{ina varianceBound} and $h \leq 1$. For the remaining expectation,	
	since $u\in C^{1,3}$,  we have 
	\[
	\begin{split}
	& \partial_x u(T-t, \bar X_t) - \partial_x u(T-t, \bar X_{\eta(t)}) \\
	= & \int_{\eta(t)}^t \left( -\partial^2_{tx} u(T-s,\bar X_s) + \partial^2_{xx} u(T-s,\bar X_s) \bar{b}(\bar X_{\eta(s)},U_{\eta(s)}) + \frac{1}{2}\partial_{xxx}^3 u(T-s,\bar X_s) \right) ds \,.
	\end{split}
	\]
	Since $u\in C^{1,3}([0,T],\bR)$, we can infer from (\ref{eg pde}) that it satisfies 
	\begin{equation} \label{eq new pde}
	\begin{cases}
	\partial_{tx} u(t,x) - \partial^2_{xx} u(t,x) b(x) - \frac{1}{2}\partial_{xxx}^3u(t,x) = \partial_x u(t,x)b'(x)\,, \quad (t,x)\in[0,\infty)\times \bR,\\
	\partial_x u(0,x) = g'(x)\,. 	
	\end{cases}	
	\end{equation}
	Hence we have
	\[
	\begin{split}
	& \partial_x u(T-t, \bar X_t) - \partial_x u(T-t, \bar X_{\eta(t)}) \\
	= &  \int_{\eta(t)}^t \left( \partial^2_{xx} u(T-s,\bar X_s) \left( \bar{b}(\bar X_{\eta(s)},U_{\eta(s)}) - b(\bar X_{s}) \right) - \partial_x u(T-s,\bar X_s) b'(\bar X_s) \right) ds \\
	\leq & \int_{\eta(t)}^t \left( \Gamma_2 e^{-\gamma_2 (T-s)}(1 + |\bar{X}_s|^{p_2}) | \bar{b}(\bar X_{\eta(s)},U_{\eta(s)}) - b(\bar X_{s}) | + C_{b'} \Gamma_1 e^{-\gamma_1 (T-s)} (1 + |\bar{X}_s|^{p_1}) \right) ds \,.
	\end{split}
	\]
	Now we can use the H\"{o}lder inequality as above with $p = 2$, take the expectation, use the Cauchy-Schwarz inequality for $\mathbb{E}$ to separate the terms involving $|\bar{X}_s|$, $|\bar{b}(\bar{X}_{\eta(s)}, U_{\eta(s)})|$ and $|b(\bar{X}_s)|$, use the linear growth conditions (\ref{boundOnbOfx}) and (\ref{ina conditionalSecondMomentBound}) for $b$ and $\bar{b}$, respectively, and then use the uniform bounds for $\mathbb{E}|\bar{X}_s|^p$ for all $p \geq 1$ as in the previous step. All these efforts combined give us a constant $\hat{C} > 0$ such that
	\begin{equation*}
	\partial_x u(T-t, \bar X_t) - \partial_x u(T-t, \bar X_{\eta(t)}) \leq \hat{C} (t - \eta(t)) e^{-(\gamma_1 \wedge \gamma_2)(T-t)} \,.
	\end{equation*}
	Proceeding as before, we finish the estimate of $E_1$ and hence the entire proof.
\end{proof}

\section{Further possible extensions}

We would like to highlight the robustness of the general approach presented in this paper by briefly discussing several possible extensions of our results to more general settings. For the lack of space, we do not cover these additional cases in detail. However, they are all based on the idea that whenever for a Markov chain $(X_k)_{k=0}^{\infty}$ we have a one-step contraction of the form (\ref{contractivity}) in some (pseudo-)distance function $f$ (see \cite{EberleMajka2018, QinHobert2019} and the references therein for further examples of such contractions), we can then use the properties of $f$ to obtain a perturbation inequality such as (\ref{perturbation}) and employ it to study error bounds for sampling algorithms based on $(X_k)_{k=0}^{\infty}$.

\subsection{Non-constant discretisation parameter}

Our analysis of ULA in Section \ref{sectionULA} is based on the one-step analysis of Euler schemes in Section \ref{sectionCoupling}. However, we do not need to use the same discretisation parameter in each step and the proofs of Theorems \ref{theoremULA} and \ref{mainTheoremULA} still work in the same way if the discretisation parameters vary between steps. More precisely, let us consider a decreasing sequence $(h_k)_{k=0}^{\infty}$ such that $h_0 \leq \min \left( \frac{K}{4L^2}, \frac{4}{K}, \frac{1}{2L}, \frac{2 c_0 \ln{\frac{3}{2}}}{27L^2 \mathcal{R}^2}, \frac{\mathcal{R}^2}{4}, \frac{c_0^2 (\ln{2})^2}{144L^2 \mathcal{R}^2} \right)$, so that all $h_k$ for $k \geq 0$ are within the range required in Theorem \ref{theoremULA}. We consider a Markov chain given by
\begin{equation*}
X_{k+1} = X_k + h_k b(X_k) + \sqrt{h_k} Z_{k+1} \,.
\end{equation*}
Then in the proof of Theorem \ref{theoremULA} for each $k \geq 0$ we have
\begin{equation*}
\mathbb{E}f(|G_{k+1} - Y_{(k+1)h_k}|) \leq (1-ch_k) \mathbb{E} f(|G_k - Y_{kh_{k-1}}|) + C_{ult} h_k^{3/2}
\end{equation*}
and hence, instead of (\ref{ULAFinalInequality}), we obtain
\begin{equation*}
\mathbb{E}f(|G_{k+1} - Y_{(k+1)h_k}|) \leq \left( \prod_{l=0}^{k}(1-ch_l)  \right) \mathbb{E}f(|G_0 - Y_0|) + C_{ult} \sum_{j=0}^{k} \left( \prod_{l=1}^{j}(1-ch_{k-(l-1)}) h_{k-j}^{3/2} \right) \,,
\end{equation*}
with the convention $\prod_{l=1}^0 (\ldots) = 1$. This leads to
\begin{equation*}
\begin{split}
W_2(\cL(X_k),\pi) &\leq \left( A\left( \prod_{l=0}^{k-1}(1-ch_l)  \right) \mathbb{E}f(|X_0 - V_0|) \right)^{1/2} \\
&+ \left( AC_{ult} \sum_{j=0}^{k-1} \left( \prod_{l=1}^{j}(1-ch_{k-(l-1)}) h_{k-j}^{3/2} \right) \right)^{1/2} \,,
\end{split}
\end{equation*}
where $V_0 \sim \pi$, and bounds for $W_1$ can be obtained in a similar way.

\subsection{Bounds in the total variation distance}\label{subsectionTotalVariation}

In order to obtain error bounds for ULA in the total variation distance, we can consider a (pseudo-)distance function with a discontinuity at zero. Namely, if $f : [0,\infty) \to [0,\infty)$ is of the form
\begin{equation}\label{fWithDiscontinuity}
f(r) := a \mathbf{1}_{(0,\infty)}(r) + f_1(r) \,,
\end{equation} 
where $a > 0$ and $f_1 : [0,\infty) \to [0,\infty)$ is a continuous, non-decreasing function, then for all probability measures $\mu$ and $\nu$ on $\mathbb{R}^d$ we have $\| \mu - \nu \|_{TV} \leq 2 a^{-1} W_f(\mu, \nu)$ (recall that $\| \mu - \nu \|_{TV} = 2 \inf_{\pi \in \Pi(\mu, \nu)} \int \mathbf{1}_{(0, \infty)}(|x-y|) \pi(dx \, dy)$ and note that $W_{\mathbf{1}_{(0,\infty)}}(\mu, \nu) = \frac{1}{2}\| \mu - \nu \|_{TV}$). Hence, if for a Markov chain $(X_k)_{k=0}^{\infty}$ we obtain a one-step contraction such as (\ref{contractivity}) in a (pseudo-)distance based on the function $f$ given by (\ref{fWithDiscontinuity}), then we can use the same reasoning as in Section \ref{sectionULA} to control $\| \cL(X_k) - \pi \|_{TV}$. Obtaining such contractions is indeed possible under our Assumptions \ref{as diss}, see e.g.\ Theorem 2.10 in \cite{EberleMajka2018}. Note that the distance function and all the constants there are explicit (although rather complicated). We leave the details to the interested reader. 

\subsection{Bounds without requiring contractivity of the drift even at infinity}

Another possible extension of our results would be to remove the assumption (\ref{driftDissipativityAtInf}) of contractivity at infinity of the drift and replace it with the weaker condition (\ref{driftDissipativeGrowth}), i.e., to assume that there exist constants $M_1$, $M_2 > 0$ such that for all $x \in \mathbb{R}^d$ we have
\begin{equation}\label{driftLyapunovCondition}
\langle b(x), x \rangle \leq M_2 - M_1 |x|^2 \,.
\end{equation}
Note that in Lemma \ref{lemmaContractivityAtInfinityImpliesLyapunov} we showed that (\ref{driftDissipativityAtInf}) implies (\ref{driftLyapunovCondition}) so the latter condition is indeed more general. Note that an extension in this direction in the context of Wasserstein contractions for diffusion processes has been recently obtained in \cite{EberleGuillinZimmer2019}, which generalized the results from \cite{Eberle2016}. An analogous result directly at the level of Euler schemes has been obtained in \cite{EberleMajka2018}, see Theorem 6.1 and Example 6.2 therein. Namely, the authors of \cite{EberleMajka2018} showed that under Assumptions \ref{as diss} with (\ref{driftDissipativityAtInf}) replaced by (\ref{driftLyapunovCondition}), one can construct a pseudo-distance $\rho_a$ of the form
\begin{equation}\label{rhoDistance}
\rho_a(x,y) := \left( a + C(V(x) + V(y)) \right)\mathbf{1}_{(0,\infty)}(|x-y|) + f_1(|x-y|) 
\end{equation} 
and a coupling $(X',Y')$ of the Euler scheme (\ref{oneStepEuler}) with a transition kernel $p$, such that $(X',Y')$ is contractive in $W_{\rho_a}$, i.e., using the notation from Corollary \ref{mainCorollary},
\begin{equation}\label{contractionWithoutContractivity}
W_{\rho_a}(\mu p , \nu p) \leq (1-ch)W_{\rho_a}(\mu,\nu)
\end{equation}
with some constant $c > 0$, for sufficiently small $h > 0$ and for all probability measures $\mu$ and $\nu$ on $\mathbb{R}^d$. In (\ref{rhoDistance}), $a$ and $C$ are positive constants, $V$ is a Lyapunov function and $f_1$ is a continuous non-decreasing function. Since $\rho_a$, similarly as $f$ in Subsection \ref{subsectionTotalVariation}, dominates $\mathbf{1}_{(0,\infty)}$ (up to a multiplicative constant) the contraction in (\ref{contractionWithoutContractivity}) allows us to control $\| \cL(X_k) - \pi \|_{TV}$, where $(X_k)_{k=0}^{\infty}$ is a scheme with the transition kernel $p$. Moreover, $\rho_a$ also dominates the truncated Wasserstein distance $W_1^{\ast}$ given by 
\begin{equation*}
W_1^{\ast}(\mu, \nu) := \inf_{\pi \in \Pi(\mu,\nu)}  \int_{\bR^d\times \bR^d} \left( |x-y| \wedge 1 \right) \, \pi(dx \, dy) \,.
\end{equation*}
Hence (\ref{contractionWithoutContractivity}) allows us to control also $W_1^{\ast}(\cL(X_k), \pi)$. By tracing the proof of our Theorem \ref{theoremSGcomparison}, we can then easily extend all these results to chains with inaccurate drifts. Again, for the lack of space, we leave the details to the reader.

\section{Appendix: Coupling for Euler schemes}

In this section we provide a proof that the random vector $(X',Y')$ presented in (\ref{ourCoupling}) is indeed a coupling of two copies of $X'$. Let us start with the case of $H = \infty$, i.e., without the truncation based on the distance of the processes before the jump.

\begin{theorem}\label{couplingTheorem}
	The joint distribution of the $2d$-dimensional random vector $(X',Y')$ defined by (\ref{ourCoupling}) is a coupling of $N(\hat{x},hI)$ and $N(\hat{y},hI)$.
	\begin{proof}
		We need to show that $Y' \sim N(\hat{y},hI)$. To this end, we will show that for any continuous bounded function $g$ we have
		\begin{equation*}
		\mathbb{E}g(Y') = \mathbb{E}g(\hat{y} + \sqrt{h}Z) \,,
		\end{equation*}
		where $Z \sim N(0,I)$. Straight from the definition of $Y'$ we get
		\begin{equation}\label{truncatedMirrorComputation}
		\mathbb{E}g(Y') = I_1 + I_2 + I_3 \,,
		\end{equation}
		where
		\begin{equation*}
		\begin{split}
		I_1 &= \int_{\mathbb{R}^d} g(\hat{x} + \sqrt{h}z) \frac{\varphi_{\hat{y}, hI}^{m}(\hat{x} + \sqrt{h}z) \wedge \varphi_{\hat{x}, hI}^{m}(\hat{x} + \sqrt{h}z) }{\varphi_{\hat{x}, hI}^{m}(\hat{x} + \sqrt{h}z)} \mathbf{1}_{\{ |\sqrt{h}z| < m \}} \varphi_{0,I}(z) dz \,, \\
		I_2 &= \int_{\mathbb{R}^d} g(\hat{y} + R_{\hat{x},\hat{y}}\sqrt{h}z) \left( 1 - \frac{\varphi_{\hat{y}, hI}^{m}(\hat{x} + \sqrt{h}z) \wedge \varphi_{\hat{x}, hI}^{m}(\hat{x} + \sqrt{h}z) }{\varphi_{\hat{x}, hI}^{m}(\hat{x} + \sqrt{h}z)} \right) \mathbf{1}_{\{ |\sqrt{h}z| < m \}}  \varphi_{0,I}(z) dz \\
		\text{and } \, \, \, I_3 &= \int_{\mathbb{R}^d} g(\hat{y} + \sqrt{h}z)  \mathbf{1}_{\{ |\sqrt{h}z| \geq m \}}  \varphi_{0,I}(z) dz \,.
		\end{split}
		\end{equation*}
		By substituting $u = \hat{x} + \sqrt{h}z$ and using the fact that $\frac{1}{\sqrt{h}} \varphi_{0,I}(\frac{u-\hat{x}}{\sqrt{h}}) = \varphi_{\hat{x},hI}(u)$, we see that
		\begin{equation*}
		I_1 = \int_{\mathbb{R}^d} g(u) \left( \varphi^{m}_{\hat{y},hI}(u) \wedge \varphi^{m}_{\hat{x},hI}(u) \right) du \,.
		\end{equation*}
		In a similar way, we can substitute $u = \hat{y} + \sqrt{h}z$ to see that
		\begin{equation*}
		I_3 = \int_{\mathbb{R}^d} g(u) \mathbf{1}_{\{ |u-\hat{y}| \geq m \}} \varphi_{\hat{y},hI}(u) du \,.
		\end{equation*}
		In order to deal with $I_2$, we substitute $u = \hat{y} + R_{\hat{x},\hat{y}} \sqrt{h}z$ and notice that then $\sqrt{h} z = R_{\hat{x},\hat{y}}(u-\hat{y})$ since $R_{\hat{x},\hat{y}}$ is an involution. Moreover, since $R_{\hat{x},\hat{y}}$ is an isometry and $|\det R_{\hat{x}, \hat{y}}| = 1$, we obtain 
		\begin{equation*}
		I_2 = \int_{\mathbb{R}^d} g(u) \left( 1 - \frac{\varphi^m_{\hat{y},hI}(\hat{x} + R_{\hat{x},\hat{y}}(u-\hat{y})) \wedge \varphi^m_{\hat{x},hI}(\hat{x} + R_{\hat{x},\hat{y}}(u-\hat{y}))}{\varphi^m_{\hat{x},hI}(\hat{x} + R_{\hat{x},\hat{y}}(u-\hat{y}))} \right) \mathbf{1}_{\{ |u-\hat{y}| < m \}} \varphi_{\hat{y},hI}(u) du \,.
		\end{equation*}
		Here we also used $\frac{1}{\sqrt{h}} \varphi_{0,I}\left( \frac{R_{\hat{x},\hat{y}}(u-\hat{y})}{\sqrt{h}}\right) = \varphi_{\hat{y},hI}(u)$, which again follows from $R_{\hat{x},\hat{y}}$ being an isometry. If we now note that $R_{\hat{x},\hat{y}}(\hat{y} - \hat{x}) = \hat{x} - \hat{y}$ and hence
		\begin{equation*}
		\varphi^m_{\hat{y},hI} (\hat{x} + R_{\hat{x},\hat{y}}(u-\hat{y})) = \varphi^m_{\hat{y}-\hat{x},hI} ( R_{\hat{x},\hat{y}}(u-\hat{y})) = \varphi^m_{\hat{x},hI}(u) \,,
		\end{equation*}
		we see that
		\begin{equation*}
		I_2 = \int_{\mathbb{R}^d} g(u) \varphi^m_{\hat{y},hI}(u) du - \int_{\mathbb{R}^d} g(u) \left( \varphi^m_{\hat{x},hI}(u) \wedge \varphi^m_{\hat{y},hI}(u) \right) du \,,
		\end{equation*}
		which implies that $I_1 + I_2 + I_3 = \int_{\mathbb{R}^d} g(u) \varphi_{\hat{y},hI}(u) du$.
	\end{proof}
\end{theorem}

Now we can consider the general case of $(X',Y')$ given by (\ref{ourCoupling}) with $H \in (0,\infty)$. We can easily check that this is indeed a coupling of $N(\hat{x},hI)$ and $N(\hat{y},hI)$ by applying Theorem \ref{couplingTheorem} and observing that for any continuous bounded function $g$ we have
\begin{equation*}
\mathbb{E}g(Y') = \mathbf{1}_{\{ \hat{r} > H \}} \mathbb{E}g(\hat{y} + \sqrt{h} Z) + \mathbf{1}_{\{ \hat{r} \leq H \}} (I_1 + I_2 + I_3) \,,
\end{equation*}
where $I_1$, $I_2$ and $I_3$ denote the expressions that appear on the right hand side of (\ref{truncatedMirrorComputation}). Since we know from the proof of Theorem \ref{couplingTheorem} that
$I_1 + I_2 + I_3 = \int_{\mathbb{R}^d} g(u) \varphi_{\hat{y},hI}(u) du$, we get our assertion.

\section{Appendix: Proofs of lemmas from Section \ref{sectionMain}}

\subsection{Proof of Lemma \ref{lemmaContractivityAtInfinityImpliesLyapunov}} 
 
 	\begin{proof}
 	If $|x| > \max\left(\mathcal{R}, \frac{2|b(0)|}{K}\right)$ then
 	\begin{equation*}
 	\langle b(x), x \rangle = \langle b(x) - b(0) , x \rangle + \langle b(0), x \rangle \leq -K|x|^2 + |b(0)| |x| \leq - K|x|^2 + \frac{K}{2|b(0)|}|b(0)||x|^2 = -\frac{K}{2}|x|^2 \,.
 	\end{equation*}
 	On the other hand, if $|x| \leq \max\left(\mathcal{R}, \frac{2|b(0)|}{K}\right)$, then
 	\begin{equation*}
 	\langle b(x), x \rangle \leq |b(x) - b(0)||x| + |b(0)||x| \leq L\left(\max\left(\mathcal{R}, \frac{2|b(0)|}{K}\right)\right)^2 + |b(0)|\max\left(\mathcal{R}, \frac{2|b(0)|}{K}\right) \,.
 	\end{equation*}
 \end{proof}

\subsection{Proof of Lemma \ref{lemmaSDEMomentUniformBound}}

	\begin{proof}
	By the It\^{o} formula combined with (\ref{driftDissipativeGrowth}), for $\tau_n := \inf \{ t > 0 : |Y_t| > n \}$ we have
	\begin{equation*}
	\begin{split}
	\mathbb{E}|Y_{t \wedge \tau_n}|^2 &= \mathbb{E}|Y_0|^2 + 2 \mathbb{E} \int_0^{t \wedge \tau_n} \langle b(Y_s) , Y_s \rangle ds + \mathbb{E}(t \wedge \tau_n)d \\
	&\leq \mathbb{E}|Y_0|^2 + t(2M_2+d) - 2M_1 \int_0^{t} \mathbb{E}|Y_{s \wedge \tau_n}|^2 ds \,.
	\end{split}
	\end{equation*}
	By the Gronwall inequality and the Fatou lemma, we get
	$\mathbb{E}|Y_t|^2 \leq \left( \mathbb{E}|Y_0|^2 + t(2M_2+d) \right)e^{-2M_1 t}$.
	Note that the function $t \mapsto t e^{-2M_1 t}$ is bounded by $\frac{1}{2M_1}$ for all $t \geq 0$, which implies the desired bound.
\end{proof}

\subsection{Proof of Lemma \ref{lemmaEulerMomentUniformEstimate}}

	\begin{proof}
	Straight from the definition (\ref{ULAprototype}) of $(X_k)_{k=0}^{\infty}$ we have
	\begin{equation*}
	\mathbb{E}|X_{k+1}|^2 = \mathbb{E}|X_k|^2 + \mathbb{E}|hb(X_k) + \sqrt{h}Z_{k+1}|^2 + 2 \mathbb{E} \langle X_k , hb(X_k) + \sqrt{h} Z_{k+1} \rangle \,.
	\end{equation*}
	Now note that (\ref{driftLipschitz}) implies that for all $x \in \mathbb{R}^d$ we have 
	\begin{equation}\label{boundOnbOfx}
	|b(x)|^2 \leq 2L^2 |x|^2 + 2 |b(0)|^2 \,.
	\end{equation}
	This, together with the fact that $\mathbb{E}|Z_{k+1}|^2 = d$ and (\ref{driftDissipativeGrowth}), implies
	\begin{equation*}
	\mathbb{E}|X_{k+1}|^2 \leq \mathbb{E}|X_k|^2 \left( 1 + 4h^2 L^2 - 2 M_1 h \right) + \left( 4h^2 |b(0)|^2 + 2hd + 2hM_2 \right) \,.
	\end{equation*}
	Hence we have
	\begin{equation*}
	\mathbb{E}|X_{k+1}|^2 \leq \mathbb{E}|X_0|^2 \left( 1 + 4h^2 L^2 - 2 M_1 h \right)^{k+1} \\
	+ 2h\left( 2h |b(0)|^2 + d + M_2 \right) \sum_{j=0}^{k} \left( 1 + 4h^2 L^2 - 2 M_1 h \right)^{j} \,.
	\end{equation*}
	Note that if $|1 + 4h^2 L^2 - 2 M_1 h| < 1$, which is equivalent to $h < \frac{M_1}{2L^2} = \frac{K}{4L^2}$, then we can bound the finite sum on the right hand side by the infinite sum, and the expression next to $\mathbb{E}|X_0|^2$ by $1$, for any $k \geq 0$. This completes the proof.
\end{proof}

\subsection{Proof of Lemma \ref{lemmaDifferenceEulerSDE}}

	\begin{proof}
	We have
	\begin{equation*}
	\begin{split}
	\mathbb{E} &\left| \int_{kh}^{(k+1)h} b(Y_s) ds  - hb(Y_{kh}) \right|^2 = \mathbb{E} \left| \int_{kh}^{(k+1)h} \left( b(Y_s)   - b(Y_{kh}) \right) ds \right|^2 \leq h \mathbb{E} \int_{kh}^{(k+1)h} \left| b(Y_s)   - b(Y_{kh}) \right|^2 ds \\
	&\leq L^2 h \mathbb{E} \int_{kh}^{(k+1)h} \left| Y_s   - Y_{kh} \right|^2 ds \leq 2L^2 h \mathbb{E} \int_{kh}^{(k+1)h} \left( \left| \int_{kh}^{s} b(Y_r) dr \right|^2 + |W_{s} - W_{kh}|^2 \right) ds \\
	&\leq 2L^2 h \mathbb{E} \int_{kh}^{(k+1)h} \left( (s-kh) \int_{kh}^{s} |b(Y_r)|^2 dr + |W_{s} - W_{kh}|^2 \right) ds \\
	&\leq 2L^2 h \left( \left( \sup_{r \leq (k+1)h} \mathbb{E}|b(Y_r)|^2 \right) \int_{kh}^{(k+1)h} (s-kh)^2 ds + d \int_{kh}^{(k+1)h} (s-kh) ds \right) \\
	&= 2L^2 h \left( \left( \sup_{r \leq (k+1)h} \mathbb{E}|b(Y_r)|^2 \right) \frac{h^3}{3} + d \frac{h^2}{2} \right) = L^2 h^3 \left( \left( \sup_{r \leq (k+1)h} \mathbb{E}|b(Y_r)|^2 \right) \frac{2h}{3} + d \right) \,.
	\end{split}
	\end{equation*}
	Now we combine (\ref{boundOnbOfx}) with Lemma \ref{lemmaSDEMomentUniformBound} and we obtain the desired bound.
\end{proof}

\subsection{Proof of Lemma \ref{lemmaEulerMomentUniformEstimateInacurate}}

	\begin{proof}
	We argue similarly as in the proof of Lemma \ref{lemmaEulerMomentUniformEstimate}. From the definition (\ref{eq ineuler}) 
	of $(\bar{X}_k)_{k=0}^{\infty}$ we have
	\begin{equation*}
	\mathbb{E}|\bar{X}_{k+1}|^2 = \mathbb{E}|\bar{X}_k|^2 + \mathbb{E}|h\bar{b}(\bar{X}_k,U_k) + \sqrt{h}Z_{k+1}|^2 + 2 \mathbb{E} \langle \bar{X}_k , h\bar{b}(\bar{X}_k,U_k) + \sqrt{h} Z_{k+1} \rangle \,.
	\end{equation*}
	Note that
	$\mathbb{E} \bE [\langle \bar{X}_k , h\bar{b}(\bar{X}_k,U_k) + \sqrt{h} Z_{k+1} \rangle | \bar{X}_k] = \mathbb{E} \langle \bar{X}_k , hb(\bar{X}_k) \rangle$.
	Furthermore, by (\ref{ina conditionalSecondMomentBound}) we have 
	\begin{equation*}
	\mathbb{E} \left[ |\bar{b}(\bar{X}_k, U_k) |^2 | \bar{X}_k \right] \leq \sigma^2 (|\bar{X}_k|^2 + 1)h^{\alpha} + 2L^2|\bar{X}_k|^2 + 2|b(0)|^2 \,.
	\end{equation*}
	Hence, using $\mathbb{E}|Z_{k+1}|^2 = d$ and (\ref{driftDissipativeGrowth}), we get
	\begin{equation*}
	\mathbb{E}|\bar{X}_{k+1}|^2 \leq \mathbb{E}|\bar{X}_k|^2 \left( 1 + 2 \sigma^2 h^{2+\alpha} + 4 h^2 L^2 - 2 M_1 h \right) + \left( 2\sigma^2 h^{2+\alpha} + 4 h^2 |b(0)|^2 + 2hd + 2hM_2 \right) \,.
	\end{equation*}
	Thus we have 
	\begin{equation*}
	\begin{split}
	\mathbb{E}|\bar{X}_{k+1}|^2 \leq & \mathbb{E}|\bar{X}_0|^2 \left(  1 + 2 \sigma^2 h^2 + 4 h^2 L^2 - 2 M_1 h  \right)^{k+1}  \\
	&+ \left( 2\sigma^2 h^2 + 4 h^2 |b(0)|^2 + 2hd + 2hM_2 \right) \sum_{j=0}^{k} \left(  1 + 2 \sigma^2 h^2 + 4 h^2 L^2 - 2 M_1 h  \right)^{j} \,,
	\end{split}
	\end{equation*}
	where we used $h^{2+\alpha} \leq h^2$ for $h \leq 1$. Note that our upper bound on $h$ guarantees that $| 1 + 2 \sigma^2 h^2 + 4 h^2 L^2 - 2 M_1 h| < 1$ and hence we can bound the finite sum on the right hand side by the infinite sum. Moreover, we bound the expression next to $\mathbb{E}|X_0|^2$ by $1$, which finishes the proof.
\end{proof}

\section{Appendix: Proofs of lemmas from Section \ref{sectionProofContractivity}}\label{appendixContractivity}

The proofs in this section are based on calculations from Section 6 in \cite{EberleMajka2018}.

\subsection{Proof of Lemma \ref{lemmaCouplingFirstMoment}}

	\begin{proof}
	It is sufficient to consider the case of $\hat{r} \leq H$ since for $\hat{r} > H$ we apply the synchronous coupling. Observe that $R' = |X' - Y'|$ can take values $0$ (when the coupled processes jump to the same point), $\hat{r}$ (when they move synchronously) or
	\begin{equation*}
	\left| \hat{x} - \hat{y} + 2 \frac{\hat{x} - \hat{y} (\hat{x} - \hat{y})^T}{| \hat{x} - \hat{y}|^2} \sqrt{h} Z \right| \,,
	\end{equation*}
	where $Z \sim N(0,I)$, when they are reflected. Hence we see that in the reflection case $X' - Y'$ is a sum of two parallel vectors and we have
	$R' = \hat{r} + 2 \sqrt{h} \langle \frac{\hat{x} - \hat{y}}{|\hat{x} - \hat{y}|} , Z \rangle$.
	However, $\langle \frac{\hat{x} - \hat{y}}{|\hat{x} - \hat{y}|} , Z \rangle$ can be interpreted as a one-dimensional projection of a Gaussian random vector and hence without loss of generality we can assume that
	$R' - \hat{r} = 2 \sqrt{h} W$,
	where $W \sim N(0,1)$ is a one-dimensional Gaussian random variable.
	
	From the definition of our coupling, we see that if $\hat{r} \leq H$, then we have
	\begin{equation*}
	\begin{split}
	\mathbb{E}\left[ R' \, | \, U \right] - \hat{r} &= 2 \int_{-\infty}^{0} |t| \left( 1 - \frac{\varphi_{\hat{r},h}^m(t)}{\varphi_{0,h}^m(t)} \right) \mathbf{1}_{\{ |t| \leq m \}} \varphi_{0,h}(t) dt - 2 \int_0^{\hat{r}/2} |t| \left( 1 - \frac{\varphi_{\hat{r},h}^m(t)}{\varphi_{0,h}^m(t)} \right) \mathbf{1}_{\{ |t| \leq m\}} \varphi_{0,h}(t) dt \\
	&+ \int_{-\infty}^{\hat{r}/2} (-\hat{r}) \frac{\varphi_{\hat{r},h}^m(t)}{\varphi_{0,h}^m(t)} \mathbf{1}_{\{ |t| \leq m \}}  \varphi_{0,h}(t) dt + \int_{\hat{r}/2}^{\infty} (-\hat{r}) \mathbf{1}_{\{ |t| \leq m \}}  \varphi_{0,h}(t) dt \,,
	\end{split}
	\end{equation*}
	where the first two terms correspond to the reflection and the next two terms correspond to the situation in which the coupled processes jump to the same point. By making the substitution $u = t - \hat{r}$ in the second and the fourth integrals we get
	\begin{equation*}
	\begin{split}
	\mathbb{E}\left[ R' \, | \, U \right] - \hat{r} &= \int_{-\infty}^{0} (-2t) \varphi_{0,h}^m(t) dt + \int_{-\infty}^{-\hat{r}} (2u + 2 \hat{r}) \varphi_{0,h}^m(u) du +\int_0^{\hat{r}/2} (-2t)\varphi_{\hat{r},h}^m(t) dt \\
	&+ \int_{-\hat{r}}^{-\hat{r}/2} (2u + 2\hat{r})\varphi_{\hat{r},h}^m(u) du +2 \int_{\hat{r}/2}^{\infty} (-\hat{r}) \varphi_{0,h}^m(t) dt \\
	&= \int_{-\infty}^{\hat{r}/2} (-2t)\varphi_{\hat{r},h}^m(t) dt + \int_{-\infty}^{-\hat{r}/2}  2\hat{r} \varphi_{\hat{r},h}^m(u) du +\int_{\hat{r}/2}^{\infty} (-2\hat{r}) \varphi_{0,h}^m(t) dt \\
	&+ \int_{-\infty}^{-\hat{r}/2} 2u \varphi_{0,h}^m(u) du = \int_{-\hat{r}/2}^{\hat{r}/2} (-2t) \varphi_{0,h}^m(t) dt = 0 \,, 
	\end{split}
	\end{equation*}
	which finishes the proof.
\end{proof}

\begin{remark}\label{remarkFirstMomentReflection}
	Note that if instead of the coupling specified in (\ref{couplingForInaccurate}), we would just use the coupling in which the processes are always reflected, we would get
	\begin{equation*}
	\mathbb{E}\left[ R' \, | \, U \right] - \hat{r} = 2 \int_{-\infty}^0 |t| \varphi_{0,h}(t) dt - 2 \int_0^{\hat{r}/2} |t| \varphi_{0,h}(t) dt + 2 \int_{\hat{r}}^{\infty} |t| \varphi_{0,h}(t) dt \,,
	\end{equation*}
	which clearly is positive for all $\hat{r}$, and hence the assertion of Lemma \ref{lemmaCouplingFirstMoment} fails.
\end{remark}

\subsection{Proof of Lemma \ref{lemmaLowerBound}}

\begin{proof}
	First observe that just like in Lemma \ref{lemmaCouplingFirstMoment}, without loss of generality we can assume that $R' - \hat{r} = 2 \sqrt{h} W$, where $W \sim N(0,1)$. We have
	\begin{equation*}
	\mathbb{E}\left[(R' - \hat{r})^2 \mathbf{1}_{\{ R' \in I_{\hat{r}} \}} \, | \, U \right] \geq \int_{-\infty}^{\hat{r}/2} (-2t)^2 \left( 1 - \frac{\varphi^m_{\hat{r},h}(t)}{\varphi^m_{0,h}(t)} \right) \varphi_{0,h}(t) \mathbf{1}_{\{ |t| \leq m \}}  \mathbf{1}_{\{ \hat{r} - 2t \in I_{\hat{r}} \}}  \mathbf{1}_{\{ |\hat{r}| \leq H \}} dt \,.
	\end{equation*}
	We can therefore skip the condition $|\hat{r}| \leq H$ from now on in our calculations and just add it again at the very end. Let us deal first with the case of $\hat{r} \geq \sqrt{h}$. Then $I_{\hat{r}} = (\hat{r} - \sqrt{h} , \hat{r})$ and we see that $\hat{r} - 2t \in (\hat{r} - \sqrt{h} , \hat{r})$ if and only if $t \in (0, \sqrt{h} / 2)$. Hence the integral above is equal to
	\begin{equation}\label{lowerBoundIntegral}
	4 \int_0^{\sqrt{h}/2 \wedge m} t^2 \left( 1 - \frac{\varphi^m_{\hat{r},h}(t)}{\varphi^m_{0,h}(t)} \right) \varphi_{0,h}(t) dt \,.
	\end{equation}
	Now observe that
	\begin{equation*}
	\frac{\varphi^m_{\hat{r},h}(t)}{\varphi^m_{0,h}(t)} = e^{-\frac{|t - \hat{r}|^2}{2h}} \mathbf{1}_{\{ |t - \hat{r}| \leq m \}} e^{\frac{|t|^2}{2h}} \leq e^{-\frac{|t - \hat{r}|^2}{2h}} e^{\frac{|t|^2}{2h}} = e^{-\frac{1}{2h}(t^2 - 2t\hat{r} + \hat{r}^2 - t^2)} = e^{-\frac{1}{h}(\frac{\hat{r}^2}{2} - t\hat{r})} \,.
	\end{equation*}
	Combining this with $h \leq 4 m^2$, we see that the integral in (\ref{lowerBoundIntegral}) is bounded from below by
	\begin{equation*}
	4 \int_0^{\sqrt{h}/2} t^2 \left( 1 - e^{-\frac{1}{h}(\frac{\hat{r}^2}{2} - t\hat{r})} \right) \varphi_{0,h}(t) dt \,.
	\end{equation*}
	Now we make the substitution $u = \frac{1}{\sqrt{h}}t$ which gives
	\begin{equation}\label{lowerBoundIntegral2}
	4 \int_0^{1/2} \sqrt{h}h u^2 \left( 1 - e^{-\frac{1}{h}(\frac{\hat{r}^2}{2} - \sqrt{h} u \hat{r})} \right) \frac{1}{\sqrt{2 \pi h}} e^{-\frac{h u^2}{2h}} du = 4h \int_0^{1/2}  u^2 \left( 1 - e^{\frac{\hat{r}}{\sqrt{h}}(u - \frac{\hat{r}}{2 \sqrt{h}})} \right) \varphi_{0,1}(u) du \,.
	\end{equation}
	Now observe that the function $s \mapsto s(u - s/2)$ is decreasing for $s \geq u$. Moreover, we have $\hat{r}/\sqrt{h} \geq 1 \geq u$ and thus we can bound the integral in (\ref{lowerBoundIntegral2}) from below by
	\begin{equation*}
	4h \int_0^{1/2}  u^2 \left( 1 - e^{u -\frac{1}{2}} \right) \varphi_{0,1}(u) du \,,
	\end{equation*}
	which finishes the proof for the case of $\hat{r} \geq \sqrt{h}$. 
	
	Assume now that $\hat{r} \leq \sqrt{h}$. Then $I_{\hat{r}} = (0, \hat{r} + \sqrt{h})$ and $\hat{r} - 2t \in I_{\hat{r}}$ if and only if $t \in (- \frac{\sqrt{h}}{2}, \frac{\hat{r}}{2})$. Hence we have
	\begin{equation*}
	\mathbb{E}\left[(R' - \hat{r})^2 \mathbf{1}_{\{ R' \in I_{\hat{r}} \}} \, | \, U \right] \geq \int_{-\sqrt{h} /2}^{0} (-2t)^2 \left( 1 - \frac{\varphi^m_{\hat{r},h}(t)}{\varphi^m_{0,h}(t)} \right) \varphi_{0,h}(t) \mathbf{1}_{\{ |\hat{r}| \leq H \}} dt \,.
	\end{equation*}
    Again for convenience we now skip the condition $\hat{r} \leq H$. After making the substitution $u = \frac{1}{\sqrt{h}}t$ and using the lower bound from the calculations from the previous case, we see that the integral above is bounded from below by
	\begin{equation*}
	4h \int_{-1/2}^{0} u^2 \left( 1 - e^{\frac{\hat{r}}{\sqrt{h}}(u - \frac{\hat{r}}{2 \sqrt{h}})} \right) \frac{1}{\sqrt{2 \pi}} e^{-\frac{u^2}{2}} du \,.
	\end{equation*}
	Since $u \in (-1/2,0)$, we easily see that
	$\frac{\hat{r}}{\sqrt{h}}(u - \frac{\hat{r}}{2 \sqrt{h}}) \leq - \frac{\hat{r}^2}{2h} \leq 0$
	and since $\hat{r} \leq \sqrt{h}$, we also have $u - \hat{r}/(2 \sqrt{h}) \geq - 1$, which implies
	$\frac{\hat{r}}{\sqrt{h}}(u - \frac{\hat{r}}{2 \sqrt{h}}) \geq - \frac{\hat{r}}{\sqrt{h}} \geq - 1$.
	We consider the function $g(s) := e^s - 1 -(1 - e^{-1})s$.
	We see that $g(-1) = g(0) = 0$ and there exists $s_0 \in (-1,0)$ such that $g'(s) < 0$ for $s \in (-1,s_0)$ and $g'(s) > 0$ for $s \in (s_0,0)$. This implies that
	$e^s - 1 \leq (1 - e^{-1})s$
	for all $s \in [-1,0]$ hence our integral is bounded from below by
	\begin{equation*}
	\begin{split}
	4h &\int_{-1/2}^{0} u^2 \left( - \frac{\hat{r}}{\sqrt{h}}(u - \frac{\hat{r}}{2 \sqrt{h}}) \right) (1 - e^{-1}) \varphi_{0,1}(u) du\\ &= \frac{4h\hat{r}}{\sqrt{h}} (1 - e^{-1}) \int_{-1/2}^{0} \left( - u^3 + \frac{\hat{r}}{2 \sqrt{h}}u^2 \right)  \varphi_{0,1}(u) du \geq 4 \sqrt{h} \hat{r}  (1 - e^{-1}) \int_{0}^{1/2} u^3 \varphi_{0,1}(u) du \,.
	\end{split}
	\end{equation*}
\end{proof}

\begin{remark}
	Note that in the proof we only use the reflection behaviour and we disregard the possibility of jumping to the same point. This is due to the fact that the probability of jumping to the same point decays exponentially fast with $h$ going to zero and hence it would not contribute to our estimates in a significant way. Hence, for Lemma \ref{lemmaLowerBound} to work, it would be sufficient to take the reflection coupling. However, then our calculations for the first moment in Lemma \ref{lemmaCouplingFirstMoment} would fail (cf. Remark \ref{remarkFirstMomentReflection}). On the other hand, for the synchronous coupling Lemma \ref{lemmaCouplingFirstMoment} holds while Lemma \ref{lemmaLowerBound} clearly fails. Hence the (truncated) mirror coupling given by (\ref{ourCoupling}) is the only one for which both these Lemmas work.
\end{remark}

\section*{Acknowledgement}
We would like to thank the associate editor and the anonymous referee for numerous useful suggestions. MBM and AM are supported by the EPSRC grant EP/P003818/1. The majority of this work was completed while MBM and AM were affiliated to King's College London.

\bibliographystyle{alpha}

\bibliography{Particles,biblio} 
\end{document}